\documentclass[mnsc,nonblindrev,copyedit]{informs3a}

\usepackage[utf8]{inputenc}
\usepackage[letterpaper,margin=1in]{geometry}
\let\OriginalManuscriptNormalSize\normalsizeXIyesCE
\renewcommand\normalsizeXIyesCE{\OriginalManuscriptNormalSize\fontsize{11}{19.71}\selectfont\abovedisplayskip=5pt plus 2pt minus 1pt\belowdisplayskip=5pt plus 2pt minus 1pt\abovedisplayshortskip=2pt plus 1pt minus 1pt\belowdisplayshortskip=4pt plus 1pt minus 1pt}
\let\OriginalManuscriptSmallSize\smallXyesCE
\renewcommand\smallXyesCE{\OriginalManuscriptSmallSize\fontsize{10}{15}\selectfont}
\let\OriginalManuscriptFootnoteSize\footnotesizeIXyesCE
\renewcommand\footnotesizeIXyesCE{\OriginalManuscriptFootnoteSize\fontsize{9}{13.5}\selectfont}
\renewcommand\ABSfont{\fontsize{9}{18.75}\selectfont}
\AtBeginDocument{\newgeometry{margin=1in,hoffset=0pt,voffset=0pt}\normalsize\raggedbottom\predisplaypenalty=10000}

\usepackage{amsmath,amssymb,mathtools}
\usepackage{booktabs}
\usepackage{array}
\usepackage{enumerate}
\usepackage[dvipsnames]{xcolor}
\usepackage{graphicx}
\usepackage{placeins}
\usepackage{tikz}
\usetikzlibrary{arrows.meta,shapes.geometric}
\usepackage{pgfplots}
\pgfplotsset{compat=1.13}
\usepackage{natbib}
\usepackage[backref=false,bookmarks,breaklinks=true,colorlinks=true,hypertexnames=false,plainpages=false,citecolor=DarkBlue,urlcolor=DarkBlue,filecolor=DarkBlue,linkcolor=DarkBlue,pdftitle={A 1.932 Guarantee for (r,Q) Policies in Continuous-Review Lost-Sales Inventory Systems},pdfauthor={Ming Hu},pdfkeywords={inventory; lost sales; fixed ordering costs; (r,Q) policies; base-stock policies; approximation guarantees; Poisson processes}]{hyperref}
\urldef{\provemeurl}\url{https://prove2.me/missions/A_%241.932%24_Guarantee_for_%24(r%2CQ)%24_Policies_in_Continuous-Review_Lost-Sales_Inventory_Systems}
\definecolor{DarkBlue}{rgb}{0,0.08,0.45}
\definecolor{lsTeal}{rgb}{0.1567,0.4717,0.4234}
\definecolor{lsOrng}{rgb}{0.8436,0.4582,0.1704}
\definecolor{lsPale}{rgb}{0.9470,0.7987,0.6011}
\definecolor{lsBand}{rgb}{0.9550,0.8924,0.7724}

\bibpunct[, ]{(}{)}{,}{a}{}{,}
\def\bibfont{\fontsize{11}{18}\selectfont}

\setlength{\floatsep}{4pt plus 2pt minus 2pt}
\setlength{\textfloatsep}{4pt plus 2pt minus 2pt}
\setlength{\intextsep}{4pt plus 2pt minus 2pt}
\TheoremsNumberedThrough
\EquationsNumberedThrough
\MANUSCRIPTNO{}
\renewcommand{\Halmos}{\unskip\nobreak\mbox{\quad$\square$}}

\newcommand{\E}{\mathbb E}
\newcommand{\Prb}{\mathbb P}
\newcommand{\OPT}{\mathrm{OPT}}

\DeclareUnicodeCharacter{03BC}{\ensuremath{\mu}}
\DeclareUnicodeCharacter{03C6}{\ensuremath{\phi}}
\DeclareUnicodeCharacter{03A6}{\ensuremath{\Phi}}
\DeclareUnicodeCharacter{03C0}{\ensuremath{\pi}}
\DeclareUnicodeCharacter{03BA}{\ensuremath{\kappa}}
\DeclareUnicodeCharacter{03B7}{\ensuremath{\eta}}
\DeclareUnicodeCharacter{03B5}{\ensuremath{\epsilon}}
\DeclareUnicodeCharacter{03B4}{\ensuremath{\delta}}
\DeclareUnicodeCharacter{03B8}{\ensuremath{\theta}}
\DeclareUnicodeCharacter{03B1}{\ensuremath{\alpha}}
\DeclareUnicodeCharacter{03B2}{\ensuremath{\beta}}
\DeclareUnicodeCharacter{03C1}{\ensuremath{\rho}}
\DeclareUnicodeCharacter{03BB}{\ensuremath{\lambda}}
\DeclareUnicodeCharacter{0393}{\ensuremath{\Gamma}}
\DeclareUnicodeCharacter{03B3}{\ensuremath{\gamma}}
\DeclareUnicodeCharacter{2113}{\ensuremath{\ell}}
\DeclareUnicodeCharacter{2264}{\ensuremath{\le}}
\DeclareUnicodeCharacter{2265}{\ensuremath{\ge}}
\DeclareUnicodeCharacter{221E}{\ensuremath{\infty}}
\DeclareUnicodeCharacter{2192}{\ensuremath{\to}}
\DeclareUnicodeCharacter{21D2}{\ensuremath{\Rightarrow}}
\DeclareUnicodeCharacter{2208}{\ensuremath{\in}}
\DeclareUnicodeCharacter{221A}{\ensuremath{\surd}}
\DeclareUnicodeCharacter{2211}{\ensuremath{\sum}}
\DeclareUnicodeCharacter{222B}{\ensuremath{\int}}
\DeclareUnicodeCharacter{2112}{\ensuremath{\mathcal L}}
\DeclareUnicodeCharacter{1D49C}{\ensuremath{\alpha}}
\DeclareUnicodeCharacter{2193}{\ensuremath{\downarrow}}
\DeclareUnicodeCharacter{2202}{\ensuremath{\partial}}
\DeclareUnicodeCharacter{2260}{\ensuremath{\ne}}
\DeclareUnicodeCharacter{21A6}{\ensuremath{\mapsto}}
\DeclareUnicodeCharacter{00B2}{\ensuremath{{}^{2}}}
\DeclareUnicodeCharacter{00B3}{\ensuremath{{}^{3}}}
\DeclareUnicodeCharacter{2074}{\ensuremath{{}^{4}}}
\DeclareUnicodeCharacter{2075}{\ensuremath{{}^{5}}}
\DeclareUnicodeCharacter{2076}{\ensuremath{{}^{6}}}
\DeclareUnicodeCharacter{2077}{\ensuremath{{}^{7}}}
\DeclareUnicodeCharacter{2078}{\ensuremath{{}^{8}}}
\DeclareUnicodeCharacter{2080}{\ensuremath{{}_{0}}}
\DeclareUnicodeCharacter{2081}{\ensuremath{{}_{1}}}
\DeclareUnicodeCharacter{2082}{\ensuremath{{}_{2}}}
\DeclareUnicodeCharacter{208B}{\ensuremath{{}_{-}}}
\usepackage{mathrsfs}

\allowdisplaybreaks
\begin{document}
\RUNAUTHOR{Hu} \RUNTITLE{$(r,Q)$ Policies for Lost Sales} \TITLE{\Large \textcolor{black}{A 1.932 Guarantee for} $(r,Q)$ \textcolor{black}{Policies}\\in Continuous-Review Lost-Sales Inventory Systems}
\ARTICLEAUTHORS{\AUTHOR{Ming Hu$^{*}$} \AFF{Rotman School of Management, University of Toronto, Toronto, Ontario M5S 3E6, Canada\\\EMAIL{ming.hu@rotman.utoronto.ca}}}
\ABSTRACT{We provide a uniform guarantee for a continuous-review lost-sales inventory system with unit Poisson demand, deterministic lead time, and fixed ordering costs: a simple algorithm computes an $(r,Q)$ policy and chooses between that policy and no replenishment, with cost at most $(\sqrt2+\sqrt6)/2<1.932$ times the unrestricted nonanticipative optimum. The algorithm starts by generating an optimal integer base-stock policy, which at zero setup cost achieves the tight worst-case factor $\sqrt2$ relative to the optimal system cost. For positive setup costs, it uses the reorder point of this integer base-stock policy and an integer lot at the economic order quantity scale. The inclusion of ordinary base stock in the continuous-review capped base-stock family also improves that family's earlier asymptotic $1.79$ upper guarantee to a uniform $\sqrt2$ guarantee valid at every finite parameter vector. To establish these guarantees, we construct a computable lower bound on the zero-setup optimal cost by tracking units already on hand or committed for delivery. We then combine it with a bound on holding and setup costs required by batching to obtain the general-setup guarantee and instance-specific performance bounds for any admissible replenishment policy. In our numerical study, the largest estimated policy-cost-to-lower-bound ratios are about $1.59$ for our heuristic $(r, Q)$ policy and $1.47$ after joint optimization of $r$ and $Q$. The heuristic also performs well against the optimized policy: over the sampled instances, its median extra cost ranges from $2.33\%$ to $3.91\%$, depending on the setup cost.} \maketitle
\begingroup
\renewcommand{\thefootnote}{\fnsymbol{footnote}}%
\footnotetext[1]{The author used AI tools (such as ChatGPT, Claude, and Refine.ink) to assist with derivations, numerical verification, and editing. The author developed and verified the substantive research questions, theoretical framework, analysis, and results. The author takes sole responsibility for the paper's content. Theorem~\ref{thm:article-main}, Corollary~\ref{cor:policy-selection}, Propositions~\ref{prop:rank-lower} to~\ref{prop:full-conversion}, and Lemma~\ref{lem:endpoints} are formally verified in Lean~4; the formal proofs are public on Prove2Me at \provemeurl. Their formal statements, including the definitions of the model and of the policies, match those of the paper, and these proofs rely only on Lean's three standard axioms.}%
\endgroup
\setlength{\abovedisplayskip}{5pt plus 2pt minus 1pt}
\setlength{\belowdisplayskip}{5pt plus 2pt minus 1pt}
\setlength{\abovedisplayshortskip}{2pt plus 1pt minus 1pt}
\setlength{\belowdisplayshortskip}{4pt plus 1pt minus 1pt}

\vspace{-2em}
\section{Introduction}\label{sec:introduction}\label{sec:article-introduction}

We establish a uniform cost guarantee for a simple replenishment rule in continuous-review lost-sales systems with fixed ordering costs. Under unit Poisson demand and deterministic lead time, our algorithm selects an $(r,Q)$ policy or no replenishment with cost at most $1.932$ times the optimal system cost when holding cost and demand rate are positive. The optimal system cost is defined over all admissible replenishment policies, including those that use the full inventory and delivery history. To our knowledge, this is the first constant-factor guarantee valid for all finite parameter values in this model with positive replenishment lead times and arbitrary fixed ordering costs.

Under a classical $(r,Q)$ policy, the manager orders $Q$ units whenever inventory position, the sum of on-hand and outstanding units, falls to $r$ \citep{HadleyWhitin1963}. The reorder point protects against lead-time demand and the lot spreads setup costs over several units. A base-stock policy orders enough to restore inventory position to its target. In our continuous-review model with unit demand, $Q=1$ implements this rule with target $r+1$: each sale triggers a one-unit replacement. No replenishment is also an option when it is not economic to carry the item. In the classical continuous-review model with unit Poisson demand, deterministic lead time, and linear holding and backlogging costs, unmet demand is backlogged until it can be served. An $(r,Q)$ policy is optimal with positive setup costs, and a base-stock policy is optimal at zero setup cost \citep{FedergruenZheng1992,Song1994}. Our results show that these simple policy classes also have uniform cost guarantees when unmet demand is lost.

Specifically, we consider unit Poisson demand of rate $\lambda$, deterministic lead time $L$, holding-cost rate $h$, lost-sales penalty $p$, and fixed order cost $K$, with all five primitives nonnegative. Let $\Pi_{\rm rQ}$ consist of classical $(r,Q)$ policies and no replenishment, and let $\OPT$ be the infimum long-run average cost over all admissible nonanticipative replenishment policies. This benchmark allows history-dependent and randomized order times and arbitrary positive integer order quantities. Following the standard nonperishable lost-sales convention \citep{Hill2007,LeviJN2008,Xin2022}, policies control replenishment only: inventory on hand is carried until it serves demand, and disposal is unavailable. Our main result is
\begin{equation}
\inf_{\pi\in\Pi_{\rm rQ}}C(\pi)
\le \frac{\sqrt2+\sqrt6}{2}\,\OPT.
\label{eq:intro-main}
\end{equation}
The proved factor is $1.93185165\ldots$, rounded upward to $1.932$ in the title. The bound holds at each finite primitive vector. For $h,\lambda>0$, the proof selects one finite classical policy or no replenishment; at zero holding cost, the conclusion is an inequality between infima.

Let $\pi_{\rm S}(S)$ denote the level-$S$ base-stock policy for $S\ge1$, with $\pi_{\rm S}(0)$ denoting no replenishment. At $K=0$, the best integer base-stock policy has the stronger guarantee
\begin{equation}
\inf_{S\in\mathbb Z_{\ge0}}C(\pi_{\rm S}(S))\le\sqrt2\,\OPT.
\label{eq:intro-value}
\end{equation}
For $h,\lambda>0$, our algorithm attains this infimum with an optimal integer base-stock policy, whose factor $\sqrt2$ is \emph{tight} against the unrestricted optimum. Because the continuous-review capped base-stock family of \citet{Xin2022} contains ordinary base stock, our upper bound also applies to that larger family. It improves the earlier asymptotic $1.79$ guarantee in the same Poisson-demand, deterministic-lead-time model and holds at every finite parameter vector. The tightness result concerns ordinary base stock at zero setup cost; the tight worst-case factor for $(r,Q)$ policies with positive setup costs remains open.

The guarantee gives a simple algorithm to configure an effective $(r,Q)$ policy. For $h,\lambda>0$, write $\mu:=\lambda L$, $a:=p\lambda/h$, and $k:=K\lambda/h$. First choose any cost-minimizing integer base-stock level $S^\star$ at zero setup cost, using the closed-form cost expression in \eqref{eq:base-stock-cost}. If $S^\star\ge1$, adopt the reorder point $S^\star-1$ and set
\[
Q_k:=\max\{1,\lfloor\sqrt{2k}\rfloor\},\qquad
b_k:=\frac{Q_k-1}{2}+\frac{k}{Q_k}.
\]
The lot size $Q_k$ follows the classical economic order quantity (EOQ) scale $\sqrt{2k}=\sqrt{2K\lambda/h}$, which balances holding and setup costs. Writing $U_\mu(a)$ for the normalized base-stock cost, the resulting classical policy has cost at most $U_\mu(a)+b_k$. Choose this $(r,Q)$ policy if its cost upper bound is below $a$; otherwise choose no replenishment. If $S^\star=0$, choose no replenishment directly. All quantities in this algorithm can be efficiently computed from the demand and cost parameters.

\begin{figure}[htbp]
\SingleSpaced
\centering
\includegraphics[width=.99\linewidth]{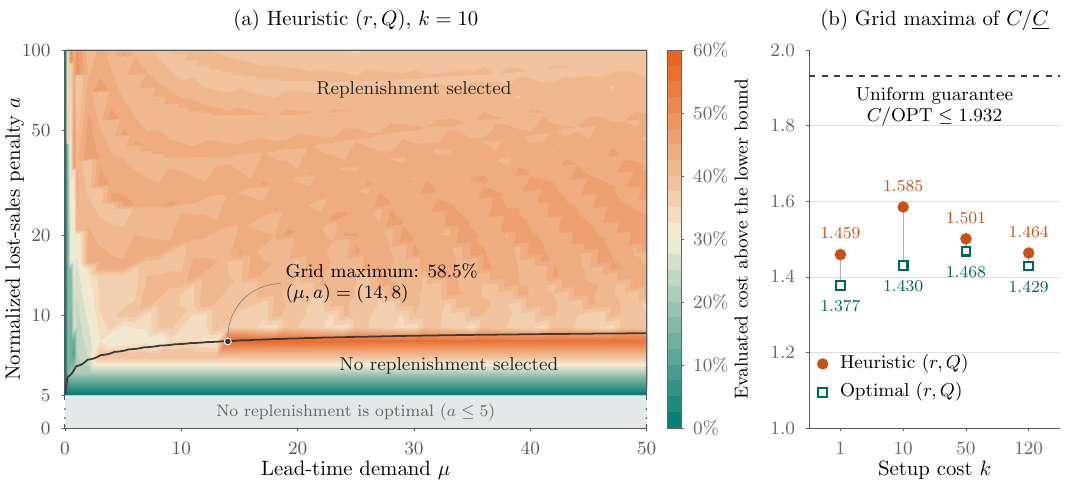}
\caption{Instance-specific performance and the uniform guarantee}
\label{fig:performance-overview}
\begin{minipage}{0.98\linewidth}
\SingleSpaced\footnotesize\emph{Notes.} Panel (a) shows $[C_{\rm heuristic}/\underline C_k(\mu,a)-1]$ at $k=10$, using the heuristic's cost and the optimal-cost lower bound from Section~\ref{sec:numerics}. The curve $U_\mu(a)+4=a$ separates the algorithm's choices. The $\mu$ axis is linear; the $a$ axis is logarithmic above $5$, with the optimal no-replenishment region $0\le a\le5$ compressed. Panel (b) gives grid maxima; ``Optimal $(r,Q)$'' denotes numerical joint optimization, including no replenishment. Costs are exact where available and independently simulated otherwise; contours interpolate grid values. The ratios use a lower bound on optimal cost, while the $1.932$ line gives the guarantee relative to the actual optimum.
\end{minipage}
\end{figure}

Figure~\ref{fig:performance-overview} shows how this heuristic $(r,Q)$ policy performs across instances. At setup cost $k=10$, the contour map locates the largest evaluated cost-to-lower-bound ratio, about $1.585$, near the switch between replenishment and no replenishment. The companion panel summarizes the four setup costs studied in Section~\ref{sec:numerics}. These evaluated ratios estimate instance-specific upper bounds on the approximation ratio, while the theorem provides the factor $1.932$ uniformly over all parameters.

\FloatBarrier

The challenge is to compare this simple policy with the optimal policy, which can use the entire observed inventory and delivery history and whose structure is unknown. A lost sale leaves no customer waiting for a later delivery. Moreover, every delivery within the next lead time has already been ordered. Even an adaptive policy must therefore balance the risk of running out against the cost of supply that arrives before it is needed. We capture this restriction by recording each committed unit's position in the service order and its remaining delivery time.

The proof separates two economic tradeoffs. First, with setup costs removed, we construct a function of the model parameters and prove that it is no greater than the cost of any admissible policy. This lower bound is concave in the lost-sales penalty, while the cost of each fixed base-stock policy is affine in that penalty. If a base-stock policy costs at most $\sqrt2$ times the lower bound at two penalty values, the same inequality holds throughout the interval between them. Carefully chosen overlapping intervals establish the factor $\sqrt2$ for every lead time and penalty value. Second, batching spreads each setup charge over several units while adding inventory. We show that every admissible policy pays for the same tradeoff, which converts any zero-setup performance guarantee $c$ into
\[
R(c):=\frac{c+\sqrt{c^2+4}}2
\]
for the general setup cost setting. Setting $c=\sqrt2$ yields the factor in \eqref{eq:intro-main}.

These proof steps also give computable lower bounds for evaluating replenishment policies. Tracking committed supply bounds the zero-setup optimal cost. For positive setup costs, we allocate part of holding cost and all setup costs to bound the cost per unit served. We then apply the zero-setup lower bound to the remaining holding cost and an adjusted lost-sales penalty, obtaining a lower bound on total cost. Comparing a policy's cost with the resulting lower bound gives an instance-specific upper bound on its optimality gap, whether the policy comes from our algorithm or another heuristic. In our numerical experiments, the largest estimated cost-to-lower-bound ratios are about $1.59$ for the heuristic $(r,Q)$ policy and $1.47$ after joint optimization of $r$ and $Q$, both below the uniform factor $1.932$. Direct cost comparisons also show that the simple heuristic performs well relative to the numerically optimized $(r,Q)$ policy: its median extra cost ranges from $2.33\%$ to $3.91\%$ across the four setup costs studied. Thus, the lower bound and the direct cost comparisons help assess the value of additional optimization on a particular instance.

\section{Literature Review}\label{sec:literature}

\indent\textbf{Continuous-review control and replenishment timing.} Classical inventory theory studies reorder-point and fixed-lot rules \citep{HadleyWhitin1963,FedergruenZheng1992}; \citet{BijvankVis2011} survey lost-sales models. With positive setup costs, \citet{JohansenThorstenson1996} and \citet{HillJohansen2006} analyze lost-sales systems with at most one order outstanding. In the Poisson setting without fixed costs, \citet{Hill2007} studies minimum-delay unit-order policies. Our guarantee allows multiple outstanding orders and compares the classical $(r,Q)$ family with unrestricted nonanticipative replenishment.

\indent\textbf{Order quantities and review intervals.} Periodic-review policies also pace replenishment. An $(S,T_{\mathrm{rev}})$ policy reviews every $T_{\mathrm{rev}}$ time units and orders up to $S$. In serial systems, \citet{ShangSongZipkin2009} study coordination under periodic-review $(R,nQ)$ policies, while \citet{ShangZhou2010} optimize review intervals jointly with reorder points and batch sizes. Here the ordering trigger remains the inventory position; the theorem selects a reorder point and a fixed lot within the classical family.

\indent\textbf{Capped base-stock policies and approximation guarantees.} \citet{Reiman2004} compares base-stock and fixed-cycle policies, termed constant-order-interval policies, as the lost-sales penalty and lead time grow together. Building on this analysis, \citet{Xin2022} establishes an asymptotic $1.79$ guarantee for the cheaper of the two policies and hence for the capped base-stock family, which contains both. Ordinary base stock belongs to that family by taking the minimum time between unit orders to zero. This inclusion transfers our finite $\sqrt2$ upper guarantee to continuous-review capped base-stock policies, improving Xin's asymptotic factor in the same lost-sales model. Tightness here concerns the optimal ordinary base-stock policy selected by our algorithm.

In periodic review, \citet{Zipkin2008} compares several lost-sales heuristics and documents limitations of ordinary base stock. \citet{Xin2021} caps the \emph{quantity} ordered toward a base-stock target, and \citet{Xin2026} proves a finite $7/3$ guarantee for this class under independent and identically distributed demand and no setup cost. That result covers broader demand distributions; ours complements theirs by covering continuous-review Poisson demand and positive setup costs. The constants concern different models and policy classes.

Table~\ref{tab:closest-guarantees} summarizes the closest finite and asymptotic guarantees and their distinct scopes.

{\setlength{\intextsep}{0pt}\setlength{\textfloatsep}{0pt}%
\begin{table}[!htbp]
\vspace*{5pt}%
\centering
\caption{Closest approximation guarantees for lost-sales replenishment policies}
\label{tab:closest-guarantees}
\fontsize{8}{12}\selectfont
\setlength{\tabcolsep}{4pt}\renewcommand{\arraystretch}{1.05}
\begin{tabular}{@{}p{.16\linewidth}p{.31\linewidth}p{.10\linewidth}p{.37\linewidth}@{}}
\toprule
Result & Review and demand & Setup cost & Guarantee\\
\midrule
\citet{Xin2022} & Continuous review, unit Poisson demand & $K=0$ & Asymptotic $1.79$ for capped base stock as $p\to\infty$ with $L$ proportional to $p$\\
\citet{Xin2026} & Periodic review, general i.i.d.\ demand & $K=0$ & Finite $7/3$ for capped base stock\\
\citet{LeviJN2008} & Periodic review, independent nonstationary demand & $K=0$ & Finite $2$ for dual-balancing policies\\
This paper & Continuous review, unit Poisson demand & Any $K\ge0$ & Finite $<1.932$ for $(r,Q)$; tight $\sqrt2$ for the algorithm's optimal base stock at $K=0$\\
\bottomrule
\end{tabular}
\par\vspace{4pt}
\begin{minipage}{0.98\linewidth}
\SingleSpaced\fontsize{8}{12}\selectfont\emph{Notes.} Each ratio uses its model's unrestricted nonanticipative benchmark; the proportionality constant is fixed. Xin (2022) caps order frequency; Xin (2026) caps order quantity. Levi et al.\ use dual balancing. Our classical family includes no replenishment. The tightness statement concerns base stock at zero setup cost.
\end{minipage}
\par
\vspace{4pt}
\end{table}}

Other results without fixed ordering costs include asymptotic optimality of base-stock policies as shortage penalties grow \citep{HuhJMR2009} and asymptotic optimality of constant-order policies as lead times grow \citep{GoldbergEtAl2016,XinGoldberg2016}. On the backlogging side, \citet{LeviEtAl2007} develop approximation algorithms for stochastic inventory models, and \citet{LeviShi2013} give a factor-three guarantee for stochastic lot sizing with fixed ordering costs and positive lead times through randomized cost balancing. Lost sales require a different lower bound because unmet demand leaves no backlog for future orders to serve.

\section{Model and Policies}\label{sec:model}\label{sec:model-policies}

\subsection{System and Admissible Policies}\label{sec:admissibility}

Customers arrive according to a Poisson process of rate \(\lambda\) and request one unit. Replenishment has deterministic lead time \(L\). Holding costs \(h\) per unit per unit time, a lost demand costs \(p\), and a positive order costs \(K\). Service is mandatory whenever on-hand inventory is positive. Inventory and order quantities are integer-valued; orders cannot be canceled and on-hand units cannot be discarded. A policy may use the observed history and independent randomization, but cannot anticipate demand. Its order epochs are locally finite and its positive quantities finite. Initial inventory and initial pipeline contain finitely many units almost surely and are chosen independently of future demand. The initial pipeline is exempt from the newly placed order lead-time requirement; no bound by \(L\) on its residual delivery times is imposed. Initial supply is free; any finite expected initialization charge would not affect an average cost. Orders at a demand epoch follow that demand. When \(L=0\), replenishment cannot retroactively serve a demand already lost.

Writing $X(t)$ for on-hand inventory, $\mathcal U(t)$ for cumulative lost demand, and $O(t)$ for positive orders, define the unnormalized objective for all nonnegative primitives by
\begin{equation}
C(\pi):=\limsup_{T\to\infty}\frac1T\E_\pi\left[h\int_0^T X(t)\,dt+p\,\mathcal U(T)+KO(T)\right],
\qquad \OPT:=\inf_{\pi\text{ admissible}}C(\pi).
\label{eq:objective}
\end{equation}
The benchmark allows arbitrary history dependence and integer order quantities. The no-disposal convention enters the proof: each delivered unit must remain in stock until a demand consumes it, which is used in the calculation based on on-hand inventory.

\subsection{Normalized Parameters}
For \(h,\lambda>0\), normalize time by \(\lambda\) and average cost by \(h\), and write
\[
\mu:=\lambda L,\qquad a:=\frac{p\lambda}{h},\qquad k:=\frac{K\lambda}{h}.
\]
 Henceforth \(N(t)\) has rate one, holding costs one, and lead time is \(\mu\). Write \(X(t)\) for on-hand inventory, \(\mathcal A(t)\) for cumulative served demand, \(\mathcal U(t):=N(t)-\mathcal A(t)\) for losses, \(B^{\rm del}(t)\) for delivered units, and \(O(t)\) for the number of positive orders. Inventory position is on-hand plus pending supply. For an admissible policy \(\pi\) define
\[
C_{\mu,a,k}(\pi):=\limsup_{T\to\infty}\frac1T\mathbb E_\pi\left[\int_0^T X(t)\,dt+a\,\mathcal U(T)+kO(T)\right],
\qquad
\operatorname{OPT}_k(\mu,a):=\inf_\pi C_{\mu,a,k}(\pi).
\]
 Finite expected cost on finite horizons is required whenever the corresponding cost coefficient is positive. No stationary-policy or Markov-policy restriction is imposed. Until the batching analysis, \(k=0\); abbreviate \(\operatorname{OPT}_0(\mu,a)\) by \(\operatorname{OPT}_0(a)\) when the mean is fixed. The no-order policy starts empty and has cost \(a\).

We use the following identity to convert changes at demand epochs into expected rates. The filtration includes the observed history and independent policy randomization, so $N(t)-t$ remains a martingale \citep{Kingman1993}.
\begin{lemma}[{\sc Poisson Compensation}]\label{lem:poisson-compensation}
Let $\psi(t)$ be predictable, meaning that its value is determined before any demand at time $t$, and bounded by a deterministic constant on each fixed finite horizon. For every finite $T$,
\[
\E\sum_{0<t\le T:\,\Delta N(t)=1}\psi(t)
=\E\int_0^T\psi(t)\,dt.
\]
The identity also holds with $T$ replaced by a stopping time $\sigma$ with $\E\sigma<\infty$, provided $\psi$ is uniformly bounded. In particular, for a bounded adapted process with right-continuous paths and left limits, we evaluate its demand contribution at $t-$; replacing $t-$ by $t$ in the time integral does not change that integral.
\end{lemma}
\proof{Proof.}
The compensated integral $\int_0^T\psi(t)\,[dN(t)-dt]$ has mean zero. Applying this identity to the positive and negative parts of $\psi(t)\mathbf1_{\{t\le\sigma\}}$ and then using monotone convergence proves the stopped version; boundedness and $\E\sigma<\infty$ make both sides finite.\Halmos
\endproof

\subsection{Replenishment Rules}\label{sec:policy}

The classical policy $\pi_{\rm rQ}(r,Q)$ has integer reorder point $r\ge0$ and lot $Q\ge1$. It starts with $r+Q$ units on hand and an empty pipeline. A served demand that reduces inventory position from $r+1$ to $r$ triggers one immediate order of $Q$ units. The post-order position is $r+Q$, so no second order is triggered at that epoch. With $r=S-1$ and $Q=1$, this is the level-$S$ base-stock policy $\pi_{\rm S}(S)$. We separately include the empty-start no-order policy in $\Pi_{\rm rQ}$.

Table~\ref{tab:policy} records the parameters and the no-replenishment option of this same family. The zero-setup analysis studies its unit-lot member; the full-cost construction changes the lot size while retaining the inventory trigger.

\begin{table}[htbp]
\centering\small
\caption{Parameters and the no-replenishment option}\label{tab:policy}
\fontsize{8}{12}\selectfont
\begin{tabular}{@{}lll@{}}
\toprule
Case & Parameters & Ordering decision\\
\midrule
Classical replenishment & Integers $r\ge0$, $Q\ge1$ & Order $Q$ when position reaches $r$\\
Unit-lot base stock & $r=S-1$, $Q=1$, $S\ge1$ & Replace each sold unit\\
No replenishment & Empty initial system & Never order\\
\bottomrule
\end{tabular}
\end{table}

\subsection{Exact Base-Stock Costs}\label{sec:frontiers}
\label{sec:base-stock-costs}

\begingroup\predisplaypenalty=0
Write
\[
p_j(t):=e^{-t}\frac{t^j}{j!},\qquad F_j(t):=\sum_{i=0}^j p_i(t),\qquad
G_l(t):=\mathbb E(l-N(t))^+,
\]
\endgroup
 where $p_i:=0$ and $F_i:=0$ for $i<0$. At the fixed mean \(\mu>0\), suppress arguments and write
\[
e_j:=\frac{p_j}{F_j},\qquad A_j:=\frac{F_j}{p_j},\qquad
H_j:=j-\mu+\mu e_j,\qquad e_0=1,\quad H_0=0.
\]
 The symbol \(A_j\) is an Erlang normalizer; \(\alpha_\mu(\eta)\) below is a calibrated penalty.

A served demand ties up one of the $j$ inventory positions until its replacement arrives $\mu$ time units later. Thus, pending replacements form an $\mathrm M/\mathrm D/j/j$ loss system. Its stationary occupied-position count has a Poisson$(\mu)$ distribution conditional on being at most $j$ \citep{Zachary2007}. Poisson arrivals see these stationary probabilities \citep{Wolff1982}, so $e_j$ is the fraction of demand lost and $H_j$ is mean on-hand inventory. Appendix~\ref{app:erlang-details} justifies the same average costs from our initialization and derives the moment formulas used below.

Thus, the stationary on-hand count \(I_j\) is distributed as \(j-N(\mu)\) conditional on \(N(\mu)\le j\):
\[
\mathbb P(I_j=l)=\frac{p_{j-l}}{F_j}
=\frac{(j)_l\mu^{-l}}{A_j},\qquad 0\le l\le j,
\qquad
A_j=\sum_{l=0}^j(j)_l\mu^{-l}.
\]
 Here $(j)_l:=j(j-1)\cdots(j-l+1)$ for integers $1\le l\le j$, and $(j)_0:=1$. In particular, $\E I_j=H_j$. The zero-setup cost of a base-stock level is its mean inventory plus the penalty times its loss fraction.
\begin{equation}\label{eq:base-stock-cost}
L_j(a):=H_j+ae_j=j-\mu+(\mu+a)e_j,
\qquad
U_\mu(a):=\min_{j\ge0}L_j(a).
\end{equation}
 The minimum is attained because \(H_j\ge j-\mu\).

\section{Main Results}\label{sec:main-results}\subsection{Main Theorem}\label{sec:article-results}

Write $\Pi_{\rm S}:=\{\pi_{\rm S}(S):S\in\mathbb Z_{\ge0}\}$ for integer base-stock policies, including no replenishment. Costs and the unrestricted benchmark $\OPT$ are defined in the original units by \eqref{eq:objective}.

\begin{theorem}[{\sc Zero-Setup and Full-Cost Guarantees}]\label{thm:article-main}
For every nonnegative primitive vector $(\lambda,L,h,p,K)$,
\begin{equation}
\inf_{\pi\in\Pi_{\rm rQ}} C(\pi)
\le\frac{\sqrt2+\sqrt6}{2}\,\OPT.
\label{eq:article-full}
\end{equation}
When $K=0$,
\begin{equation}
\inf_{\pi\in\Pi_{\rm S}} C(\pi)\le\sqrt2\,\OPT.
\label{eq:article-operating}
\end{equation}
For $h,\lambda>0$, the algorithm in Corollary~\ref{cor:policy-selection} returns an optimal integer base-stock policy at $K=0$. Its $\sqrt2$ guarantee is tight: in normalized parameters,
\begin{equation}
\sup_{\mu,a>0}\frac{U_\mu(a)}{\operatorname{OPT}_0(\mu,a)}=\sqrt2.
\label{eq:article-sharp}
\end{equation}
The ratio is well defined: for $\mu,a>0$, the lower bound in \eqref{eq:concave-lower-bound} satisfies $\operatorname{OPT}_0(\mu,a)\ge\mathcal L_\mu(a)\ge\min\{1,a\}>0$.
\end{theorem}

For $h,\lambda>0$, dividing cost by $h$ gives the equivalent normalized inequalities $U_\mu(a)\le\sqrt2\operatorname{OPT}_0(\mu,a)$ and $C_{\rm rQ}(\mu,a,k)\le R\operatorname{OPT}_k(\mu,a)$, where $R=(\sqrt2+\sqrt6)/2$ and $C_{\rm rQ}(\mu,a,k)$ denotes the infimum over $\Pi_{\rm rQ}$. The following corollary gives a finite policy attaining the full-cost guarantee without knowing $\OPT$.

\begin{corollary}[{\sc An Explicit Policy Selection}]\label{cor:policy-selection}
Suppose $h,\lambda>0$. Compute a cost-minimizing integer base-stock level $S^\star$ at zero setup cost and set
\[
Q_k:=\max\{1,\lfloor\sqrt{2k}\rfloor\},\qquad
b_k:=\frac{Q_k-1}{2}+\frac{k}{Q_k}.
\]
If $S^\star\ge1$ and $U_\mu(a)+b_k<a$, choose $(r,Q)=(S^\star-1,Q_k)$; otherwise choose no replenishment. The selected policy has cost at most $R\OPT$, with $R=\sqrt2$ when $K=0$ and $R=(\sqrt2+\sqrt6)/2$ for general $K$. At $K=0$, it is optimal among integer base-stock policies, including level zero.
\end{corollary}

The batching argument proves that the replenishment candidate costs at most $U_\mu(a)+b_k$, with $b_k\le\sqrt{2k}$. Comparing this bound with the no-order cost $a$ gives the algorithm whose performance is shown in Figure~\ref{fig:performance-overview}. A factor of $1.932$ means that its cost is at most $93.2\%$ above the unrestricted optimum. Theorem~\ref{thm:article-main} treats zero holding cost as an inequality between infima; a finite minimizing policy need not exist there.

\subsection{Proof Structure}\label{sec:proof-roadmap}
We organize the proof around the two decisions in the algorithm.
\begin{enumerate}[(i)]
\item \emph{Bound the zero-setup cost of every policy.} In Section~\ref{sec:lower}, we construct a concave function $\mathcal L_\mu(a)$ below $\operatorname{OPT}_0(\mu,a)$. Tracking committed units ensures that this bound applies even to history-dependent replenishment.
\item \emph{Find an integer base-stock level within $\sqrt2$ of that bound.} Section~\ref{sec:comparison-plan} compares base-stock costs with $\mathcal L_\mu$ at selected penalties. Concavity fills the intervals between those penalties, proving $U_\mu\le\sqrt2\mathcal L_\mu$.
\item \emph{Include setup costs.} Section~\ref{sec:batching} bounds both the cost of batching the selected base-stock policy and the holding/setup cost that every policy must incur. Matching these bounds gives $R^2=\sqrt2R+1$, hence $R=(\sqrt2+\sqrt6)/2$.
\end{enumerate}
The main text explains the inequalities connecting these steps. The appendices supply the probability comparisons and finite-horizon cost calculations. The Gaussian comparisons used in the appendices are one-sided inequalities valid at every finite $\mu$.

\subsection{A Lower Bound from Committed Supply}\label{sec:lower}\label{sec:committed-supply}

\subsubsection{Rank and remaining delivery time.}\label{sec:rank-potential}
For the lower-bound proof, we may start with an empty system: Appendix~\ref{app:initial-supply} shows that removing any finite initial supply changes service by only a sublinear expected amount. All subsequent residual delivery times therefore lie in $[0,\mu]$.

Fix a time and look ahead by one lead time. Every unit that can arrive within that interval has already been ordered. We list the units on hand and in transit in the order in which they can serve customers, breaking delivery ties deterministically. We track a unit's position in the service order, which we call its \emph{rank}. Units on hand have remaining delivery time zero; pending units have ordered remaining times in $[0,\mu]$. Deterministic lead time ensures that a new order appends units to this list at remaining time $\mu$.

\begin{figure}[!t]
\SingleSpaced
\centering
\input{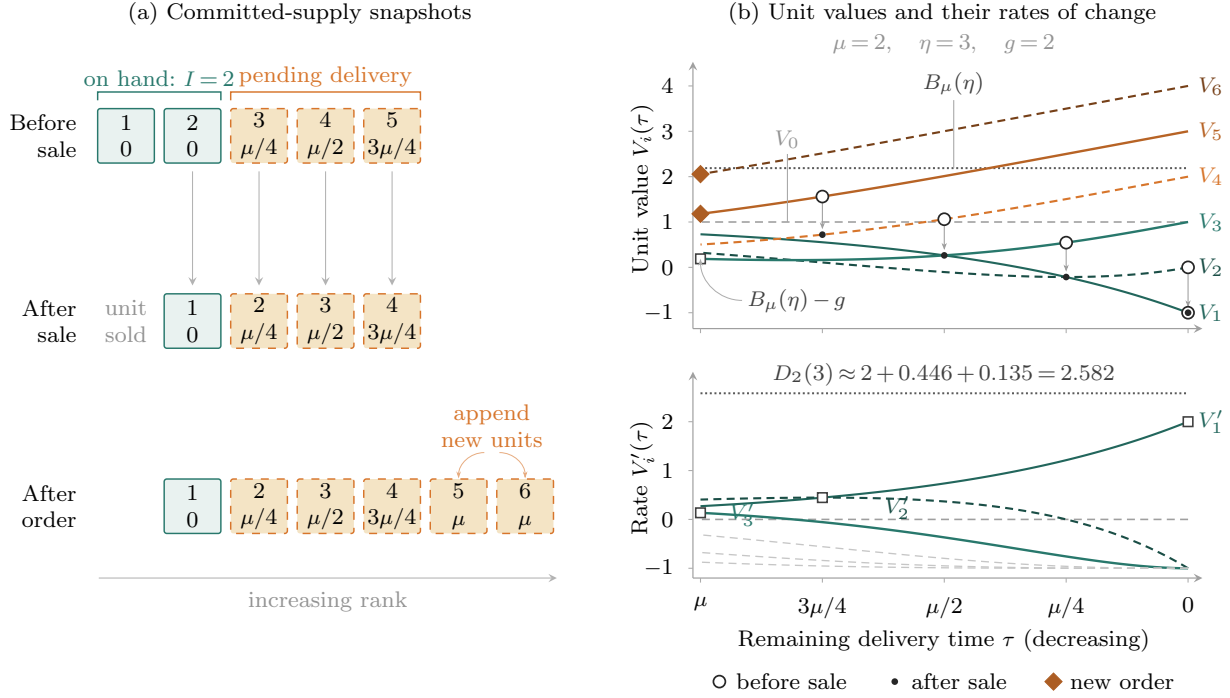}
\caption{Committed supply, unit values, and their rates of change}
\label{fig:rank-supply}
\begin{minipage}{0.98\linewidth}
\SingleSpaced\footnotesize\emph{Notes.} Each box shows the unit's rank on the top line and its remaining delivery time on the bottom line; solid borders indicate on-hand units and dashed borders pending units. Columns follow the same units through two instantaneous events. In panel (b), circles mark their values before and after the sale, and diamonds mark newly ordered units. Vertical arrows show rank changes. The dotted line marks $B_2(3)\approx2.188$; the marked minimum at $\tau=\mu$ is $B_2(3)-g=V_3(\mu)$. In the rate plot, squares mark positive individual maxima, whose sum is $D_2(3)\approx2.582$; ranks $i\ge4$ have nonpositive rates. The auxiliary function $V_0$ has no corresponding unit.
\end{minipage}
\end{figure}

Figure~\ref{fig:rank-supply}(a) shows how a sale and a new order change this list. A sale removes the first unit and reduces every subsequent rank by one. Between events, each pending unit moves closer to delivery. Because deliveries never cross, nothing else happens to the list: ranks shift down by exactly one at each sale, remaining times decrease at unit rate, and new units join at the end. We assign a value to each unit according to its rank and remaining delivery time, and sum these values in a potential function. The values are chosen so that, while inventory is positive, the two changes a pending unit undergoes cancel each other exactly.

Fix a proposed lower bound $g$ on average zero-setup cost and an auxiliary penalty $\eta\ge1$. Assign to a unit of rank $i\ge1$ and remaining delivery time $\tau$ the value
\[
V_i(\tau):=G_i(\tau)+\eta\Prb(N(\tau)\ge i)-g,
\qquad V_0(\tau):=\eta-g.
\]
Recall that $G_i(\tau)=\E[(i-N(\tau))^+]$. The two terms before $-g$ measure the expected excess of the rank over demand before delivery and the weighted probability that demand reaches the rank. The upper plot in Figure~\ref{fig:rank-supply}(b) shows these values of $V_i(\tau)$, with remaining delivery time decreasing from left to right. Between events, a pending unit moves right along the curve for its current rank. A sale moves each surviving unit to the curve for the next lower rank at the same remaining time. The new order adds two units at $\tau=\mu$.

While inventory is positive, every demand is served, and a pending unit of rank $i$ experiences two effects. Random sales, which arrive at expected rate one, move it from the curve for rank $i$ to the curve for rank $i-1$ at its current remaining time, and the passage of time moves it along its current curve toward delivery, changing its value at rate $-V_i'(\tau)$. Since $G_i-G_{i-1}=F_{i-1}$, the values satisfy
\[
V_i'(\tau)=V_{i-1}(\tau)-V_i(\tau),\qquad i\ge2,
\]
so the expected change from rank shifts is exactly offset by the aging drift, whatever the unit's remaining time. The positive-inventory part of the argument therefore requires no knowledge of when pending units will arrive. For the $I$ units on hand, $V_i(0)=i-g$: a sale removes the first unit, of value $1-g$, and lowers each other on-hand value by one, so these units contribute $g-I$ to the change at a sale; with sales at expected rate one, holding cost at rate $I$ brings the total rate to $g$. The service order itself matters only through two facts: ranks are well defined, and because the deterministic lead time rules out order crossing, each sale lowers every rank by exactly one.

Two events are not covered by this cancellation, and each gives one of the two quantities in the lower bound. The first is an order. A newly ordered unit enters the list with remaining time $\mu$ and, at rank $l$, carries value $V_l(\mu)$; nothing offsets its entry, so we require its value to be nonnegative. Newly added units have nonnegative value exactly when $g$ is at most
\begin{equation}\label{eq:rank-minimum}
B_\mu(\eta):=g+\min_{l\ge0}V_l(\mu)=\min_{l\ge0}\{G_l(\mu)+\eta\Prb(N(\mu)\ge l)\}.
\end{equation}
Here $l$ is an integer rank\footnote{Rank zero contributes $\eta$, while rank one contributes $\eta-(\eta-1)e^{-\mu}\le\eta$, so including rank zero leaves the minimum unchanged. The rank-zero term $\eta$ is the starting term of the telescoping sum in Appendix~\ref{app:calibration-proof}, which expresses $\eta-B_\mu(\eta)$ as a sum of rank increments.}, and the $-g$ in $V_l$ cancels the leading $g$, so $B_\mu(\eta)$ does not depend on $g$. Thus, $B_\mu(\eta)-g$ is the smallest value a new unit can carry, and $g\le B_\mu(\eta)$ guarantees that ordering never lowers the sum of unit values.

The second is a stockout. With no unit on hand, demand is lost and leaves all ranks unchanged, so the aging drift of every pending unit goes uncompensated and the assigned values change at total rate $-\sum_i V_i'(\tau_i)$. Define the greatest possible rate of decrease by
\begin{equation}\label{eq:schedule-supremum}
D_\mu(\eta):=\sup_{m\ge0,\,0\le\tau_1\le\cdots\le\tau_m\le\mu}
\sum_{i=1}^m V_i'(\tau_i),
\end{equation}
where the derivative is $V_i'(\tau)=\eta p_{i-1}(\tau)-F_{i-1}(\tau)$. The supremum covers all numbers of possible pending units and delivery schedules, including the empty sum and simultaneous deliveries. In the system, the realized configuration $(\tau_1(t),\ldots,\tau_{m(t)}(t))$ at time $t$ is random: a unit ordered at time $s$ has remaining time $\mu-(t-s)$, so the configuration is determined by past order epochs, which under a history-dependent policy depend on earlier sales and lost demands. But the supremum does not: at every time, in every realization, and under every admissible policy, the current configuration is one of the schedules admitted in the supremum, so its rate of decrease $\sum_iV_i'(\tau_i)$ is at most $D_\mu(\eta)$. Hence, the definition takes a supremum rather than an expectation, and $D_\mu(\eta)$ is the largest rate at which aging can lower the values in any configuration. At a stockout, lost-sales cost accrues at rate $a$, while approaching deliveries can decrease the potential function at rate at most $D_\mu(\eta)$, so cost plus the change in the potential function has rate at least $a-D_\mu(\eta)$. Requiring $a-D_\mu(\eta)\ge g$ says that the lost-sales penalty covers the uncompensated drift.

Select $g_\eta:=\min\{B_\mu(\eta),a-D_\mu(\eta)\}$, the largest value meeting the two requirements imposed by the events not covered by the cancellation. Use $g=g_\eta$ in the unit values defined above. A nonpositive cost lower bound is vacuous, so suppose $g_\eta>0$. Choose an integer $M\ge\max\{1,\mu+g_\eta\}$ and let $\Phi_M$ be the sum of $V_i(\tau_i)$ over the first $M$ committed units (all units if fewer). This bounded potential function gains only nonnegative retained terms at orders, because $g_\eta\le B_\mu(\eta)$, and does not jump at deliveries, since a delivered unit keeps its rank and $V_i$ is continuous.

Let $\Delta_{\rm dem}\Phi_M(t)$ denote the change in the potential function that a demand at time $t$ would cause, computed before any resulting order adds new units. We record the demand and any order placed immediately afterward as two separate changes. Let $\dot\Phi_M(t)$ denote the change per unit time as pending deliveries approach. Since expected demand has rate one, the jump $\Delta_{\rm dem}\Phi_M$ is also the expected change from demand per unit time, so the following calculation is in expectation: the holding and loss rates plus the demand and aging contributions satisfy
{\renewcommand{\arraystretch}{0.6}%
\[
I+a\mathbf1_{\{I=0\}}+\Delta_{\rm dem}\Phi_M+\dot\Phi_M
\ge
\left\{\begin{array}{@{}l@{\quad}l@{}}
I+(g_\eta-I)=g_\eta, & 1\le I\le M,\\
I\ge g_\eta, & I>M,\\
a-D_\mu(\eta)\ge g_\eta, & I=0.
\end{array}\right.
\]}%
The holding-cost rate is $I$ because its coefficient has been normalized to one.\footnote{Average cost has been divided by the original holding-cost rate $h$, so $I$ units on hand incur normalized holding cost at rate $I$, corresponding to $hI$ in the original units.}
Appendix~\ref{app:rank-potential} verifies these three cases. Since expected demand has rate one, Lemma~\ref{lem:poisson-compensation} identifies $\E\int_0^T\Delta_{\rm dem}\Phi_M(t)\,dt$ as the expected sum of the changes caused by processing demand before any resulting orders. Integrating the rate bound then gives
\[
\begin{aligned}
&\E\int_0^T\!\bigl[I(t)+a\mathbf1_{\{I(t)=0\}}\bigr]\,dt+\E[\Phi_M(T)-\Phi_M(0)]\\
&\quad\ge\E\int_0^T\!\bigl[I(t)+a\mathbf1_{\{I(t)=0\}}\bigr]\,dt+\E\int_0^T\!\bigl[\Delta_{\rm dem}\Phi_M(t)+\dot\Phi_M(t)\bigr]\,dt
\ge g_\eta T.
\end{aligned}
\]
The first integral is expected holding and lost-sales cost. The second integral includes the expected changes from demand and approaching deliveries. The total expected change in the potential function, $\E[\Phi_M(T)-\Phi_M(0)]$, also includes the expected sum of the nonnegative increases when orders are placed, whether immediately after demand or at other times. Adding this nonnegative expected sum gives the first inequality. Since $\Phi_M$ is bounded, $\E[\Phi_M(T)-\Phi_M(0)]/T$ approaches zero as $T\to\infty$, yielding the following result.

\begin{proposition}[{\sc Committed-Supply Lower Bound}]\label{prop:rank-lower}
For every $a\ge0$ and $\eta\ge1$,
\begin{equation}\label{eq:committed-supply-bound}
\operatorname{OPT}_0(\mu,a)\ge\min\{B_\mu(\eta),a-D_\mu(\eta)\}.
\end{equation}
\end{proposition}

Appendix~\ref{app:rank-potential} completes the finite-horizon argument. Appendix~\ref{app:initial-supply} also shows that finite initial supply does not affect the bound.

\subsubsection{Balancing the two requirements.}\label{sec:calibration}
The weight $\eta$ trades one requirement against the other. Both $B_\mu(\eta)$ and $D_\mu(\eta)$ are nondecreasing in $\eta$, so a larger weight raises $B_\mu(\eta)$ but lowers $a-D_\mu(\eta)$. For a chosen $\eta$, setting $a=B_\mu(\eta)+D_\mu(\eta)$ makes both terms in \eqref{eq:committed-supply-bound} equal to $B_\mu(\eta)$, and the following identity expresses this balancing penalty as an expected maximum.

\begin{proposition}[{\sc Exact Penalty Calibration}]\label{prop:calibration}
For $\eta\ge1$, define
\begin{equation}\label{eq:calibration-map}
\alpha_\mu(\eta):=\eta\left[1+\E\sup_{0\le t\le\mu}\{(1-1/\eta)t-N(t)\}\right].
\end{equation}
Then
\begin{equation}\label{eq:calibration-identity}
B_\mu(\eta)+D_\mu(\eta)=\alpha_\mu(\eta).
\end{equation}
The function $\alpha_\mu$ is a continuous strictly increasing bijection from $[1,\infty)$ onto itself.
\end{proposition}

Inside the supremum in \eqref{eq:calibration-map}, a line of slope $1-1/\eta$ rises between demand arrivals and falls by one at each arrival. Its largest value measures the greatest excess of this auxiliary cumulative supply over demand during one lead time. The maximum includes time zero and is, therefore, nonnegative. The expected maximum gives a convenient way to calculate the balancing penalty.

Appendix~\ref{app:calibration-proof} proves the identity~\eqref{eq:calibration-identity} by maximizing each rank's aging term and summing the increases over its value at remaining time $\mu$. In particular, every penalty $a\ge1$ corresponds to a unique calibrated weight $\eta$, and at that penalty, \eqref{eq:committed-supply-bound} gives the bound $B_\mu(\eta)$. The function $\alpha_\mu$ therefore identifies, for each penalty $a\ge1$, the unique weight $\eta$ that makes the two lower-bound terms equal.

Figure~\ref{fig:calibration-surfaces} shows how $B_\mu(\eta)$ and $D_\mu(\eta)$ vary with the mean lead-time demand and the auxiliary weight. Panel (c) shows the sum $B_\mu(\eta)+D_\mu(\eta)$, which Proposition~\ref{prop:calibration} identifies with $\alpha_\mu(\eta)$, and panel (d) shows sample paths of the auxiliary process in \eqref{eq:calibration-map}, whose expected maximum determines that sum. The marked point in panels (a) to (c), and the parameters of panel (d), are those of Figure~\ref{fig:rank-supply}.

\begin{figure}[htbp]
\SingleSpaced
\centering
\includegraphics[width=.98\linewidth]{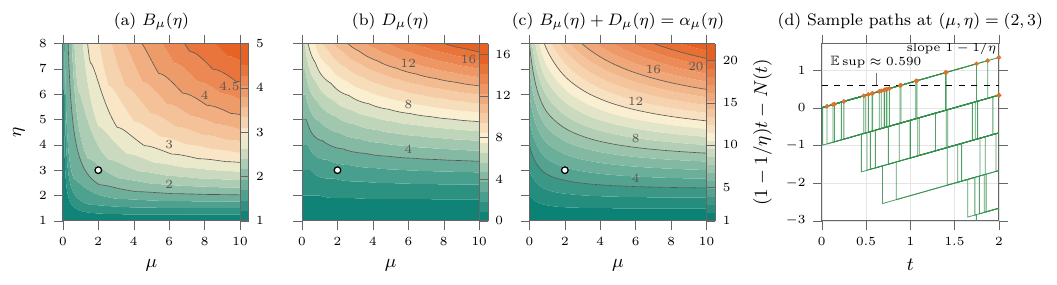}
\caption{$B_\mu(\eta)$, $D_\mu(\eta)$, their sum $\alpha_\mu(\eta)$, and sample paths of the auxiliary process}
\label{fig:calibration-surfaces}
\begin{minipage}{.98\linewidth}
\SingleSpaced\footnotesize\emph{Notes.} Panels (a) to (c) show filled contours over $0\le\mu\le10$ and $1\le\eta\le8$, each with its own color scale; kinks in panel (a) occur where the minimizing rank in \eqref{eq:rank-minimum} changes. Open dots mark $(\mu,\eta)=(2,3)$, where $B_2(3)\approx2.188$, $D_2(3)\approx2.582$, and $\alpha_2(3)\approx4.770$. Panel (d) shows $27$ simulated paths of $(1-1/\eta)t-N(t)$ at $(\mu,\eta)=(2,3)$; the dotted line shows $(1-1/\eta)t$, the path when no demand occurs; filled dots mark the simulated paths' maxima over $[0,\mu]$. The dashed line marks the expected maximum over all demand paths, $\alpha_2(3)/3-1\approx0.590$ from \eqref{eq:calibration-identity}.
\end{minipage}
\end{figure}

\subsubsection{The concave lower curve.}\label{sec:concave-lower-bound}
Define
\begin{equation}\label{eq:concave-lower-bound}
\mathcal L_\mu(a):=\sup_{\eta\ge1}\min\{B_\mu(\eta),a-D_\mu(\eta)\}.
\end{equation}
For each fixed $\eta$, the bound $\min\{B_\mu(\eta),a-D_\mu(\eta)\}$ rises with slope one until it reaches the cap $B_\mu(\eta)$. Figure~\ref{fig:lower-curve-envelope}(a) shows these capped lines and the upper envelope obtained by choosing the best weight at each penalty. Panel (b) shows $U_\mu(a)/\mathcal L_\mu(a)$, an upper bound on the base-stock approximation ratio. Figure~\ref{fig:base-stock-tightness} later provides the complementary lower bound that establishes tightness.

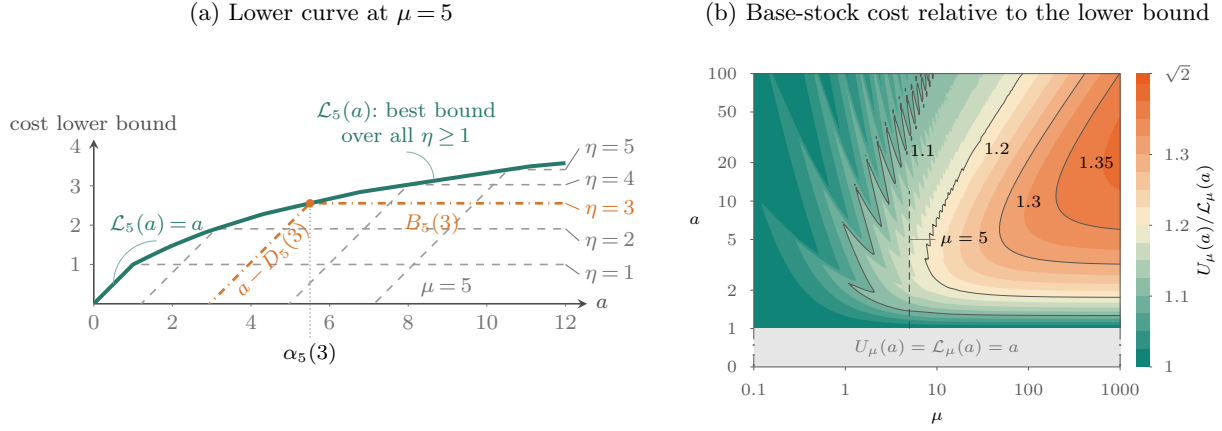
\begin{figure}[htbp]
\SingleSpaced
\centering
\newsavebox{\LowerCurveRatioPanel}
\sbox{\LowerCurveRatioPanel}{\includegraphics[width=.44\linewidth]{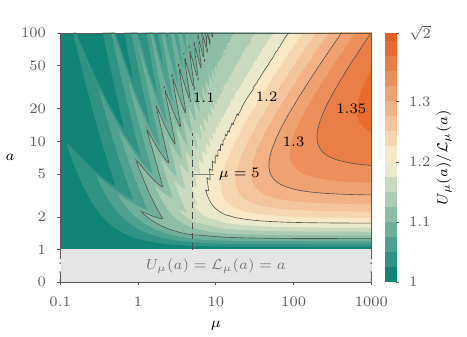}}
\begin{minipage}[t]{0.54\linewidth}
\centering\footnotesize
(a) Lower curve at $\mu=5$ \par\smallskip
\parbox[c][\dimexpr\ht\LowerCurveRatioPanel+\dp\LowerCurveRatioPanel\relax][c]{\linewidth}{\centering
\begin{tikzpicture}[x=.52cm,y=.52cm,font=\scriptsize,>=stealth]
\draw[->,black!70,line width=.6pt] (0,0)--(12.55,0) node[anchor=west] {$a$};
\draw[->,black!70,line width=.6pt] (0,0)--(0,4.18) node[anchor=south] {cost lower bound};
\foreach \x in {0,2,4,6,8,10,12} {\draw[black!55] (\x,0)--(\x,-.07) node[anchor=north] {$\x$};}
\foreach \y in {1,2,3,4} {\draw[black!55] (0,\y)--(-.10,\y) node[anchor=east] {$\y$};}
\begin{scope}
\clip (0,0) rectangle (12,4.1);
\draw[gray!80,dashed,line width=.6pt] (0,-0.00000000)--(1.00000000,1.00000000)--(12,1.00000000);
\draw[gray!80,dashed,line width=.6pt] (0,-1.20600408)--(3.11279581,1.90679173)--(12,1.90679173);
\draw[lsOrng,dash dot,line width=1.0pt] (0,-2.94496219)--(5.50081918,2.55585699)--(12,2.55585699);
\draw[gray!80,dashed,line width=.6pt] (0,-4.97245773)--(8.00191261,3.02945488)--(12,3.02945488);
\draw[gray!80,dashed,line width=.6pt] (0,-7.13273124)--(10.54622547,3.41349423)--(12,3.41349423);
\draw[lsTeal,line width=1.5pt,line join=round] plot coordinates {(0.000000,0.000000) (1.000000,1.000000) (1.029313,1.014607) (1.058625,1.029214) (1.087938,1.043820) (1.117251,1.058427) (1.146563,1.073034) (1.175876,1.087641) (1.205189,1.102248) (1.234501,1.116854) (1.263814,1.131461) (1.293127,1.146068) (1.322439,1.160675) (1.351752,1.175282) (1.381065,1.189888) (1.398652,1.198652) (1.410377,1.204297) (1.439690,1.218408) (1.469003,1.232520) (1.498316,1.246631) (1.527676,1.260742) (1.557134,1.274854) (1.586694,1.288965) (1.616358,1.303076) (1.646128,1.317188) (1.676005,1.331299) (1.705988,1.345411) (1.736078,1.359522) (1.766274,1.373633) (1.796574,1.387745) (1.826976,1.401856) (1.857478,1.415967) (1.888079,1.430079) (1.918775,1.444190) (1.949566,1.458301) (1.969320,1.467333) (1.980447,1.471967) (2.011418,1.484840) (2.042474,1.497712) (2.073614,1.510585) (2.104835,1.523458) (2.136135,1.536331) (2.167511,1.549203) (2.198960,1.562076) (2.230481,1.574949) (2.262071,1.587822) (2.293728,1.600694) (2.325450,1.613567) (2.357234,1.626440) (2.389090,1.639313) (2.421051,1.652186) (2.453116,1.665058) (2.485283,1.677931) (2.517550,1.690804) (2.549914,1.703677) (2.582373,1.716549) (2.614925,1.729422) (2.647566,1.742295) (2.680294,1.755168) (2.713109,1.768040) (2.746006,1.780913) (2.778985,1.793786) (2.812042,1.806659) (2.845177,1.819531) (2.857920,1.824475) (2.878386,1.831133) (2.911668,1.841941) (2.945022,1.852750) (2.978445,1.863558) (3.011936,1.874366) (3.045492,1.885175) (3.079113,1.895983) (3.112796,1.906792) (3.146540,1.917600) (3.180343,1.928409) (3.214205,1.939217) (3.248123,1.950026) (3.282096,1.960834) (3.316123,1.971642) (3.350202,1.982451) (3.384332,1.993259) (3.418512,2.004068) (3.452740,2.014876) (3.487016,2.025685) (3.521339,2.036493) (3.555706,2.047301) (3.590118,2.058110) (3.624573,2.068918) (3.659069,2.079727) (3.693607,2.090535) (3.728185,2.101344) (3.762802,2.112152) (3.797457,2.122961) (3.832150,2.133769) (3.866879,2.144577) (3.901644,2.155386) (3.936444,2.166194) (3.971278,2.177003) (4.006145,2.187811) (4.041045,2.198620) (4.075976,2.209428) (4.110939,2.220237) (4.145932,2.231045) (4.180955,2.241853) (4.216007,2.252662) (4.251088,2.263470) (4.286197,2.274279) (4.311046,2.281923) (4.321341,2.284332) (4.356528,2.292560) (4.391757,2.300788) (4.427026,2.309016) (4.462336,2.317244) (4.497685,2.325472) (4.533073,2.333700) (4.568499,2.341928) (4.603962,2.350156) (4.639462,2.358384) (4.674997,2.366612) (4.710568,2.374840) (4.746173,2.383068) (4.781811,2.391296) (4.817483,2.399524) (4.853188,2.407752) (4.888925,2.415980) (4.924693,2.424208) (4.960491,2.432436) (4.996320,2.440664) (5.032179,2.448892) (5.068067,2.457121) (5.103984,2.465349) (5.139928,2.473577) (5.175901,2.481805) (5.211900,2.490033) (5.247926,2.498261) (5.283979,2.506489) (5.320057,2.514717) (5.356161,2.522945) (5.392289,2.531173) (5.428442,2.539401) (5.464619,2.547629) (5.500819,2.555857) (5.537043,2.564085) (5.573289,2.572313) (5.609559,2.580541) (5.645850,2.588769) (5.682162,2.596997) (5.718497,2.605225) (5.754852,2.613453) (5.791228,2.621681) (5.827624,2.629909) (5.864040,2.638137) (5.900476,2.646365) (5.936931,2.654593) (5.973406,2.662822) (6.009899,2.671050) (6.046410,2.679278) (6.082940,2.687506) (6.119487,2.695734) (6.156052,2.703962) (6.192635,2.712190) (6.229234,2.720418) (6.265851,2.728646) (6.302484,2.736874) (6.339133,2.745102) (6.375798,2.753330) (6.412479,2.761558) (6.449176,2.769786) (6.485887,2.778014) (6.522614,2.786242) (6.559356,2.794470) (6.596113,2.802698) (6.632884,2.810926) (6.669669,2.819154) (6.706468,2.827382) (6.743281,2.835610) (6.769324,2.841429) (6.780108,2.843083) (6.816948,2.848730) (6.853801,2.854378) (6.890667,2.860026) (6.927547,2.865673) (6.964438,2.871321) (7.001343,2.876969) (7.038260,2.882616) (7.075188,2.888264) (7.112129,2.893912) (7.149082,2.899559) (7.186046,2.905207) (7.223022,2.910854) (7.260009,2.916502) (7.297007,2.922150) (7.334017,2.927797) (7.371037,2.933445) (7.408068,2.939093) (7.445109,2.944740) (7.482161,2.950388) (7.519224,2.956036) (7.556296,2.961683) (7.593378,2.967331) (7.630471,2.972979) (7.667573,2.978626) (7.704685,2.984274) (7.741806,2.989921) (7.778937,2.995569) (7.816077,3.001217) (7.853226,3.006864) (7.890385,3.012512) (7.927552,3.018160) (7.964728,3.023807) (8.001913,3.029455) (8.039106,3.035103) (8.076308,3.040750) (8.113518,3.046398) (8.150737,3.052045) (8.187963,3.057693) (8.225198,3.063341) (8.262441,3.068988) (8.299691,3.074636) (8.336950,3.080284) (8.374216,3.085931) (8.411490,3.091579) (8.448771,3.097227) (8.486059,3.102874) (8.523355,3.108522) (8.560659,3.114169) (8.597969,3.119817) (8.635287,3.125465) (8.672611,3.131112) (8.709942,3.136760) (8.747280,3.142408) (8.784625,3.148055) (8.821977,3.153703) (8.859335,3.159351) (8.896700,3.164998) (8.934071,3.170646) (8.971449,3.176293) (9.008832,3.181941) (9.046222,3.187589) (9.083619,3.193236) (9.121021,3.198884) (9.158429,3.204532) (9.195844,3.210179) (9.233264,3.215827) (9.270690,3.221475) (9.308121,3.227122) (9.345559,3.232770) (9.383002,3.238417) (9.420450,3.244065) (9.457905,3.249713) (9.495364,3.255360) (9.532829,3.261008) (9.570299,3.266656) (9.607775,3.272303) (9.645256,3.277951) (9.682742,3.283599) (9.720233,3.289246) (9.757729,3.294894) (9.795230,3.300541) (9.832736,3.306189) (9.870247,3.311837) (9.907763,3.317484) (9.945283,3.323132) (9.982809,3.328780) (10.020339,3.334427) (10.057873,3.340075) (10.095412,3.345723) (10.132956,3.351370) (10.170505,3.357018) (10.208057,3.362665) (10.245614,3.368313) (10.283176,3.373961) (10.320742,3.379608) (10.358312,3.385256) (10.395886,3.390904) (10.433465,3.396551) (10.471048,3.402199) (10.508635,3.407847) (10.546225,3.413494) (10.583821,3.419142) (10.621423,3.424790) (10.659031,3.430437) (10.696644,3.436085) (10.734263,3.441732) (10.771888,3.447380) (10.809518,3.453028) (10.847154,3.458675) (10.884795,3.464323) (10.922442,3.469971) (10.960094,3.475618) (10.997751,3.481266) (11.035414,3.486914) (11.073082,3.492561) (11.089983,3.495095) (11.110755,3.497023) (11.148434,3.500520) (11.186117,3.504018) (11.223805,3.507515) (11.261499,3.511012) (11.299197,3.514510) (11.336901,3.518007) (11.374609,3.521504) (11.412322,3.525002) (11.450039,3.528499) (11.487762,3.531996) (11.525489,3.535494) (11.563221,3.538991) (11.600957,3.542488) (11.638698,3.545985) (11.676443,3.549483) (11.714193,3.552980) (11.751948,3.556477) (11.789706,3.559975) (11.827470,3.563472) (11.865237,3.566969) (11.903009,3.570467) (11.940785,3.573964) (11.978565,3.577461) (12.000000,3.579445)};
\end{scope}
\draw[black!40,densely dotted,line width=.5pt] (5.50081918,-.86)--(5.50081918,2.55585699);
\fill[lsOrng] (5.50081918,2.55585699) circle (1.6pt);
\node[anchor=north,inner sep=2pt] at (5.50081918,-.86) {$\alpha_5(3)$};
\draw[black!55,line width=.4pt] (12,1.00000000)--(12.35,0.75);
\node[anchor=west,inner sep=2pt,text=black!55] at (12.35,0.75) {$\eta=1$};
\draw[black!55,line width=.4pt] (12,1.90679173)--(12.35,1.60);
\node[anchor=west,inner sep=2pt,text=black!55] at (12.35,1.60) {$\eta=2$};
\draw[lsOrng,line width=.4pt] (12,2.55585699)--(12.35,2.40);
\node[anchor=west,inner sep=2pt,text=lsOrng] at (12.35,2.40) {$\eta=3$};
\draw[black!55,line width=.4pt] (12,3.02945488)--(12.35,3.15);
\node[anchor=west,inner sep=2pt,text=black!55] at (12.35,3.15) {$\eta=4$};
\draw[black!55,line width=.4pt] (12,3.41349423)--(12.35,3.95);
\node[anchor=west,inner sep=2pt,text=black!55] at (12.35,3.95) {$\eta=5$};
\node[anchor=south,text=lsTeal,inner sep=2pt] at (1.6,1.6) {$\mathcal L_5(a)=a$};
\draw[lsTeal!60,line width=.3pt] (1.6,1.6) arc[start angle=90,end angle=180,radius=1.1];
\node[anchor=south,text=lsTeal,align=center,inner sep=2pt] at (7.95,3.86) {$\mathcal L_5(a)$: best bound\\over all $\eta\ge1$};
\draw[lsTeal!60,line width=.3pt] (7.95,3.86) arc[start angle=90,end angle=0,radius=.7];
\node[text=lsOrng,rotate=45,inner sep=1.5pt] at (4.55,1.12) {$a-D_5(3)$};
\node[text=lsOrng,anchor=north] at (8.6,2.52) {$B_5(3)$};
\node[text=black!55,anchor=west] at (8.1,.43) {$\mu=5$};
\end{tikzpicture}}
\end{minipage}\hfill
\begin{minipage}[t]{0.44\linewidth}
\centering\footnotesize
(b) Base-stock cost relative to the lower bound \par\smallskip
\usebox{\LowerCurveRatioPanel}
\end{minipage}
\caption{The concave lower curve and its comparison with base-stock cost}
\label{fig:lower-curve-envelope}
\begin{minipage}{0.98\linewidth}
\SingleSpaced\footnotesize\emph{Notes.} In panel (a), thin curves show $\min\{B_5(\eta),a-D_5(\eta)\}$; the dash-dot curve uses $\eta=3$ and meets its cap at $\alpha_5(3)$. The thick envelope optimizes over $\eta\ge1$; both axes have equal scales. Panel (b) numerically evaluates $U_\mu(a)/\mathcal L_\mu(a)$, an upper bound on the base-stock approximation ratio. Both axes are logarithmic outside the compressed strip $0\le a\le1$, where $U_\mu(a)=\mathcal L_\mu(a)=a$ (the ratio is undefined at zero). The dashed $\mu=5$ segment links to panel (a) over $1\le a\le12$. Tightness is proved in Section~\ref{sec:article-sharpness}.
\end{minipage}
\end{figure}

Propositions~\ref{prop:rank-lower} and \ref{prop:calibration} give
\begin{equation}\label{eq:lower-bound-comparison}
\mathcal L_\mu(a)\le\operatorname{OPT}_0(\mu,a),\qquad
\mathcal L_\mu(\alpha_\mu(\eta))\ge B_\mu(\eta).\mbox{\footnotemark}
\end{equation}
\footnotetext{For the second inequality, $\mathcal L_\mu(\alpha_\mu(\eta))=\sup_{\eta'\ge1}\min\{B_\mu(\eta'),\alpha_\mu(\eta)-D_\mu(\eta')\}\ge\min\{B_\mu(\eta),\alpha_\mu(\eta)-D_\mu(\eta)\}=B_\mu(\eta)$. The inequality chooses $\eta'=\eta$ in the supremum. The last equality uses Proposition~\ref{prop:calibration}: $\alpha_\mu(\eta)=B_\mu(\eta)+D_\mu(\eta)$, so both entries in the minimum equal $B_\mu(\eta)$.}%
We establish the concavity of $\mathcal L_\mu$, which is needed for the policy comparison. Raising the penalty makes the cost at each fixed base-stock level increase linearly. Between two penalties, a concave lower curve lies above the straight line joining its endpoint values. Section~\ref{sec:comparison-plan} uses this geometry to compare the cost of a fixed base-stock level with $\sqrt2\,\mathcal L_\mu(a)$ throughout a penalty interval.

The rank minimum makes $B_\mu$ concave, and the schedule supremum makes $D_\mu$ convex. Hence, $\Lambda(a,\eta):=\min\{B_\mu(\eta),a-D_\mu(\eta)\}$ is jointly concave: its hypograph is the intersection of the hypographs of its two concave arguments. For penalties $\underline a,\overline a$ and weights $\underline\eta,\overline\eta\ge1$, joint concavity gives
\[
 \Lambda(q\underline a+(1-q)\overline a,\,q\underline\eta+(1-q)\overline\eta)
 \ge q\Lambda(\underline a,\underline\eta)+(1-q)\Lambda(\overline a,\overline\eta),\qquad 0\le q\le1.
\]
Taking arbitrarily close maximizers on the right proves concavity of $\mathcal L_\mu$. Each $\Lambda(\cdot,\eta)$ is nondecreasing and $1$-Lipschitz, and these properties pass to the supremum. Moreover, $D_\mu(\eta)\ge0$ and the choice $\eta=1$ gives $\Lambda(a,1)=\min\{1,a\}$, so $\min\{1,a\}\le\mathcal L_\mu(a)\le a$ for $a\ge0$. Thus, $\mathcal L_\mu$ is finite and continuous on the entire penalty domain.

We next determine the optimal cost exactly when $0\le a\le1$ or $\mu=0$. These cases reduce the remaining proof of the $\sqrt2$ guarantee to $\mu>0$ and $a>1$, and the equality at $a=1$ provides the starting point for the later interval comparisons.

For $0\le a\le1$, choosing $\eta=1$ in \eqref{eq:concave-lower-bound} and using $B_\mu(1)=1$ and $D_\mu(1)=0$\footnote{At $\eta=1$, $G_l(\mu)\ge\Prb(N(\mu)<l)$ for each integer $l\ge1$, so every term in the minimum defining $B_\mu(1)$ is at least one; $l=0$ gives one. Every summand in $D_\mu(1)$ is $p_{i-1}(\tau)-F_{i-1}(\tau)\le0$, and an empty delivery schedule gives zero.} gives $\mathcal L_\mu(a)\ge\min\{1,a\}=a$. The policy that never replenishes achieves average cost $a$. The lower bound and feasible cost therefore match, proving that no replenishment is optimal and
\[
\mathcal L_\mu(a)=\operatorname{OPT}_0(\mu,a)=U_\mu(a)=a
\qquad(0\le a\le1).
\]
When $\mu=0$, $B_0(\eta)=1$ and $D_0(\eta)=\eta-1$.\footnote{At zero lead time, $N(0)=0$. The rank-zero term in $B_0(\eta)$ is $\eta$, and every integer rank $l\ge1$ contributes $l$, so the minimum is one for $\eta\ge1$. All remaining delivery times are zero. In $D_0(\eta)$, the first pending unit contributes $\eta-1$ and each additional unit contributes $-1$; the empty sum is zero. Thus, the supremum is $\eta-1$.} Replenishment is immediate, so base stock one serves every demand and costs $1$, while no replenishment costs $a$. Choosing $\eta=1$ gives the lower bound $\min\{a,1\}$, which matches the cheaper feasible choice. Hence,
\[
\mathcal L_0(a)=U_0(a)=\operatorname{OPT}_0(0,a)=\min\{a,1\}.
\]

\subsection{The \texorpdfstring{$\sqrt2$}{Square-Root-Two} Zero-Setup Guarantee}
\label{sec:base-stock-guarantee}\label{sec:comparison-plan}

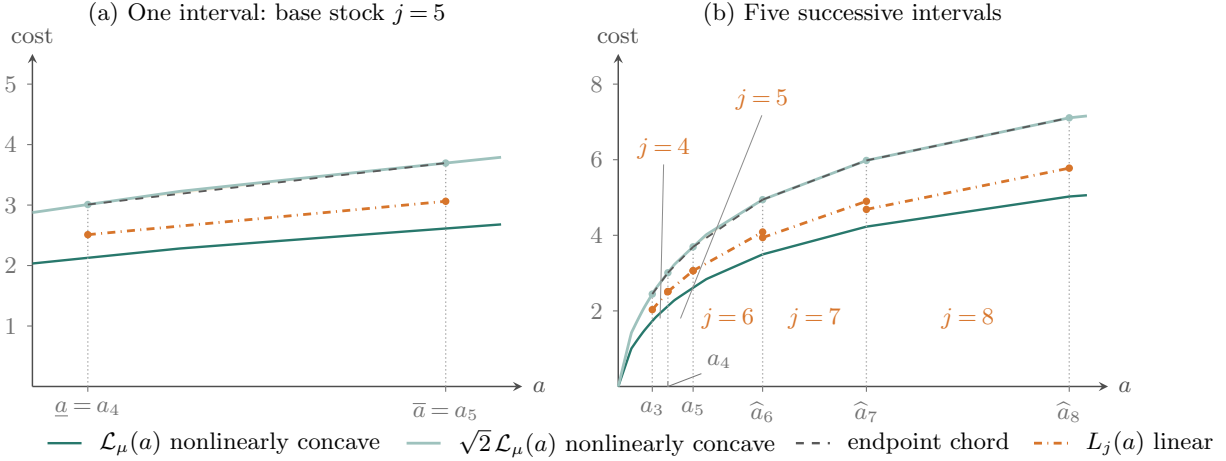
\begin{figure}[!htbp]
\SingleSpaced
\centering
\begin{tikzpicture}[x=1cm,y=1cm,font=\footnotesize,>=stealth]
\begin{scope}[xshift=0cm]
\node[anchor=south] at (3.1,4.63) {(a) One interval: base stock $j=5$};
\draw[->,black!70,line width=.55pt] (0,0)--(6.48,0) node[anchor=west] {$a$};
\draw[->,black!70,line width=.55pt] (0,0)--(0,4.4) node[anchor=south] {cost};
\draw[black!55] (0,0.800)--(-.07,0.800) node[anchor=east] {$1$};
\draw[black!55] (0,1.600)--(-.07,1.600) node[anchor=east] {$2$};
\draw[black!55] (0,2.400)--(-.07,2.400) node[anchor=east] {$3$};
\draw[black!55] (0,3.200)--(-.07,3.200) node[anchor=east] {$4$};
\draw[black!55] (0,4.000)--(-.07,4.000) node[anchor=east] {$5$};
\draw[black!55] (0.732917,0)--(0.732917,-.07) node[anchor=north] {$\underline a=a_4$};
\draw[black!55] (5.467083,0)--(5.467083,-.07) node[anchor=north] {$\overline a=a_5$};
\fill[lsBand,opacity=.6] plot coordinates {(0.732917,2.407162) (0.811820,2.418541) (0.890723,2.429909) (0.969626,2.441267) (1.048528,2.452613) (1.127431,2.463950) (1.206334,2.475276) (1.285237,2.486593) (1.364139,2.497900) (1.443042,2.509198) (1.521945,2.520486) (1.600848,2.531766) (1.679750,2.543037) (1.758653,2.554299) (1.837556,2.565553) (1.916459,2.576799) (1.995361,2.586521) (2.074264,2.595064) (2.153167,2.603597) (2.232070,2.612121) (2.310972,2.620636) (2.389875,2.629143) (2.468778,2.637641) (2.547681,2.646130) (2.626583,2.654612) (2.705486,2.663085) (2.784389,2.671551) (2.863292,2.680009) (2.942194,2.688459) (3.021097,2.696902) (3.100000,2.705338) (3.178903,2.713767) (3.257806,2.722188) (3.336708,2.730603) (3.415611,2.739011) (3.494514,2.747413) (3.573417,2.755808) (3.652319,2.764196) (3.731222,2.772579) (3.810125,2.780955) (3.889028,2.789325) (3.967930,2.797690) (4.046833,2.806049) (4.125736,2.814402) (4.204639,2.822749) (4.283541,2.831091) (4.362444,2.839427) (4.441347,2.847759) (4.520250,2.856085) (4.599152,2.864406) (4.678055,2.872722) (4.756958,2.881033) (4.835861,2.889339) (4.914763,2.897640) (4.993666,2.905937) (5.072569,2.914229) (5.151472,2.922516) (5.230374,2.930799) (5.309277,2.939078) (5.388180,2.947352) (5.467083,2.955622)}--(0.732917,2.407162)--cycle;
\draw[lsTeal,line width=.9pt,line join=round] plot coordinates {(0.000000,1.626971) (0.025833,1.629633) (0.051667,1.632295) (0.077500,1.634955) (0.103333,1.637615) (0.129167,1.640273) (0.155000,1.642930) (0.180833,1.645586) (0.206667,1.648241) (0.232500,1.650896) (0.258333,1.653549) (0.284167,1.656201) (0.310000,1.658852) (0.335833,1.661502) (0.361667,1.664151) (0.387500,1.666800) (0.413333,1.669447) (0.439167,1.672093) (0.465000,1.674739) (0.490833,1.677383) (0.516667,1.680027) (0.542500,1.682670) (0.568333,1.685311) (0.594167,1.687952) (0.620000,1.690592) (0.645833,1.693231) (0.671667,1.695869) (0.697500,1.698507) (0.723333,1.701143) (0.749167,1.703779) (0.775000,1.706413) (0.800833,1.709047) (0.826667,1.711680) (0.852500,1.714312) (0.878333,1.716944) (0.904167,1.719574) (0.930000,1.722204) (0.955833,1.724833) (0.981667,1.727461) (1.007500,1.730088) (1.033333,1.732715) (1.059167,1.735341) (1.085000,1.737966) (1.110833,1.740590) (1.136667,1.743213) (1.162500,1.745836) (1.188333,1.748458) (1.214167,1.751079) (1.240000,1.753700) (1.265833,1.756319) (1.291667,1.758938) (1.317500,1.761557) (1.343333,1.764174) (1.369167,1.766791) (1.395000,1.769407) (1.420833,1.772023) (1.446667,1.774637) (1.472500,1.777252) (1.498333,1.779865) (1.524167,1.782478) (1.550000,1.785090) (1.575833,1.787701) (1.601667,1.790312) (1.627500,1.792922) (1.653333,1.795531) (1.679167,1.798140) (1.705000,1.800748) (1.730833,1.803355) (1.756667,1.805962) (1.782500,1.808568) (1.808333,1.811174) (1.834167,1.813778) (1.860000,1.816383) (1.885833,1.818986) (1.911667,1.821589) (1.937500,1.824191) (1.963333,1.826493) (1.989167,1.828472) (2.015000,1.830451) (2.040833,1.832429) (2.066667,1.834406) (2.092500,1.836382) (2.118333,1.838358) (2.144167,1.840333) (2.170000,1.842307) (2.195833,1.844281) (2.221667,1.846254) (2.247500,1.848226) (2.273333,1.850198) (2.299167,1.852169) (2.325000,1.854139) (2.350833,1.856109) (2.376667,1.858078) (2.402500,1.860047) (2.428333,1.862014) (2.454167,1.863981) (2.480000,1.865948) (2.505833,1.867914) (2.531667,1.869879) (2.557500,1.871843) (2.583333,1.873807) (2.609167,1.875771) (2.635000,1.877733) (2.660833,1.879696) (2.686667,1.881657) (2.712500,1.883618) (2.738333,1.885578) (2.764167,1.887538) (2.790000,1.889497) (2.815833,1.891456) (2.841667,1.893414) (2.867500,1.895371) (2.893333,1.897328) (2.919167,1.899284) (2.945000,1.901240) (2.970833,1.903195) (2.996667,1.905150) (3.022500,1.907104) (3.048333,1.909057) (3.074167,1.911010) (3.100000,1.912963) (3.125833,1.914915) (3.151667,1.916866) (3.177500,1.918817) (3.203333,1.920767) (3.229167,1.922717) (3.255000,1.924666) (3.280833,1.926615) (3.306667,1.928563) (3.332500,1.930511) (3.358333,1.932458) (3.384167,1.934404) (3.410000,1.936351) (3.435833,1.938296) (3.461667,1.940241) (3.487500,1.942186) (3.513333,1.944130) (3.539167,1.946074) (3.565000,1.948017) (3.590833,1.949960) (3.616667,1.951902) (3.642500,1.953844) (3.668333,1.955785) (3.694167,1.957726) (3.720000,1.959666) (3.745833,1.961606) (3.771667,1.963546) (3.797500,1.965485) (3.823333,1.967423) (3.849167,1.969361) (3.875000,1.971299) (3.900833,1.973236) (3.926667,1.975173) (3.952500,1.977109) (3.978333,1.979045) (4.004167,1.980980) (4.030000,1.982915) (4.055833,1.984850) (4.081667,1.986784) (4.107500,1.988718) (4.133333,1.990651) (4.159167,1.992584) (4.185000,1.994516) (4.210833,1.996448) (4.236667,1.998380) (4.262500,2.000311) (4.288333,2.002242) (4.314167,2.004172) (4.340000,2.006102) (4.365833,2.008032) (4.391667,2.009961) (4.417500,2.011889) (4.443333,2.013818) (4.469167,2.015746) (4.495000,2.017673) (4.520833,2.019600) (4.546667,2.021527) (4.572500,2.023454) (4.598333,2.025380) (4.624167,2.027305) (4.650000,2.029231) (4.675833,2.031155) (4.701667,2.033080) (4.727500,2.035004) (4.753333,2.036928) (4.779167,2.038851) (4.805000,2.040774) (4.830833,2.042697) (4.856667,2.044619) (4.882500,2.046541) (4.908333,2.048463) (4.934167,2.050384) (4.960000,2.052305) (4.985833,2.054225) (5.011667,2.056146) (5.037500,2.058065) (5.063333,2.059985) (5.089167,2.061904) (5.115000,2.063823) (5.140833,2.065741) (5.166667,2.067659) (5.192500,2.069577) (5.218333,2.071495) (5.244167,2.073412) (5.270000,2.075328) (5.295833,2.077245) (5.321667,2.079161) (5.347500,2.081077) (5.373333,2.082992) (5.399167,2.084907) (5.425000,2.086822) (5.450833,2.088737) (5.476667,2.090651) (5.502500,2.092565) (5.528333,2.094478) (5.554167,2.096391) (5.580000,2.098304) (5.605833,2.100217) (5.631667,2.102129) (5.657500,2.104041) (5.683333,2.105953) (5.709167,2.107864) (5.735000,2.109775) (5.760833,2.111686) (5.786667,2.113597) (5.812500,2.115507) (5.838333,2.117417) (5.864167,2.119326) (5.890000,2.121236) (5.915833,2.123145) (5.941667,2.125054) (5.967500,2.126962) (5.993333,2.128870) (6.019167,2.130778) (6.045000,2.132686) (6.070833,2.134593) (6.096667,2.136500) (6.122500,2.138407) (6.148333,2.140313) (6.174167,2.142219) (6.200000,2.144125)};
\draw[lsTeal!45,line width=1.05pt,line join=round] plot coordinates {(0.000000,2.300884) (0.025833,2.304650) (0.051667,2.308413) (0.077500,2.312176) (0.103333,2.315937) (0.129167,2.319696) (0.155000,2.323454) (0.180833,2.327210) (0.206667,2.330965) (0.232500,2.334719) (0.258333,2.338471) (0.284167,2.342222) (0.310000,2.345971) (0.335833,2.349719) (0.361667,2.353466) (0.387500,2.357211) (0.413333,2.360955) (0.439167,2.364697) (0.465000,2.368438) (0.490833,2.372178) (0.516667,2.375917) (0.542500,2.379654) (0.568333,2.383390) (0.594167,2.387125) (0.620000,2.390858) (0.645833,2.394590) (0.671667,2.398321) (0.697500,2.402051) (0.723333,2.405779) (0.749167,2.409507) (0.775000,2.413233) (0.800833,2.416958) (0.826667,2.420681) (0.852500,2.424404) (0.878333,2.428125) (0.904167,2.431845) (0.930000,2.435564) (0.955833,2.439282) (0.981667,2.442999) (1.007500,2.446714) (1.033333,2.450429) (1.059167,2.454142) (1.085000,2.457855) (1.110833,2.461566) (1.136667,2.465276) (1.162500,2.468985) (1.188333,2.472693) (1.214167,2.476400) (1.240000,2.480106) (1.265833,2.483811) (1.291667,2.487515) (1.317500,2.491217) (1.343333,2.494919) (1.369167,2.498620) (1.395000,2.502320) (1.420833,2.506019) (1.446667,2.509716) (1.472500,2.513413) (1.498333,2.517109) (1.524167,2.520804) (1.550000,2.524498) (1.575833,2.528191) (1.601667,2.531883) (1.627500,2.535574) (1.653333,2.539264) (1.679167,2.542954) (1.705000,2.546642) (1.730833,2.550329) (1.756667,2.554016) (1.782500,2.557701) (1.808333,2.561386) (1.834167,2.565070) (1.860000,2.568753) (1.885833,2.572435) (1.911667,2.576116) (1.937500,2.579796) (1.963333,2.583051) (1.989167,2.585850) (2.015000,2.588648) (2.040833,2.591445) (2.066667,2.594241) (2.092500,2.597037) (2.118333,2.599831) (2.144167,2.602624) (2.170000,2.605416) (2.195833,2.608207) (2.221667,2.610998) (2.247500,2.613787) (2.273333,2.616575) (2.299167,2.619363) (2.325000,2.622149) (2.350833,2.624935) (2.376667,2.627719) (2.402500,2.630503) (2.428333,2.633286) (2.454167,2.636068) (2.480000,2.638849) (2.505833,2.641629) (2.531667,2.644408) (2.557500,2.647186) (2.583333,2.649964) (2.609167,2.652740) (2.635000,2.655516) (2.660833,2.658291) (2.686667,2.661065) (2.712500,2.663838) (2.738333,2.666611) (2.764167,2.669382) (2.790000,2.672153) (2.815833,2.674923) (2.841667,2.677692) (2.867500,2.680460) (2.893333,2.683227) (2.919167,2.685994) (2.945000,2.688760) (2.970833,2.691525) (2.996667,2.694289) (3.022500,2.697052) (3.048333,2.699815) (3.074167,2.702577) (3.100000,2.705338) (3.125833,2.708098) (3.151667,2.710858) (3.177500,2.713617) (3.203333,2.716375) (3.229167,2.719132) (3.255000,2.721889) (3.280833,2.724645) (3.306667,2.727400) (3.332500,2.730154) (3.358333,2.732908) (3.384167,2.735661) (3.410000,2.738413) (3.435833,2.741165) (3.461667,2.743916) (3.487500,2.746666) (3.513333,2.749415) (3.539167,2.752164) (3.565000,2.754912) (3.590833,2.757660) (3.616667,2.760407) (3.642500,2.763153) (3.668333,2.765898) (3.694167,2.768643) (3.720000,2.771387) (3.745833,2.774130) (3.771667,2.776873) (3.797500,2.779615) (3.823333,2.782357) (3.849167,2.785098) (3.875000,2.787838) (3.900833,2.790577) (3.926667,2.793316) (3.952500,2.796055) (3.978333,2.798792) (4.004167,2.801529) (4.030000,2.804266) (4.055833,2.807002) (4.081667,2.809737) (4.107500,2.812472) (4.133333,2.815206) (4.159167,2.817939) (4.185000,2.820672) (4.210833,2.823404) (4.236667,2.826136) (4.262500,2.828867) (4.288333,2.831597) (4.314167,2.834327) (4.340000,2.837057) (4.365833,2.839785) (4.391667,2.842514) (4.417500,2.845241) (4.443333,2.847968) (4.469167,2.850695) (4.495000,2.853421) (4.520833,2.856146) (4.546667,2.858871) (4.572500,2.861596) (4.598333,2.864319) (4.624167,2.867043) (4.650000,2.869765) (4.675833,2.872488) (4.701667,2.875209) (4.727500,2.877930) (4.753333,2.880651) (4.779167,2.883371) (4.805000,2.886091) (4.830833,2.888810) (4.856667,2.891528) (4.882500,2.894246) (4.908333,2.896964) (4.934167,2.899681) (4.960000,2.902397) (4.985833,2.905113) (5.011667,2.907829) (5.037500,2.910544) (5.063333,2.913258) (5.089167,2.915973) (5.115000,2.918686) (5.140833,2.921399) (5.166667,2.924112) (5.192500,2.926824) (5.218333,2.929536) (5.244167,2.932247) (5.270000,2.934958) (5.295833,2.937668) (5.321667,2.940378) (5.347500,2.943087) (5.373333,2.945796) (5.399167,2.948504) (5.425000,2.951212) (5.450833,2.953920) (5.476667,2.956627) (5.502500,2.959333) (5.528333,2.962040) (5.554167,2.964745) (5.580000,2.967451) (5.605833,2.970155) (5.631667,2.972860) (5.657500,2.975564) (5.683333,2.978267) (5.709167,2.980970) (5.735000,2.983673) (5.760833,2.986375) (5.786667,2.989077) (5.812500,2.991779) (5.838333,2.994480) (5.864167,2.997180) (5.890000,2.999880) (5.915833,3.002580) (5.941667,3.005279) (5.967500,3.007978) (5.993333,3.010677) (6.019167,3.013375) (6.045000,3.016073) (6.070833,3.018770) (6.096667,3.021467) (6.122500,3.024164) (6.148333,3.026860) (6.174167,3.029556) (6.200000,3.032251)};
\draw[black!35,densely dotted,line width=.5pt] (0.732917,0.000000)--(0.732917,2.407162);
\fill[lsTeal!45] (0.732917,2.407162) circle (1.4pt);
\draw[black!35,densely dotted,line width=.5pt] (5.467083,0.000000)--(5.467083,2.955622);
\fill[lsTeal!45] (5.467083,2.955622) circle (1.4pt);
\draw[black!65,dashed,line width=.7pt] (0.732917,2.407162)--(5.467083,2.955622);
\draw[lsOrng,dash dot,line width=1.05pt] (0.732917,2.008320)--(5.467083,2.449934);
\fill[lsOrng] (0.732917,2.008320) circle (1.4pt);
\fill[lsOrng] (5.467083,2.449934) circle (1.4pt);
\end{scope}
\begin{scope}[xshift=7.75cm]
\node[anchor=south] at (3.1,4.63) {(b) Five successive intervals};
\draw[->,black!70,line width=.55pt] (0,0)--(6.48,0) node[anchor=west] {$a$};
\draw[->,black!70,line width=.55pt] (0,0)--(0,4.4) node[anchor=south] {cost};
\draw[black!55] (0,1.000)--(-.07,1.000) node[anchor=east] {$2$};
\draw[black!55] (0,2.000)--(-.07,2.000) node[anchor=east] {$4$};
\draw[black!55] (0,3.000)--(-.07,3.000) node[anchor=east] {$6$};
\draw[black!55] (0,4.000)--(-.07,4.000) node[anchor=east] {$8$};
\draw[black!55] (0.450377,0)--(0.450377,-.07) node[anchor=north] {$a_3$};
\draw[black!55] (0.656599,0)--(0.656599,.07);
\node[anchor=south west,text=black!55,inner sep=2pt] (penaltyLabelFour) at (1.100000,.170000) {$a_4$};
\draw[black!55,line width=.3pt] (penaltyLabelFour.south west)--(0.656599,0);
\draw[black!55] (0.990331,0)--(0.990331,-.07) node[anchor=north] {$a_5$};
\draw[black!55] (1.909942,0)--(1.909942,-.07) node[anchor=north] {$\widehat a_6$};
\draw[black!55] (3.282819,0)--(3.282819,-.07) node[anchor=north] {$\widehat a_7$};
\draw[black!55] (5.966187,0)--(5.966187,-.07) node[anchor=north] {$\widehat a_8$};
\fill[lsBand,opacity=.6] plot coordinates {(0.450377,1.222934) (0.453815,1.228502) (0.457252,1.234061) (0.460689,1.239611) (0.464126,1.245152) (0.467563,1.250684) (0.471000,1.256208) (0.474437,1.261723) (0.477874,1.267231) (0.481311,1.272730) (0.484748,1.278222) (0.488185,1.283706) (0.491622,1.289182) (0.495059,1.293921) (0.498496,1.298507) (0.501933,1.303087) (0.505370,1.307661) (0.508807,1.312229) (0.512244,1.316792) (0.515681,1.321349) (0.519118,1.325901) (0.522555,1.330447) (0.525992,1.334989) (0.529429,1.339525) (0.532866,1.344056) (0.536303,1.348582) (0.539740,1.353104) (0.543177,1.357620) (0.546614,1.362132) (0.550051,1.366640) (0.553488,1.371143) (0.556925,1.375641) (0.560362,1.380135) (0.563799,1.384625) (0.567236,1.389111) (0.570673,1.393592) (0.574110,1.398070) (0.577547,1.402543) (0.580984,1.407013) (0.584421,1.411479) (0.587858,1.415941) (0.591295,1.420399) (0.594732,1.424854) (0.598169,1.429305) (0.601606,1.433753) (0.605043,1.438197) (0.608480,1.442637) (0.611917,1.447075) (0.615354,1.451509) (0.618791,1.455939) (0.622228,1.460367) (0.625665,1.464791) (0.629102,1.469212) (0.632539,1.473631) (0.635977,1.478046) (0.639414,1.482458) (0.642851,1.486867) (0.646288,1.491274) (0.649725,1.495677) (0.653162,1.500078) (0.656599,1.504476)}--(0.450377,1.222934)--cycle;
\fill[lsBand,opacity=.6] plot coordinates {(0.656599,1.504476) (0.662161,1.511588) (0.667723,1.518693) (0.673285,1.525792) (0.678847,1.532883) (0.684410,1.539969) (0.689972,1.547048) (0.695534,1.554121) (0.701096,1.561187) (0.706659,1.568249) (0.712221,1.575304) (0.717783,1.582354) (0.723345,1.589398) (0.728907,1.596437) (0.734470,1.603471) (0.740032,1.610499) (0.745594,1.616576) (0.751156,1.621915) (0.756718,1.627248) (0.762281,1.632576) (0.767843,1.637898) (0.773405,1.643214) (0.778967,1.648525) (0.784530,1.653831) (0.790092,1.659132) (0.795654,1.664428) (0.801216,1.669719) (0.806778,1.675006) (0.812341,1.680287) (0.817903,1.685564) (0.823465,1.690836) (0.829027,1.696104) (0.834589,1.701368) (0.840152,1.706627) (0.845714,1.711882) (0.851276,1.717133) (0.856838,1.722380) (0.862400,1.727623) (0.867963,1.732862) (0.873525,1.738097) (0.879087,1.743328) (0.884649,1.748556) (0.890212,1.753780) (0.895774,1.759001) (0.901336,1.764218) (0.906898,1.769432) (0.912460,1.774642) (0.918023,1.779849) (0.923585,1.785053) (0.929147,1.790254) (0.934709,1.795451) (0.940271,1.800645) (0.945834,1.805837) (0.951396,1.811025) (0.956958,1.816210) (0.962520,1.821393) (0.968083,1.826573) (0.973645,1.831750) (0.979207,1.836924) (0.984769,1.842095) (0.990331,1.847264)}--(0.656599,1.504476)--cycle;
\fill[lsBand,opacity=.6] plot coordinates {(0.990331,1.847264) (1.005658,1.861493) (1.020985,1.875704) (1.036312,1.889896) (1.051639,1.904071) (1.066966,1.918229) (1.082292,1.932371) (1.097619,1.946498) (1.112946,1.960610) (1.128273,1.974707) (1.143600,1.988791) (1.158927,2.002862) (1.174253,2.014497) (1.189580,2.024138) (1.204907,2.033771) (1.220234,2.043397) (1.235561,2.053014) (1.250888,2.062625) (1.266214,2.072228) (1.281541,2.081825) (1.296868,2.091416) (1.312195,2.101000) (1.327522,2.110577) (1.342849,2.120149) (1.358175,2.129716) (1.373502,2.139277) (1.388829,2.148832) (1.404156,2.158382) (1.419483,2.167928) (1.434810,2.177468) (1.450136,2.187004) (1.465463,2.196535) (1.480790,2.206062) (1.496117,2.215584) (1.511444,2.225102) (1.526771,2.234616) (1.542097,2.244127) (1.557424,2.253633) (1.572751,2.263136) (1.588078,2.272635) (1.603405,2.282130) (1.618732,2.291622) (1.634059,2.301111) (1.649385,2.310597) (1.664712,2.320079) (1.680039,2.329558) (1.695366,2.339034) (1.710693,2.348508) (1.726020,2.357978) (1.741346,2.367446) (1.756673,2.376911) (1.772000,2.386373) (1.787327,2.395833) (1.802654,2.405290) (1.817981,2.414745) (1.833307,2.424197) (1.848634,2.433645) (1.863961,2.443090) (1.879288,2.452532) (1.894615,2.461970) (1.909942,2.471405)}--(0.990331,1.847264)--cycle;
\fill[lsBand,opacity=.6] plot coordinates {(1.909942,2.471405) (1.932823,2.480125) (1.955704,2.488840) (1.978585,2.497551) (2.001467,2.506258) (2.024348,2.514962) (2.047229,2.523662) (2.070111,2.532359) (2.092992,2.541052) (2.115873,2.549742) (2.138754,2.558429) (2.161636,2.567113) (2.184517,2.575794) (2.207398,2.584472) (2.230280,2.593147) (2.253161,2.601820) (2.276042,2.610490) (2.298923,2.619157) (2.321805,2.627822) (2.344686,2.636485) (2.367567,2.645145) (2.390449,2.653803) (2.413330,2.662459) (2.436211,2.671113) (2.459092,2.679764) (2.481974,2.688414) (2.504855,2.697061) (2.527736,2.705707) (2.550618,2.714350) (2.573499,2.722992) (2.596380,2.731632) (2.619261,2.740271) (2.642143,2.748907) (2.665024,2.757542) (2.687905,2.766175) (2.710787,2.774807) (2.733668,2.783437) (2.756549,2.792066) (2.779430,2.800693) (2.802312,2.809319) (2.825193,2.817944) (2.848074,2.826567) (2.870955,2.835189) (2.893837,2.843809) (2.916718,2.852428) (2.939599,2.861046) (2.962481,2.869663) (2.985362,2.878278) (3.008243,2.886893) (3.031124,2.895506) (3.054006,2.904118) (3.076887,2.912729) (3.099768,2.921339) (3.122650,2.929948) (3.145531,2.938556) (3.168412,2.947163) (3.191293,2.955769) (3.214175,2.964374) (3.237056,2.972978) (3.259937,2.981581) (3.282819,2.990183)}--(1.909942,2.471405)--cycle;
\fill[lsBand,opacity=.6] plot coordinates {(3.282819,2.990183) (3.327541,2.999611) (3.372264,3.009037) (3.416987,3.018461) (3.461710,3.027884) (3.506433,3.037305) (3.551155,3.046724) (3.595878,3.056142) (3.640601,3.065558) (3.685324,3.074973) (3.730047,3.084387) (3.774769,3.093799) (3.819492,3.103210) (3.864215,3.112620) (3.908938,3.122029) (3.953661,3.131437) (3.998384,3.140843) (4.043106,3.150248) (4.087829,3.159653) (4.132552,3.169056) (4.177275,3.178458) (4.221998,3.187860) (4.266720,3.197260) (4.311443,3.206660) (4.356166,3.216059) (4.400889,3.225457) (4.445612,3.234854) (4.490334,3.244250) (4.535057,3.253646) (4.579780,3.263041) (4.624503,3.272435) (4.669226,3.281828) (4.713948,3.291221) (4.758671,3.300613) (4.803394,3.310004) (4.848117,3.319395) (4.892840,3.328785) (4.937562,3.338175) (4.982285,3.347564) (5.027008,3.356952) (5.071731,3.366340) (5.116454,3.375728) (5.161177,3.385115) (5.205899,3.394501) (5.250622,3.403887) (5.295345,3.413272) (5.340068,3.422657) (5.384791,3.432041) (5.429513,3.441425) (5.474236,3.450809) (5.518959,3.460192) (5.563682,3.469575) (5.608405,3.478957) (5.653127,3.488339) (5.697850,3.497720) (5.742573,3.507102) (5.787296,3.516482) (5.832019,3.525863) (5.876741,3.535243) (5.921464,3.544622) (5.966187,3.554002)}--(3.282819,2.990183)--cycle;
\draw[lsTeal,line width=.9pt,line join=round] plot coordinates {(0.000000,0.000000) (0.019375,0.056250) (0.038750,0.112500) (0.058125,0.168750) (0.077500,0.225000) (0.096875,0.281250) (0.116250,0.337500) (0.135625,0.393750) (0.155000,0.450000) (0.174375,0.503114) (0.193750,0.531144) (0.213125,0.559174) (0.232500,0.587204) (0.251875,0.614695) (0.271250,0.641694) (0.290625,0.668358) (0.310000,0.694669) (0.329375,0.720654) (0.348750,0.745237) (0.368125,0.768446) (0.387500,0.791453) (0.406875,0.814285) (0.426250,0.836910) (0.445625,0.859290) (0.465000,0.881451) (0.484375,0.903418) (0.503750,0.923132) (0.523125,0.941301) (0.542500,0.959354) (0.561875,0.977301) (0.581250,0.995153) (0.600625,1.012918) (0.620000,1.030606) (0.639375,1.048221) (0.658750,1.065771) (0.678125,1.083261) (0.697500,1.100696) (0.716875,1.118080) (0.736250,1.135416) (0.755625,1.149897) (0.775000,1.163005) (0.794375,1.176068) (0.813750,1.189088) (0.833125,1.202069) (0.852500,1.215013) (0.871875,1.227922) (0.891250,1.240799) (0.910625,1.253646) (0.930000,1.266464) (0.949375,1.279255) (0.968750,1.292021) (0.988125,1.304763) (1.007500,1.317483) (1.026875,1.330181) (1.046250,1.342859) (1.065625,1.355518) (1.085000,1.368158) (1.104375,1.380781) (1.123750,1.393388) (1.143125,1.405979) (1.162500,1.418556) (1.181875,1.427855) (1.201250,1.436469) (1.220625,1.445073) (1.240000,1.453669) (1.259375,1.462257) (1.278750,1.470837) (1.298125,1.479410) (1.317500,1.487976) (1.336875,1.496535) (1.356250,1.505087) (1.375625,1.513633) (1.395000,1.522173) (1.414375,1.530707) (1.433750,1.539236) (1.453125,1.547760) (1.472500,1.556278) (1.491875,1.564791) (1.511250,1.573300) (1.530625,1.581804) (1.550000,1.590303) (1.569375,1.598799) (1.588750,1.607290) (1.608125,1.615777) (1.627500,1.624260) (1.646875,1.632740) (1.666250,1.641216) (1.685625,1.649689) (1.705000,1.658158) (1.724375,1.666624) (1.743750,1.675087) (1.763125,1.683546) (1.782500,1.692003) (1.801875,1.700457) (1.821250,1.708908) (1.840625,1.717356) (1.860000,1.725800) (1.879375,1.734240) (1.898750,1.742676) (1.918125,1.749753) (1.937500,1.754973) (1.956875,1.760191) (1.976250,1.765407) (1.995625,1.770620) (2.015000,1.775833) (2.034375,1.781043) (2.053750,1.786251) (2.073125,1.791458) (2.092500,1.796663) (2.111875,1.801866) (2.131250,1.807068) (2.150625,1.812268) (2.170000,1.817467) (2.189375,1.822664) (2.208750,1.827860) (2.228125,1.833055) (2.247500,1.838248) (2.266875,1.843439) (2.286250,1.848630) (2.305625,1.853819) (2.325000,1.859007) (2.344375,1.864193) (2.363750,1.869379) (2.383125,1.874563) (2.402500,1.879746) (2.421875,1.884928) (2.441250,1.890109) (2.460625,1.895289) (2.480000,1.900468) (2.499375,1.905646) (2.518750,1.910823) (2.538125,1.915999) (2.557500,1.921174) (2.576875,1.926348) (2.596250,1.931521) (2.615625,1.936693) (2.635000,1.941865) (2.654375,1.947035) (2.673750,1.952205) (2.693125,1.957374) (2.712500,1.962542) (2.731875,1.967709) (2.751250,1.972876) (2.770625,1.978042) (2.790000,1.983207) (2.809375,1.988371) (2.828750,1.993535) (2.848125,1.998698) (2.867500,2.003860) (2.886875,2.009022) (2.906250,2.014183) (2.925625,2.019344) (2.945000,2.024503) (2.964375,2.029662) (2.983750,2.034821) (3.003125,2.039979) (3.022500,2.045136) (3.041875,2.050293) (3.061250,2.055449) (3.080625,2.060605) (3.100000,2.065760) (3.119375,2.070915) (3.138750,2.076069) (3.158125,2.081223) (3.177500,2.086376) (3.196875,2.091528) (3.216250,2.096680) (3.235625,2.101832) (3.255000,2.106983) (3.274375,2.112134) (3.293750,2.116008) (3.313125,2.118896) (3.332500,2.121784) (3.351875,2.124672) (3.371250,2.127559) (3.390625,2.130446) (3.410000,2.133333) (3.429375,2.136220) (3.448750,2.139106) (3.468125,2.141993) (3.487500,2.144879) (3.506875,2.147764) (3.526250,2.150650) (3.545625,2.153536) (3.565000,2.156421) (3.584375,2.159306) (3.603750,2.162191) (3.623125,2.165075) (3.642500,2.167960) (3.661875,2.170844) (3.681250,2.173728) (3.700625,2.176612) (3.720000,2.179496) (3.739375,2.182379) (3.758750,2.185263) (3.778125,2.188146) (3.797500,2.191029) (3.816875,2.193912) (3.836250,2.196794) (3.855625,2.199677) (3.875000,2.202559) (3.894375,2.205442) (3.913750,2.208324) (3.933125,2.211206) (3.952500,2.214087) (3.971875,2.216969) (3.991250,2.219851) (4.010625,2.222732) (4.030000,2.225613) (4.049375,2.228494) (4.068750,2.231375) (4.088125,2.234256) (4.107500,2.237137) (4.126875,2.240017) (4.146250,2.242897) (4.165625,2.245778) (4.185000,2.248658) (4.204375,2.251538) (4.223750,2.254418) (4.243125,2.257298) (4.262500,2.260177) (4.281875,2.263057) (4.301250,2.265936) (4.320625,2.268816) (4.340000,2.271695) (4.359375,2.274574) (4.378750,2.277453) (4.398125,2.280332) (4.417500,2.283210) (4.436875,2.286089) (4.456250,2.288968) (4.475625,2.291846) (4.495000,2.294724) (4.514375,2.297603) (4.533750,2.300481) (4.553125,2.303359) (4.572500,2.306237) (4.591875,2.309115) (4.611250,2.311992) (4.630625,2.314870) (4.650000,2.317748) (4.669375,2.320625) (4.688750,2.323503) (4.708125,2.326380) (4.727500,2.329257) (4.746875,2.332134) (4.766250,2.335011) (4.785625,2.337888) (4.805000,2.340765) (4.824375,2.343642) (4.843750,2.346519) (4.863125,2.349395) (4.882500,2.352272) (4.901875,2.355148) (4.921250,2.358025) (4.940625,2.360901) (4.960000,2.363777) (4.979375,2.366653) (4.998750,2.369529) (5.018125,2.372405) (5.037500,2.375281) (5.056875,2.378157) (5.076250,2.381033) (5.095625,2.383909) (5.115000,2.386784) (5.134375,2.389660) (5.153750,2.392535) (5.173125,2.395411) (5.192500,2.398286) (5.211875,2.401161) (5.231250,2.404037) (5.250625,2.406912) (5.270000,2.409787) (5.289375,2.412662) (5.308750,2.415537) (5.328125,2.418412) (5.347500,2.421287) (5.366875,2.424162) (5.386250,2.427036) (5.405625,2.429911) (5.425000,2.432786) (5.444375,2.435660) (5.463750,2.438535) (5.483125,2.441409) (5.502500,2.444284) (5.521875,2.447158) (5.541250,2.450032) (5.560625,2.452906) (5.580000,2.455781) (5.599375,2.458655) (5.618750,2.461529) (5.638125,2.464403) (5.657500,2.467277) (5.676875,2.470151) (5.696250,2.473024) (5.715625,2.475898) (5.735000,2.478772) (5.754375,2.481646) (5.773750,2.484519) (5.793125,2.487393) (5.812500,2.490267) (5.831875,2.493140) (5.851250,2.496014) (5.870625,2.498887) (5.890000,2.501760) (5.909375,2.504634) (5.928750,2.507507) (5.948125,2.510380) (5.967500,2.513158) (5.986875,2.514625) (6.006250,2.516092) (6.025625,2.517559) (6.045000,2.519026) (6.064375,2.520492) (6.083750,2.521959) (6.103125,2.523426) (6.122500,2.524893) (6.141875,2.526359) (6.161250,2.527826) (6.180625,2.529293) (6.200000,2.530759)};
\draw[lsTeal!45,line width=1.05pt,line join=round] plot coordinates {(0.000000,0.000000) (0.019375,0.079550) (0.038750,0.159099) (0.058125,0.238649) (0.077500,0.318198) (0.096875,0.397748) (0.116250,0.477297) (0.135625,0.556847) (0.155000,0.636396) (0.174375,0.711511) (0.193750,0.751152) (0.213125,0.790792) (0.232500,0.830432) (0.251875,0.869309) (0.271250,0.907493) (0.290625,0.945200) (0.310000,0.982410) (0.329375,1.019159) (0.348750,1.053924) (0.368125,1.086746) (0.387500,1.119284) (0.406875,1.151573) (0.426250,1.183570) (0.445625,1.215219) (0.465000,1.246560) (0.484375,1.277627) (0.503750,1.305506) (0.523125,1.331201) (0.542500,1.356731) (0.561875,1.382112) (0.581250,1.407359) (0.600625,1.432483) (0.620000,1.457496) (0.639375,1.482409) (0.658750,1.507228) (0.678125,1.531963) (0.697500,1.556619) (0.716875,1.581203) (0.736250,1.605721) (0.755625,1.626200) (0.775000,1.644738) (0.794375,1.663211) (0.813750,1.681625) (0.833125,1.699982) (0.852500,1.718288) (0.871875,1.736544) (0.891250,1.754755) (0.910625,1.772923) (0.930000,1.791051) (0.949375,1.809140) (0.968750,1.827194) (0.988125,1.845214) (1.007500,1.863202) (1.026875,1.881160) (1.046250,1.899089) (1.065625,1.916991) (1.085000,1.934868) (1.104375,1.952720) (1.123750,1.970549) (1.143125,1.988355) (1.162500,2.006141) (1.181875,2.019293) (1.201250,2.031474) (1.220625,2.043642) (1.240000,2.055799) (1.259375,2.067944) (1.278750,2.080078) (1.298125,2.092202) (1.317500,2.104315) (1.336875,2.116419) (1.356250,2.128514) (1.375625,2.140600) (1.395000,2.152678) (1.414375,2.164747) (1.433750,2.176809) (1.453125,2.188863) (1.472500,2.200909) (1.491875,2.212949) (1.511250,2.224982) (1.530625,2.237008) (1.550000,2.249029) (1.569375,2.261043) (1.588750,2.273051) (1.608125,2.285054) (1.627500,2.297051) (1.646875,2.309043) (1.666250,2.321030) (1.685625,2.333012) (1.705000,2.344989) (1.724375,2.356962) (1.743750,2.368930) (1.763125,2.380894) (1.782500,2.392854) (1.801875,2.404810) (1.821250,2.416761) (1.840625,2.428708) (1.860000,2.440649) (1.879375,2.452585) (1.898750,2.464516) (1.918125,2.474524) (1.937500,2.481906) (1.956875,2.489285) (1.976250,2.496662) (1.995625,2.504036) (2.015000,2.511407) (2.034375,2.518775) (2.053750,2.526141) (2.073125,2.533504) (2.092500,2.540865) (2.111875,2.548224) (2.131250,2.555580) (2.150625,2.562935) (2.170000,2.570287) (2.189375,2.577637) (2.208750,2.584985) (2.228125,2.592331) (2.247500,2.599675) (2.266875,2.607017) (2.286250,2.614357) (2.305625,2.621696) (2.325000,2.629032) (2.344375,2.636367) (2.363750,2.643701) (2.383125,2.651032) (2.402500,2.658362) (2.421875,2.665691) (2.441250,2.673018) (2.460625,2.680343) (2.480000,2.687668) (2.499375,2.694990) (2.518750,2.702311) (2.538125,2.709631) (2.557500,2.716950) (2.576875,2.724267) (2.596250,2.731583) (2.615625,2.738898) (2.635000,2.746211) (2.654375,2.753524) (2.673750,2.760835) (2.693125,2.768145) (2.712500,2.775454) (2.731875,2.782761) (2.751250,2.790068) (2.770625,2.797374) (2.790000,2.804678) (2.809375,2.811982) (2.828750,2.819284) (2.848125,2.826586) (2.867500,2.833887) (2.886875,2.841186) (2.906250,2.848485) (2.925625,2.855783) (2.945000,2.863080) (2.964375,2.870376) (2.983750,2.877671) (3.003125,2.884966) (3.022500,2.892260) (3.041875,2.899552) (3.061250,2.906844) (3.080625,2.914136) (3.100000,2.921426) (3.119375,2.928716) (3.138750,2.936005) (3.158125,2.943293) (3.177500,2.950581) (3.196875,2.957868) (3.216250,2.965154) (3.235625,2.972439) (3.255000,2.979724) (3.274375,2.987009) (3.293750,2.992487) (3.313125,2.996572) (3.332500,3.000656) (3.351875,3.004740) (3.371250,3.008823) (3.390625,3.012906) (3.410000,3.016989) (3.429375,3.021071) (3.448750,3.025153) (3.468125,3.029235) (3.487500,3.033316) (3.506875,3.037398) (3.526250,3.041479) (3.545625,3.045559) (3.565000,3.049639) (3.584375,3.053720) (3.603750,3.057799) (3.623125,3.061879) (3.642500,3.065958) (3.661875,3.070037) (3.681250,3.074116) (3.700625,3.078194) (3.720000,3.082272) (3.739375,3.086350) (3.758750,3.090428) (3.778125,3.094505) (3.797500,3.098583) (3.816875,3.102660) (3.836250,3.106736) (3.855625,3.110813) (3.875000,3.114889) (3.894375,3.118965) (3.913750,3.123041) (3.933125,3.127117) (3.952500,3.131192) (3.971875,3.135268) (3.991250,3.139343) (4.010625,3.143418) (4.030000,3.147492) (4.049375,3.151567) (4.068750,3.155641) (4.088125,3.159715) (4.107500,3.163789) (4.126875,3.167863) (4.146250,3.171936) (4.165625,3.176009) (4.185000,3.180082) (4.204375,3.184155) (4.223750,3.188228) (4.243125,3.192301) (4.262500,3.196373) (4.281875,3.200446) (4.301250,3.204518) (4.320625,3.208590) (4.340000,3.212662) (4.359375,3.216733) (4.378750,3.220805) (4.398125,3.224876) (4.417500,3.228947) (4.436875,3.233018) (4.456250,3.237089) (4.475625,3.241160) (4.495000,3.245230) (4.514375,3.249301) (4.533750,3.253371) (4.553125,3.257441) (4.572500,3.261511) (4.591875,3.265581) (4.611250,3.269651) (4.630625,3.273721) (4.650000,3.277790) (4.669375,3.281860) (4.688750,3.285929) (4.708125,3.289998) (4.727500,3.294067) (4.746875,3.298136) (4.766250,3.302204) (4.785625,3.306273) (4.805000,3.310342) (4.824375,3.314410) (4.843750,3.318478) (4.863125,3.322546) (4.882500,3.326614) (4.901875,3.330682) (4.921250,3.334750) (4.940625,3.338818) (4.960000,3.342886) (4.979375,3.346953) (4.998750,3.351020) (5.018125,3.355088) (5.037500,3.359155) (5.056875,3.363222) (5.076250,3.367289) (5.095625,3.371356) (5.115000,3.375423) (5.134375,3.379489) (5.153750,3.383556) (5.173125,3.387622) (5.192500,3.391689) (5.211875,3.395755) (5.231250,3.399821) (5.250625,3.403887) (5.270000,3.407953) (5.289375,3.412019) (5.308750,3.416085) (5.328125,3.420151) (5.347500,3.424217) (5.366875,3.428282) (5.386250,3.432348) (5.405625,3.436413) (5.425000,3.440478) (5.444375,3.444544) (5.463750,3.448609) (5.483125,3.452674) (5.502500,3.456739) (5.521875,3.460804) (5.541250,3.464869) (5.560625,3.468933) (5.580000,3.472998) (5.599375,3.477063) (5.618750,3.481127) (5.638125,3.485192) (5.657500,3.489256) (5.676875,3.493320) (5.696250,3.497385) (5.715625,3.501449) (5.735000,3.505513) (5.754375,3.509577) (5.773750,3.513641) (5.793125,3.517705) (5.812500,3.521769) (5.831875,3.525833) (5.851250,3.529896) (5.870625,3.533960) (5.890000,3.538023) (5.909375,3.542087) (5.928750,3.546150) (5.948125,3.550214) (5.967500,3.554142) (5.986875,3.556217) (6.006250,3.558291) (6.025625,3.560366) (6.045000,3.562440) (6.064375,3.564514) (6.083750,3.566589) (6.103125,3.568663) (6.122500,3.570737) (6.141875,3.572812) (6.161250,3.574886) (6.180625,3.576960) (6.200000,3.579034)};
\draw[black!35,densely dotted,line width=.5pt] (0.450377,0.000000)--(0.450377,1.222934);
\fill[lsTeal!45] (0.450377,1.222934) circle (1.4pt);
\draw[black!35,densely dotted,line width=.5pt] (0.656599,0.000000)--(0.656599,1.504476);
\fill[lsTeal!45] (0.656599,1.504476) circle (1.4pt);
\draw[black!35,densely dotted,line width=.5pt] (0.990331,0.000000)--(0.990331,1.847264);
\fill[lsTeal!45] (0.990331,1.847264) circle (1.4pt);
\draw[black!35,densely dotted,line width=.5pt] (1.909942,0.000000)--(1.909942,2.471405);
\fill[lsTeal!45] (1.909942,2.471405) circle (1.4pt);
\draw[black!35,densely dotted,line width=.5pt] (3.282819,0.000000)--(3.282819,2.990183);
\fill[lsTeal!45] (3.282819,2.990183) circle (1.4pt);
\draw[black!35,densely dotted,line width=.5pt] (5.966187,0.000000)--(5.966187,3.554002);
\fill[lsTeal!45] (5.966187,3.554002) circle (1.4pt);
\draw[black!65,dashed,line width=.7pt] (0.450377,1.222934)--(0.656599,1.504476);
\draw[lsOrng,dash dot,line width=1.05pt] (0.450377,1.016710)--(0.656599,1.255200);
\fill[lsOrng] (0.450377,1.016710) circle (1.4pt);
\fill[lsOrng] (0.656599,1.255200) circle (1.4pt);
\draw[black!65,dashed,line width=.7pt] (0.656599,1.504476)--(0.990331,1.847264);
\draw[lsOrng,dash dot,line width=1.05pt] (0.656599,1.255200)--(0.990331,1.531209);
\fill[lsOrng] (0.656599,1.255200) circle (1.4pt);
\fill[lsOrng] (0.990331,1.531209) circle (1.4pt);
\draw[black!65,dashed,line width=.7pt] (0.990331,1.847264)--(1.909942,2.471405);
\draw[lsOrng,dash dot,line width=1.05pt] (0.990331,1.531209)--(1.909942,2.043410);
\fill[lsOrng] (0.990331,1.531209) circle (1.4pt);
\fill[lsOrng] (1.909942,2.043410) circle (1.4pt);
\node[text=lsOrng,inner sep=1pt] at (1.450136,0.900000) {$j=6$};
\draw[black!65,dashed,line width=.7pt] (1.909942,2.471405)--(3.282819,2.990183);
\draw[lsOrng,dash dot,line width=1.05pt] (1.909942,1.969571)--(3.282819,2.449931);
\fill[lsOrng] (1.909942,1.969571) circle (1.4pt);
\fill[lsOrng] (3.282819,2.449931) circle (1.4pt);
\node[text=lsOrng,inner sep=1pt] at (2.596380,0.900000) {$j=7$};
\draw[black!65,dashed,line width=.7pt] (3.282819,2.990183)--(5.966187,3.554002);
\draw[lsOrng,dash dot,line width=1.05pt] (3.282819,2.342729)--(5.966187,2.888432);
\fill[lsOrng] (3.282819,2.342729) circle (1.4pt);
\fill[lsOrng] (5.966187,2.888432) circle (1.4pt);
\node[text=lsOrng,inner sep=1pt] at (4.624503,0.900000) {$j=8$};
\node[text=lsOrng,inner sep=2pt] (stockLabelFour) at (0.600000,3.150000) {$j=4$};
\node[text=lsOrng,inner sep=2pt] (stockLabelFive) at (1.900000,3.800000) {$j=5$};
\draw[black!45,line width=.35pt] (stockLabelFour.south)--(0.553488,0.900000);
\draw[black!45,line width=.35pt] (stockLabelFive.south)--(0.823465,0.900000);
\end{scope}
\draw[lsTeal,line width=.9pt] (0.2,-.75)--(0.65,-.75);
\node[anchor=west] at (0.75,-.75) {$\mathcal L_\mu(a)$ nonlinearly concave};
\draw[lsTeal!45,line width=1.05pt] (4.95,-.75)--(5.40,-.75);
\node[anchor=west] at (5.50,-.75) {$\sqrt2\,\mathcal L_\mu(a)$ nonlinearly concave};
\draw[black!65,dashed,line width=.9pt] (10.10,-.75)--(10.55,-.75);
\node[anchor=west] at (10.65,-.75) {endpoint chord};
\draw[lsOrng,dash dot,line width=1.05pt] (13.25,-.75)--(13.70,-.75);
\node[anchor=west] at (13.80,-.75) {$L_j(a)$ linear};
\end{tikzpicture}
\caption{Base-stock comparisons across penalty intervals}
\label{fig:stock-interpolation}
\begin{minipage}{0.98\linewidth}
\SingleSpaced\footnotesize\emph{Notes.} Here $\mu=5$. Dark and light solid curves show $\mathcal L_\mu(a)$ and $\sqrt2\,\mathcal L_\mu(a)$. Each dash-dot policy-cost segment lies below its dashed endpoint chord; concavity keeps the chord below the scaled lower curve. Panel (a) fixes $j=5$ on $[\underline a,\overline a]=[a_4,a_5]$. Panel (b) uses levels $4,5,6,7,8$ on consecutive intervals with endpoints $a_3,a_4,a_5,\widehat a_6,\widehat a_7,\widehat a_8$. The penalties $a_j$, $j=3,4,5$, are neighboring policy-cost intersections; $\widehat a_j=\alpha_\mu(1/e_j)$, $j=6,7,8$, equate the two lower-bound terms.
\end{minipage}
\end{figure}

\subsubsection{Two endpoint comparisons cover an interval.}
\label{sec:interpolation}\label{sec:interpolation-endpoints}
For a fixed base stock $j$, the cost $L_j(a)=H_j+ae_j$ is affine in the penalty $a$. Increasing the base-stock level lowers the loss fraction, so neighboring cost lines intersect. We will compare the lines at these intersections and at selected calibrated penalties. Write $v_j:=\operatorname{Var}(I_j)$, the variance of on-hand inventory. Appendix~\ref{app:erlang-intersections} derives the intersection penalty and common cost below. The Erlang recursion gives
\[
 \delta_j:=e_j-e_{j+1}=\frac{e_j(H_j+1)}{\mu+H_j+1}>0.
\]
The two affine costs therefore intersect at
\begin{equation}\label{eq:intersection-penalty}
 a_j:=\delta_j^{-1}-\mu=A_j(1+v_j/(H_j+1)).
\end{equation}
Figure~\ref{fig:stock-interpolation}(a) uses the intersections $a_4$ and $a_5$ as endpoints of the interval for base stock five. At $a_j$, the neighboring base-stock levels have the common cost
\[
 L_j(a_j)=L_{j+1}(a_j)=H_j+1+\frac{v_j}{H_j+1}
     =j+1-\mu+\frac\mu{H_j+1}.
\]

Figure~\ref{fig:stock-interpolation}(a) illustrates the comparison for one fixed base-stock level over a single penalty interval. Panel (b) shows how different base-stock levels cover successive intervals, linked through penalties at which both neighboring base-stock levels meet the bound. The dark solid curve shows the unscaled lower bound $\mathcal L_\mu(a)$.

\FloatBarrier

The decisive step uses any feasible base-stock level whose endpoint costs satisfy the desired comparison. Suppose base stock $j$ satisfies $L_j(\underline a)\le\sqrt2\,\mathcal L_\mu(\underline a)$ and $L_j(\overline a)\le\sqrt2\,\mathcal L_\mu(\overline a)$ at two penalties $\underline a$ and $\overline a$. For $a:=(1-q)\underline a+q\overline a$, $0\le q\le1$, the affinity of $L_j$ makes $L_j(a)$ the corresponding weighted average of $L_j(\underline a)$ and $L_j(\overline a)$. The endpoint bounds and the concavity of $\mathcal L_\mu$ then give
\begin{equation}\label{eq:fixed-stock-interpolation}
L_j(a)\le\sqrt2\bigl[(1-q)\mathcal L_\mu(\underline a)+q\mathcal L_\mu(\overline a)\bigr]\le\sqrt2\mathcal L_\mu(a).
\end{equation}
Since $U_\mu\le L_j$, this establishes the guarantee at every penalty between the two endpoints, in either order. We use this same interpolation throughout.

\subsubsection{Establishing the endpoint bounds.}\label{sec:endpoint-bounds}
Three comparisons provide the endpoints. Small base-stock levels need an initial connection to the exact no-replenishment case. For larger base-stock levels, we use intersection penalties $a_j$ when $j\le\mu$ and $\mu\ge3$, and calibrated penalties $\widehat a_j$ when $j\ge\mu$. Both endpoint comparisons apply at $j=\mu\ge3$. The following lemma records precisely what each comparison establishes.

\begin{lemma}[{\sc Endpoint Comparisons}]\label{lem:endpoints}
Fix $\mu>0$, and let $a_j$ be the intersection of adjacent base-stock cost lines in \eqref{eq:intersection-penalty}.
\begin{enumerate}[(i)]
\item Define
\[
\eta^\circ:=A_1=1+\mu^{-1},\qquad
a^\circ:=1+\frac{2-e^{-\mu}}{\mu},\qquad
g^\circ:=1+\frac{1-e^{-\mu}}{\mu}.
\]
Then $a^\circ=\alpha_\mu(\eta^\circ)$, $\mathcal L_\mu(a^\circ)\ge g^\circ$, and
\begin{equation}\label{eq:low-price-anchor}
\max\{L_1(a^\circ),L_2(a^\circ)\}<\frac54g^\circ\quad(\mu>0),
\qquad L_3(a^\circ)<\frac75g^\circ\quad(\mu\ge2).
\end{equation}
\item If $j\ge1$ is an integer and $\mu\ge\max\{j,3\}$, then
\begin{equation}\label{eq:overloaded-endpoint-main}
L_j(a_j)=L_{j+1}(a_j)\le\sqrt2\mathcal L_\mu(a_j).
\end{equation}
\item For every integer $j\ge\mu$, write $\widehat a_j:=\alpha_\mu(A_j)$, where $A_j=1/e_j$. Then
\begin{equation}\label{eq:calibrated-endpoint}
\max\{L_j(\widehat a_j),L_{j+1}(\widehat a_j)\}<\sqrt2 B_\mu(A_j)
\le\sqrt2\mathcal L_\mu(\widehat a_j).
\end{equation}
The sequence $\widehat a_j$ increases to infinity.
\end{enumerate}
\end{lemma}

\phantomsection{\bf A common starting penalty for small base-stock levels.}\label{sec:low-price-anchor} Lemma~\ref{lem:endpoints} part~(i) gives one penalty at which base stocks one and two, and base stock three when needed, all have cost strictly below $\sqrt2\,\mathcal L_\mu(a^\circ)$, because $5/4$ and $7/5$ are below $\sqrt2$. The calibrated weight $A_1$ makes ranks one and two tie in $B_\mu$. At this weight, the auxiliary cumulative supply $(1-1/\eta^\circ)t$ remains below one unit throughout $0\le t\le\mu$, so it can exceed demand only before the first arrival. This gives the explicit formulas for $a^\circ$ and $g^\circ$. The bound for base stock one at $a^\circ$, together with the exact equality at $a=1$, establishes the guarantee on $[1,a^\circ]$. The bounds for the other base-stock levels at $a^\circ$ provide endpoints for the remaining interval comparisons. Appendix~\ref{app:small-stock} gives the details.

{\bf Base-stock levels at or below mean lead-time demand.} When $j\le\mu$, we compare base stocks $j$ and $j+1$ at their common-cost penalty $a_j$. A single lower bound then covers both. Recall that $\mathcal L_\mu(a)$ is the supremum of $\min\{B_\mu(\eta),a-D_\mu(\eta)\}$ over all $\eta\ge1$, as defined in \eqref{eq:concave-lower-bound}. Thus, it suffices to find one weight $\widehat{\eta}_j\ge1$ for which the bounds $B_\mu(\widehat{\eta}_j)$ and $a_j-D_\mu(\widehat{\eta}_j)$ are both at least $L_j(a_j)/\sqrt2$. For that choice, $\mathcal L_\mu(a_j)\ge\min\{B_\mu(\widehat{\eta}_j),a_j-D_\mu(\widehat{\eta}_j)\}\ge L_j(a_j)/\sqrt2$, which proves the comparison for both base-stock levels.

Appendix~\ref{app:below-mean-comparisons} makes this choice in three steps. An inequality between mean on-hand inventory and its variance first bounds the common cost. This bound shows that the penalty at the intersection is large enough to pay for the largest stockout aging contribution. Finally, an explicit weight $\widehat{\eta}_j$ makes the bound $B_\mu(\widehat{\eta}_j)$ large enough. The Poisson-to-Gaussian inequality in Appendix~\ref{app:poisson-gaussian} substantiates this last comparison for every finite $\mu$.

\phantomsection{\bf Base-stock levels at or above mean lead-time demand.}\label{sec:above-mean-endpoints} When $j\ge\mu$, the weight $A_j=1/e_j$ makes adjacent ranks $j$ and $j+1$ tie in the lower bound. In this range, we use the calibrated penalty $\widehat a_j=\alpha_\mu(A_j)$, at which the two lower-bound terms are equal: $B_\mu(A_j)=\widehat a_j-D_\mu(A_j)$. The intersection penalties $a_j$, used in Lemma~\ref{lem:endpoints} part~(ii) for $j\le\mu$, equate the two policy costs, $L_j(a_j)=L_{j+1}(a_j)$. These penalties satisfy different equalities and need not coincide. Figure~\ref{fig:stock-interpolation}(b) shows how successive calibrated penalties $\widehat a_6,\widehat a_7,\widehat a_8$ provide endpoints for the policy comparisons. To establish Lemma~\ref{lem:endpoints} part~(iii), we show that both $L_j(\widehat a_j)$ and $L_{j+1}(\widehat a_j)$ are below $\sqrt2\,B_\mu(A_j)$, which is at most $\sqrt2\,\mathcal L_\mu(\widehat a_j)$.

To compare these policy costs with $\sqrt2\,B_\mu(A_j)$, we first bound $B_\mu(A_j)$ from below using an identity relating the two tails of the Poisson distribution. Recall that $v_j=\operatorname{Var}(N(\mu)\mid N(\mu)\le j)$, and define $v_+:=\operatorname{Var}(N(\mu)\mid N(\mu)>j)$. Appendix~\ref{app:upper-stock} proves
\begin{equation}\label{eq:underloaded-root-main}
B_\mu(A_j)\bigl(B_\mu(A_j)-j-1+\mu\bigr)-\mu=\frac{v_jv_+}{\bigl(H_j+\E[N(\mu)-j\mid N(\mu)>j]\bigr)^2}\ge0.
\end{equation}
The product of conditional variances is nonnegative, so $B_\mu(A_j)$ is at least the positive root of $x(x-j-1+\mu)=\mu$. The same appendix shows that both neighboring base-stock costs are below $\sqrt2$ times that root. Thus, at each calibrated penalty $\widehat a_j$, the costs of both neighboring base-stock levels $j$ and $j+1$ satisfy the required comparison.

Write $\widehat a_j=\alpha_\mu(A_j)$. The penalties $\widehat a_j$ increase to infinity because $A_j$ does and $\alpha_\mu$ is increasing with $\alpha_\mu(\eta)\ge \eta$. The cost $L_{j+1}(a)$ is below $\sqrt2\,\mathcal L_\mu(a)$ at both $a=\widehat a_j$ and $a=\widehat a_{j+1}$. By interpolation at this fixed base-stock level, the same comparison holds at every penalty between these endpoints. Applying this argument on each successive interval gives
\begin{equation}\label{eq:underloaded-tail}
 U_\mu(a)<\sqrt2\mathcal L_\mu(a)
 \qquad\bigl(a\ge \widehat a_S,\quad S\ge\mu,\ S\in\mathbb{N}\bigr).
\end{equation}

\subsubsection{Connecting the intervals.}
\label{sec:interval-coverage}
The low-price anchor bounds base stock one at $a^\circ$, while $L_1(1)=\mathcal L_\mu(1)=1$. Interpolation at a fixed base-stock level covers $[1,a^\circ]$.

If $0<\mu<3$, write $S:=\lceil\mu\rceil$. The anchor bounds base stock $S$ at $a^\circ$, and \eqref{eq:calibrated-endpoint} bounds its cost at $\widehat a_S$. Interpolating the same base-stock level between these penalties and using \eqref{eq:underloaded-tail} covers every $a\ge1$.

If $\mu\ge3$, write $S_0:=\lfloor\mu\rfloor$. For each $j=3,\ldots,S_0$, we have $j\le\mu$, so Lemma~\ref{lem:endpoints} part~(ii) bounds the costs $L_j(a_j)$ and $L_{j+1}(a_j)$. Interpolate base stock three between $a^\circ$ and $a_3$, base stock $j$ between $a_{j-1}$ and $a_j$ for $4\le j\le S_0$, and base stock $S_0+1$ between $a_{S_0}$ and $\widehat a_{S_0+1}$. These successive intervals form a connected set containing $a^\circ$ and $\widehat a_{S_0+1}$, regardless of their endpoint order.

Taking $S=S_0+1>\mu$ in \eqref{eq:underloaded-tail} covers every penalty $a\ge \widehat a_{S_0+1}$. In this range, we use base stock $S_0+2$ on $[\widehat a_{S_0+1},\widehat a_{S_0+2}]$, base stock $S_0+3$ on $[\widehat a_{S_0+2},\widehat a_{S_0+3}]$, and so on. Interpolation establishes the cost comparison throughout each interval, and $\widehat a_j\to\infty$ ensures that these intervals cover the entire remaining penalty range. Together with the preceding intervals and the comparison on $[1,a^\circ]$, this establishes the guarantee for every $a\ge1$.

The exact cases $a\le1$ and $\mu=0$ were established above. Each interval uses one feasible base-stock level, and overlapping intervals cover the entire penalty range. The algorithm chooses a cost-minimizing level $S^\star$, whose cost is no greater than that of any level $j$ used in these interval comparisons. Therefore,
\[
 L_{S^\star}(a)=U_\mu(a)\le L_j(a)
 \le\sqrt2\,\mathcal L_\mu(a)
 \le\sqrt2\,\operatorname{OPT}_0(\mu,a),\qquad \mu,a\ge0,
\]
where $j$ is a level justified by the interval comparisons or the exact boundary cases. This proves the $\sqrt2$ guarantee for the cost-minimizing integer base-stock policy returned by the algorithm at zero setup cost.

Appendix~\ref{app:base-stock-maps} compares the base-stock levels used in the proof with the cost-minimizing levels used by the algorithm.

\subsection{Tightness of the Base-Stock Factor}\label{sec:article-sharpness}
The worst-case sequence shows why replenishment timing matters. Base stock replaces each sale immediately, so closely spaced sales generate closely spaced deliveries one lead time later, although those sales do not predict demand at delivery. Along $\mu=a^2$, lead time grows faster than the lost-sales penalty. This is the long-lead-time setting in which regular replenishment can reduce supply variability, consistent with the asymptotic optimality of constant-order policies in periodic-review lost-sales models \citep{GoldbergEtAl2016,XinGoldberg2016} and the continuous-review constant-order-interval analysis of \citet{Xin2022}. The argument below establishes the needed comparison directly in our model.

To turn this explanation into a tightness result, we compare two bounds along $\mu=a^2$ as $a\to\infty$. Appendix~\ref{app:sharpness-proof} constructs a policy ordering one unit every $1/(1-\ell)$ time units. It serves a long-run fraction $1-\ell$ of demand and has mean inventory at most $1/(2\ell)$. Choosing the loss fraction $\ell_a:=1/\sqrt{2a}$ gives
\[
\operatorname{OPT}_0(\mu,a)\le\sqrt{2a}\qquad(a>1/2).
\]
The same appendix uses the Erlang moment identities to show that every integer base-stock level has cost at least $2\sqrt{a\omega}-\omega$ when $0<a<\mu$, where $\omega:=1-a/\mu$. Consequently,
\[
\frac{U_{a^2}(a)}{\operatorname{OPT}_0(a^2,a)}
\ge\frac{2\sqrt{a-1}-1+a^{-1}}{\sqrt{2a}}
\longrightarrow\sqrt2\qquad\text{as }a\to\infty.
\]
Together with the zero-setup upper bound, this proves \eqref{eq:article-sharp}. Along this sequence, base stock incurs cost on the scale $2\sqrt a$, while a policy with evenly spaced deliveries achieves the scale $\sqrt{2a}$. The comparison concerns actual feasible policies and the unrestricted optimum. It establishes tightness for the optimal base-stock policy returned by the algorithm at zero setup cost; the $(r,Q)$ policy selected for positive setup costs retains the upper guarantee \eqref{eq:article-full} in Theorem~\ref{thm:article-main}.

Figure~\ref{fig:base-stock-tightness} displays the proved lower bound on the approximation ratio of the algorithm's optimal base-stock policy at zero setup cost. Along $\mu=a^2$, this bound and the universal upper bound $\sqrt2$ converge to the same value. The shrinking interval therefore shows why the factor cannot be improved for ordinary base stock.

\begin{figure}[htbp]
\SingleSpaced
\centering
\includegraphics[width=.99\linewidth]{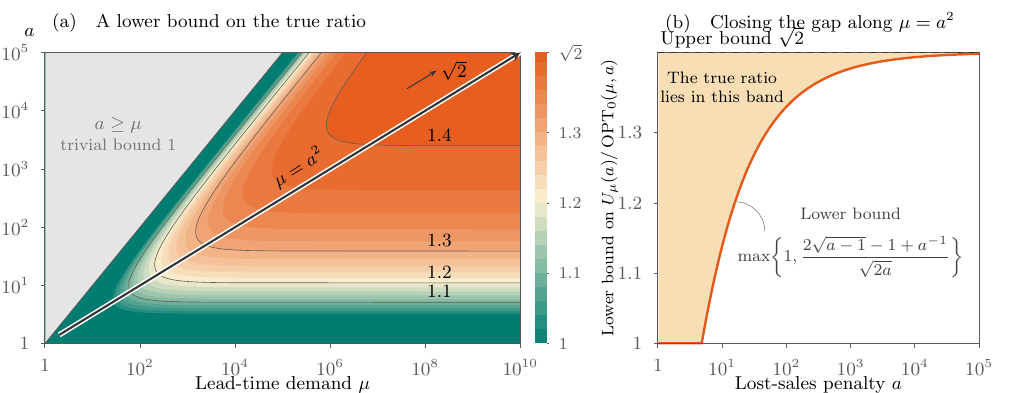}
\caption{The base-stock approximation ratio approaches its tight bound}
\label{fig:base-stock-tightness}
\begin{minipage}{0.98\linewidth}
\SingleSpaced\footnotesize\emph{Notes.} For $k=0$, panel (a) shows the analytical lower bound $\max\{1,[2\sqrt{a(1-a/\mu)}-(1-a/\mu)]/\sqrt{2a}\}$ on $U_\mu(a)/\operatorname{OPT}_0(\mu,a)$ for $1<a<\mu$; gray denotes the trivial bound $1$. Panel (b) follows $\mu=a^2$: the true ratio lies between the displayed lower bound and $\sqrt2$, which converge as $a\to\infty$.
\end{minipage}
\end{figure}

For the $(r,Q)$ policy generated by our algorithm, the worst-case factor with positive setup costs lies in $[\sqrt2,(\sqrt2+\sqrt6)/2]$. To show that the algorithm's worst-case factor is at least $\sqrt2$, fix $k=1$, so the selected lot is one, and use the sequence $\mu=a^2$ as $a\to\infty$. Appendix~\ref{app:sharpness-proof} gives the detailed proof of this lower bound. This does not establish a $\sqrt2$ lower bound for the entire $(r,Q)$ family, which can choose other lot sizes.

\subsection{Including Fixed Ordering Costs}\label{sec:batching}
When the setup cost is zero, we select a base-stock level to balance holding and lost-sales costs. A setup charge creates an additional incentive to combine orders. We now show why choosing a lot at the EOQ scale preserves a uniform guarantee.

For the full-cost comparison, suppose more generally that
\begin{equation}\label{eq:zero-setup-premise}
U_\mu(a)\le c\operatorname{OPT}_0(\mu,a).
\end{equation}
The result already proved gives $c=\sqrt2$. Writing the conversion for a general $c$ reveals how the zero-setup comparison and batching tradeoff determine the final constant.

\subsubsection{Batching the selected base-stock policy.}
Start from an integer base-stock level $S\ge1$, keep the reorder point $r=S-1$, and increase the lot size to $Q$.

Under the batched rule, inventory position cycles through $S,S+1,\ldots,S+Q-1$ at sales; deliveries transfer units from the pipeline to on-hand inventory without changing position. A coupling argument in Appendix~\ref{app:batch-upper} shows that cumulative sales under this rule are at least those of a stationary level-$S$ base-stock system over every time horizon. Batching therefore adds a controlled inventory and setup cost while preserving the baseline service performance. Recall
\[
\chi:=\sqrt{2k},\qquad Q_k:=\max\{1,\lfloor\chi\rfloor\},\qquad
b_k:=\frac{Q_k-1}{2}+\frac{k}{Q_k}.
\]
The cost calculation in that appendix proves
\begin{equation}\label{eq:batch-upper-main}
C_{\mu,a,k}(\pi_{\rm rQ}(S-1,Q_k))\le L_S(a)+b_k,
\qquad b_k\le\chi.
\end{equation}
The term $(Q_k-1)/2$ is the extra inventory allowance from traversing $Q_k$ inventory positions; $k/Q_k$ allocates the setup charge across the units in a lot. This is the same holding/setup tradeoff as in the EOQ formula. Balancing $Q/2$ and $k/Q$ gives $\sqrt{2k}=\sqrt{2K\lambda/h}$; subtracting the constant $1/2$ in our inventory allowance does not change this lot-size calculation. The chosen $Q_k$ rounds this scale down, with a minimum lot size of one. The proof establishes \eqref{eq:batch-upper-main} as a bound on total cost, without assuming that the positions are uniform in calendar time.

Selecting $S^\star$ and comparing with no replenishment therefore gives
\begin{equation}\label{eq:batch-menu-main}
C_{\rm rQ}\le\min\{a,U_\mu(a)+\chi\}
\le\min\{a,c\operatorname{OPT}_0(\mu,a)+\chi\}.
\end{equation}
The implemented rule uses the tighter integer surcharge $b_k$. When $S^\star=0$, no replenishment already achieves this bound.

\subsubsection{The batching cost faced by every policy.}
$\operatorname{OPT}_k(\mu,a)$ also includes setup costs. Every replenishment policy faces the same basic tradeoff: placing fewer, larger orders saves setup charges but leaves more units waiting for demand. We use a fraction $\theta\in[0,1]$ of holding cost, together with all setup costs, to quantify this tradeoff. The remaining fraction $1-\theta$ will be used in the zero-setup comparison established earlier.

Write $b:=\sqrt{2k\theta}$. Appendix~\ref{app:batch-lower} proves that, for every empty-start admissible policy and every horizon $T$,
\begin{equation}\label{eq:batch-credit-main}
\theta\,\E\int_0^T X(t)\,dt+k\E O(T)
\ge\sqrt{2k\theta}\,\E\,\mathcal A(T).
\end{equation}
The left side is the fraction $\theta$ of expected holding cost plus expected setup cost. The right side is $b$ times the expected number of units served. This bound allows arbitrary order quantities and timing, including decisions based on past demand; it applies beyond the classical policy family.

We now add the remaining holding cost and all lost-sales costs. Every demand is either served or lost, so $b\,\E\,\mathcal A(T)=bT-b\,\E\mathcal U(T)$. Consequently, the lower bound on total expected cost is $bT$ plus the remaining holding cost and lost-sales cost with penalty $a-b$. The remaining holding-cost coefficient is $1-\theta$. For $\theta<1$ and $a\ge b$, these remaining costs form a zero-setup problem. Dividing them by $1-\theta$ restores holding-cost coefficient one and changes the penalty to $(a-b)/(1-\theta)$. Comparing these costs with the zero-setup optimal cost gives
\begin{equation}\label{eq:operating-transform-main}
\operatorname{OPT}_k(\mu,a)
\ge b+(1-\theta)\operatorname{OPT}_0\left(\mu,\frac{a-b}{1-\theta}\right).
\end{equation}

At $\theta=1$, all holding and setup costs are used in \eqref{eq:batch-credit-main}, with $b=\sqrt{2k}$. These costs are at least $b$ times expected served demand, and lost-sales cost is $a$ times expected lost demand. Since every demand is served or lost, $\operatorname{OPT}_k(\mu,a)\ge\min\{a,\sqrt{2k}\}$. Appendix~\ref{app:batch-lower} gives the finite-horizon details, and the initial-supply argument extends the long-run bounds to the full admissible class.

\subsubsection{Matching the two cost bounds.}
\begin{proposition}[{\sc From Zero to General Setup Cost}]\label{prop:full-conversion}
If \eqref{eq:zero-setup-premise} holds with $c>0$, then
\begin{equation}\label{eq:batching-conversion}
 C_{\rm rQ}\le R\operatorname{OPT}_k(\mu,a),\qquad
 R:=\frac{c+\sqrt{c^2+4}}2.
\end{equation}
\end{proposition}

\proof{Proof.}
The lower bound uses $b=\sqrt{2k\theta}$ times expected served demand. To match the upper bound $(\sqrt{2k}+c\,\operatorname{OPT}_0(\mu,a))/R$, we want $b=\sqrt{2k}/R$ and $1-\theta=c/R$. These requirements give $R^2=cR+1$ and determine the split $\theta_c:=R^{-2}$, explaining the stated choice of $R>1$. If $a\le R\sqrt{2k}$, take $\theta=1$ and $b=\sqrt{2k}$; then $\operatorname{OPT}_k\ge\min\{a,\sqrt{2k}\}\ge a/R$. If $a>R\sqrt{2k}$, take $\theta=\theta_c$ and $b=\sqrt{2k}/R$ in \eqref{eq:operating-transform-main}. Since
\[
 \frac{a-\sqrt{2k}/R}{1-\theta_c}\ge a,\qquad
 1-\theta_c=c/R,
\]
monotonicity in the loss penalty gives $\operatorname{OPT}_k\ge(\sqrt{2k}+c\,\operatorname{OPT}_0(\mu,a))/R$. Combining the two cases with \eqref{eq:batch-menu-main} proves the claim. No division by $a$, $k$, or $\operatorname{OPT}_0(\mu,a)$ is required, so the proof includes their zero boundaries. At $c=\sqrt2$, $R=(\sqrt2+\sqrt6)/2<1.932$.
\Halmos
\endproof

For this coefficient-matching argument, the best split minimizes $\max\{\theta^{-1/2},c/(1-\theta)\}$ over $0<\theta<1$. The first term decreases and the second increases, so their unique intersection gives the minimizing split $\theta_c=R^{-2}$, with $R^2=cR+1$. A different fixed split cannot improve this conversion factor; a stronger lower bound or a different policy comparison may do so.

\subsubsection{Zero cost parameters and zero lead time.}\label{sec:zero-boundaries}
If $\lambda=0$ or $p=0$, the empty no-order policy has zero cost. If $h=0$ and $\lambda,p>0$, the physical depletion policy $(0,Q)$ serves $Q$ demands and then waits one lead time; its average cost is
\[
 C(\pi_{\rm rQ}(0,Q))
 =\frac{K\lambda+p\lambda^2L}{Q+\lambda L}\longrightarrow0,\qquad Q\to\infty.
\]
Thus, the optimum and the classical-family infimum are both zero, although these infima may not be attained by any finite lot size. When $K=0$, base stock $S$ has cost $p\lambda e_S(\lambda L)\to0$ as $S\to\infty$, so the base-stock infimum is also zero. For $L=0$ and $h,\lambda>0$, replenishment is immediate, and the zero-setup optimal cost is $\operatorname{OPT}_0(0,a)=\min\{a,1\}$. The batching bounds and the conversion argument remain valid at $\mu=0$, so the guarantee in \eqref{eq:batching-conversion} also covers this case with general setup costs.

\proof{Proof of Theorem~\ref{thm:article-main} and Corollary~\ref{cor:policy-selection}.}
Section~\ref{sec:comparison-plan} proves $U_\mu\le\sqrt2\mathcal L_\mu\le\sqrt2\operatorname{OPT}_0$. Section~\ref{sec:article-sharpness} provides a matching sequence for base stock. Proposition~\ref{prop:full-conversion}, with $c=\sqrt2$, proves the full-cost bound for the selected policy because its cost is at most $\min\{a,U_\mu(a)+b_k\}\le\min\{a,U_\mu(a)+\sqrt{2k}\}$. When $K=0$, $Q_k=1$, $b_k=0$, and $U_\mu(a)\le L_0(a)=a$. The selected policy is therefore a base-stock policy with cost at most $\min\{a,U_\mu(a)\}=U_\mu(a)$, so it attains the minimum integer base-stock cost and has the tight factor $\sqrt2$. Multiplying normalized costs by $h$ proves the primitive statements for $h,\lambda>0$. The zero-parameter cases just established complete the theorem.\Halmos
\endproof

\section{Numerical Experiments}\label{sec:numerics}

We implement Corollary~\ref{cor:policy-selection} by computing $S^\star$ through the following finite search. Start with $e_0=1$ and $L_0(a)=a$. For $S=1,2,\ldots$, calculate
\[
e_S=\frac{\mu e_{S-1}}{S+\mu e_{S-1}},\qquad
L_S(a)=S-\mu+(\mu+a)e_S.
\]
Retain the smallest level $S^\star$ attaining the lowest cost found. Since $L_S(a)\ge S-\mu$, stop once $S-\mu$ exceeds that cost: no larger level can improve it. This gives $U_\mu(a)=L_{S^\star}(a)$; the remaining policy choices follow Corollary~\ref{cor:policy-selection}.

A small refinement preserves the guarantee: whenever the replenishment candidate's exact cost is available, compare that cost with $a$ instead of its upper bound. Appendix~\ref{app:numerical-search} gives exact formulas when $Q=1$, $r<Q$, or $\mu=0$. This refinement can select replenishment at additional instances and cannot increase cost. The experiments below retain the stated upper-bound rule so that they evaluate Corollary~\ref{cor:policy-selection} directly.

\begin{figure}[!htbp]
\SingleSpaced
\centering
\includegraphics[width=.99\linewidth]{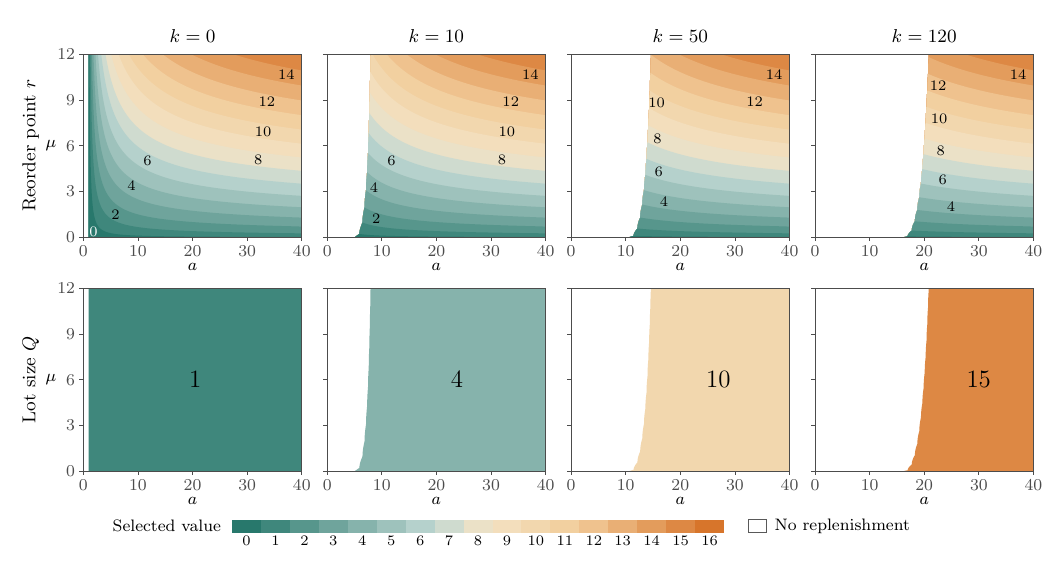}
\caption{Selected reorder points and lot sizes across setup costs}
\label{fig:rq-parameter-maps}
\begin{minipage}{0.98\linewidth}
\SingleSpaced\footnotesize\emph{Notes.} Columns use $k=0,10,50,120$. Shading gives the selected $r=S^\star-1$ (top) and $Q_k$ (bottom). Colored regions labeled $r=0$ use replenishment; white regions indicate no replenishment. Ties select the smallest minimizing base-stock level, and $U_\mu(a)+b_k=a$ selects no replenishment. Boundaries use the Erlang recursion with interpolation between sampled $\mu$ values.
\end{minipage}
\end{figure}

Figure~\ref{fig:rq-parameter-maps} shows the reorder points and lot sizes selected by the algorithm over normalized lead-time demand $\mu$ and lost-sales penalty $a$, for $k=0,10,50,120$. Each column fixes setup cost, with the selected reorder point $r$ above and lot size $Q$ below. At $k=0$, the algorithm selects the cost-minimizing integer base-stock policy, with $r=S^\star-1$ and $Q=1$ wherever replenishment is chosen.

The upper panels agree wherever they all select replenishment, because the rule uses the same $S^\star$ for every $k$. Setup cost changes the order quantity, which is $1,4,10,15$ in the four columns. Each lower panel therefore has one lot size throughout its replenishment region. The surcharges are $b_k=0,4,9.5,15$. Across these four setup costs, the white region expands because a larger surcharge makes the replenishment cost upper bound reach or exceed $a$ at more parameter values.

For example, at $(\mu,a)=(5,20)$, all four columns retain $r=7$ and use the four lot sizes above. The zero-setup cost is about $4.7512$, giving cost bounds $4.7512$, $8.7512$, $14.2512$, and $19.7512$. The first is exact; the other three are upper bounds. At $(\mu,a)=(12,20)$, the $k=120$ column instead selects no replenishment. These choices illustrate how the algorithm combines the zero-setup reorder point with batching and the option of losing every demand. The no-replenishment region is determined by the cost upper bound; computing the candidate's exact cost could justify replenishment at additional parameter values.

To assess the $1.932$ guarantee for the classical policy family, we now minimize full cost jointly over $r$ and $Q$ at each numerical instance. We compare this choice with the heuristic, including no replenishment in both cases. Write $(r^\star,Q^\star)$ for the pair selected by the numerical search. Holding-cost bounds give finite ranges containing every integer pair that can improve on no replenishment. We evaluate policy costs exactly when $\mu=0$, $Q=1$, or $r<Q$, and use simulation for the remaining pairs. The reported simulation costs are evaluated independently after selection. Appendix~\ref{app:numerical-search} gives the cost formulas, search bounds, and accuracy checks.

The optimal system cost over all admissible policies remains unknown. We compare both policies with the same lower bound on $\operatorname{OPT}_k(\mu,a)$. Combining \eqref{eq:operating-transform-main} with Proposition~\ref{prop:rank-lower} gives
\[
\operatorname{OPT}_k(\mu,a)\ge\min\bigl\{\sqrt{2k\theta}+(1-\theta)B_\mu(\eta),\ a-(1-\theta)D_\mu(\eta)\bigr\}
\]
for $\eta\ge1$, $0\le\theta<1$, and $\sqrt{2k\theta}\le a$. We maximize this bound over $\theta$ and $4{,}097$ logarithmically spaced weights from $1$ to $10^4$. We define $\underline C_k(\mu,a)$ as the larger of this maximized bound and the integer batching bound $\min\{a,\min_{Q\in\mathbb N}[(Q+1)/2+k/Q]\}$ derived in Appendix~\ref{app:numerical-search}. For either policy $\pi$ and $a>0$,
\[
\frac{C_{\mu,a,k}(\pi)}{\operatorname{OPT}_k(\mu,a)}-1
\le\frac{C_{\mu,a,k}(\pi)}{\underline C_k(\mu,a)}-1.
\]
Thus, evaluating the right side estimates an upper bound on the policy's relative optimality gap. Figure~\ref{fig:performance-overview} highlights the heuristic $(r,Q)$ policy at $k=10$ and summarizes the largest evaluated ratios for $k=1,10,50,120$. Figure~\ref{fig:positive-setup-gap} in Appendix~\ref{app:numerical-search} gives all eight contour panels, with the heuristic above and the numerically optimized policy below. The search uses $a=0,1,\ldots,100$ and $\mu\in\{0,0.25,0.5,0.75,1,2,\ldots,50\}$.

The largest estimated cost-to-lower-bound ratio on the grid is about $1.59$ for the heuristic and $1.47$ after joint optimization of $r$ and $Q$, compared with the uniform factor $1.932$. We also compare the two policy costs directly. Figure~\ref{fig:heuristic-extra-cost-distribution} shows the distribution of the heuristic's extra cost as a percentage of the numerically optimized cost. Across the $5{,}400$ grid points with $a>0$ for each setup cost, the median extra cost of the heuristic $(r,Q)$ policy relative to the numerically optimized cost is $3.00\%$, $2.33\%$, $3.07\%$, and $3.91\%$ for $k=1,10,50,120$, respectively.

\begin{figure}[!htbp]
\SingleSpaced
\centering
\includegraphics[width=.99\linewidth]{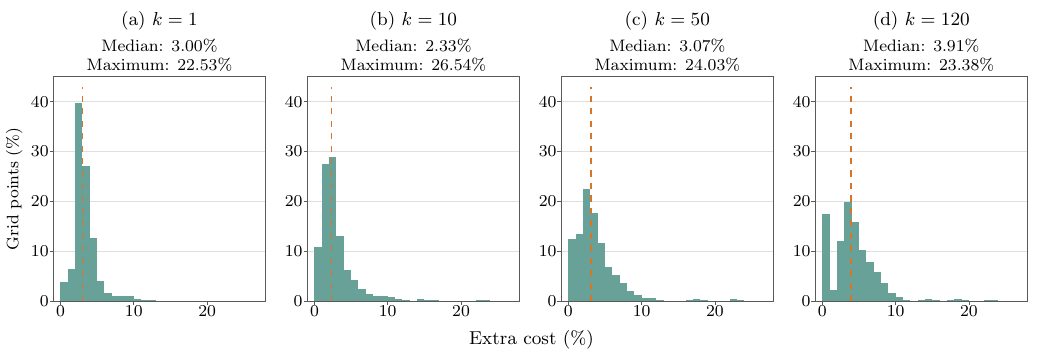}
\caption{Distribution of extra cost from using the heuristic $(r,Q)$ policy}
\label{fig:heuristic-extra-cost-distribution}
\begin{minipage}{0.98\linewidth}
\SingleSpaced\footnotesize\emph{Notes.} Extra cost is the heuristic's cost increase over numerical joint optimization, divided by the optimized cost, in percent. Each panel weights $5{,}400$ grid points with $a>0$ equally; $a=0$ is excluded because both costs vanish. Bars show percentages in bins one percentage point wide; dashed lines mark medians. Costs are exact where available and independently simulated otherwise.
\end{minipage}
\end{figure}

The largest increase is $26.54\%$, at $(\mu,a,k)=(14,8,10)$: the algorithm selects no replenishment at cost $8$, while the joint search selects $(r^\star,Q^\star)=(9,5)$ at estimated cost $6.322$. Equivalently, joint optimization saves $21.0\%$ of the heuristic $(r,Q)$ policy's cost. The algorithm's rejected candidate is $(12,4)$, with cost upper bound about $8.0205$. Because $r\ge Q$ for this candidate, the simple exact-cost refinement above does not resolve this particular case.

Joint optimization of $r$ and $Q$ also reduces cost when both policies replenish. At $(\mu,a,k)=(5,20,50)$, the heuristic chooses $(r,Q)=(7,10)$, while the joint search chooses $(r^\star,Q^\star)=(4,11)$. Their costs are approximately $12.936$ and $12.024$, respectively. Thus, the heuristic costs $7.58\%$ more. Both costs are exact renewal calculations because $r<Q$.\footnote{\label{fn:single-outstanding-order}When an order is triggered, inventory position is $r<Q$, so no earlier lot of $Q$ units can still be pending. Each cycle therefore starts with $r$ units on hand and one lot due after $\mu$ time units. Independent future Poisson demand makes successive order cycles independent with the same distribution. Average cost equals expected cycle cost divided by the expected cycle length, both computable from Poisson probabilities; see Appendix~\ref{app:numerical-search}.} Joint optimization lowers the reorder point and increases the lot size slightly, changing both decisions to reflect setup cost.

The numerical cost ratios are below $1.932$ throughout the sampled grid. Determining the tight worst-case factor with positive setup costs requires bounding the gap between the lower bound and the optimal system cost and considering parameters outside this range.

\section{Concluding Remarks}\label{sec:conclusion}\label{sec:article-conclusion}

An inventory-triggered $(r,Q)$ policy with a fixed integer lot suffices for a finite uniform guarantee in the continuous-review lost-sales model without disposal. Including no replenishment, this family has infimum cost at most $(\sqrt2+\sqrt6)/2<1.932$ times the optimum over all admissible nonanticipative policies. When holding cost and demand rate are positive, our algorithm selects a finite classical policy or no replenishment that meets the full-cost guarantee. At zero setup cost, it returns an optimal integer base-stock policy with the tight worst-case factor $\sqrt2$.

The algorithm separates inventory protection from the cost of placing orders. It first chooses a cost-minimizing base-stock level without setup charges. When this level is positive, it retains the reorder point and uses a lot at the EOQ scale to spread setup costs while preserving at least the baseline service performance. The rule selects no replenishment when its replenishment cost upper bound reaches the cost of losing every demand. Joint optimization can adjust both parameters to the setup cost: across the sampled grid in Section~\ref{sec:numerics}, the median extra cost of the heuristic $(r,Q)$ policy relative to numerical joint optimization ranges from $2.33\%$ to $3.91\%$ for the four setup costs.

Two lower bounds connect these policy decisions to costs that every admissible policy must incur. The ranks and remaining delivery times of units on hand or on order describe supply committed before future demand is known. Tracking these units gives a lower bound on zero-setup cost that is concave in the lost-sales penalty. This concavity turns comparisons at selected penalties into guarantees over the intervals between them. A second lower bound captures the tradeoff between holding inventory and ordering frequently. Matching this bound with the cost of batching the base-stock policy converts the zero-setup factor into the full-cost guarantee.

The tight worst-case factor for the $(r,Q)$ family with positive setup costs remains open. The base-stock example does not establish a $\sqrt2$ lower bound for the $(r,Q)$ policy family: even at $K=0$, allowing $Q>1$ may improve its worst-case factor. What is sharp is the base-stock coefficient used in our conversion, and the holding-cost split is optimal for that coefficient-matching argument. Improving the $1.932$ guarantee therefore calls for additional structure, such as a direct joint analysis of the reorder point and lot size. The computable lower bound can help identify instances on which such improvements matter.

Unit Poisson demand gives the exact base-stock cost formulas and the identities used to track committed supply. A deterministic lead time fixes each ordered unit's delivery date and preserves the order of deliveries. Excluding disposal ensures that every delivered unit remains in inventory until it serves demand. Non-Poisson demand, random lead times, and inventory disposal are natural extensions, each requiring new lower bounds or policy comparisons.

\label{end:manuscript}\clearpage
\renewcommand\normalsizeXIyesCE{\OriginalManuscriptNormalSize\fontsize{11}{19.71}\selectfont\abovedisplayskip=7pt plus 2pt minus 2pt\belowdisplayskip=7pt plus 2pt minus 2pt\abovedisplayshortskip=3pt plus 2pt minus 1pt\belowdisplayshortskip=5pt plus 2pt minus 1pt}
\normalsize\raggedbottom
\predisplaypenalty=10000 %
\setcounter{page}{1}
\renewcommand\thepage{A\arabic{page}}
\renewcommand\thesection{\Alph{section}}
\renewcommand\thefigure{\Alph{section}.\arabic{figure}}\setcounter{figure}{0}
\renewcommand\theequation{\Alph{section}.\arabic{equation}}
\renewcommand\thedefinition{\Alph{section}.\arabic{definition}}
\renewcommand\thelemma{\Alph{section}.\arabic{lemma}}
\renewcommand\theproposition{\Alph{section}.\arabic{proposition}}
\renewcommand\thetheorem{\Alph{section}.\arabic{theorem}}
\renewcommand\theremark{\Alph{section}.\arabic{remark}}
\makeatletter
\@addtoreset{figure}{section}\@addtoreset{equation}{section}\@addtoreset{definition}{section}\@addtoreset{lemma}{section}\@addtoreset{proposition}{section}\@addtoreset{theorem}{section}\@addtoreset{remark}{section}
\makeatother
\setcounter{section}{0}\setcounter{equation}{0}\setcounter{definition}{0}\setcounter{lemma}{0}\setcounter{proposition}{0}\setcounter{theorem}{0}\setcounter{remark}{0}
\setlength{\jot}{2pt}
\begin{center}
\textbf{\large Online Appendices to\\``A 1.932 Guarantee for $(r,Q)$ Policies\\in Continuous-Review Lost-Sales Inventory Systems''}
\end{center}
\noindent These online appendices give the full arguments behind the base-stock cost formulas, the all-policy lower bound, the $\sqrt2$ zero-setup guarantee, the tightness construction, and the batching conversion. We recall the principal notation before the proofs.

\noindent The table summarizes the quantities used throughout the appendix.
\begin{center}
\fontsize{8}{12}\selectfont
\begin{tabular}{>{\raggedright\arraybackslash}p{.26\linewidth}>{\raggedright\arraybackslash}p{.65\linewidth}}
\toprule
Quantity & Meaning and where it is established\\
\midrule
\(\mu,a,k\) & Normalized lead-time demand, loss penalty, and setup cost; Section~\ref{sec:model}.\\
\(p_j,F_j,G_l\) & Poisson mass, cumulative probability, and expected excess; Section~\ref{sec:frontiers}.\\
\(e_j,A_j,H_j,v_j\) & Level-$j$ base-stock loss fraction, Erlang normalizer, mean on-hand inventory, and its variance; Section~\ref{sec:frontiers} and Appendix~\ref{app:erlang-details}.\\
\(L_j(a),U_\mu(a)\) & Level-$j$ base-stock zero-setup cost and its minimum over integer base-stock levels; Section~\ref{sec:frontiers}.\\
\(B_\mu(\eta),D_\mu(\eta),\mathcal L_\mu(a)\) & Rank-based zero-setup lower-bound quantities, with weight \(\eta\ge1\); Section~\ref{sec:lower}.\\
\(C_{\mu,a,k}(\pi),\operatorname{OPT}_k(\mu,a)\) & Normalized policy cost and unrestricted optimal cost; Section~\ref{sec:model}.\\
\(S^\star,Q_k,b_k\) & Selected base-stock level, integer lot size, and batching surcharge; Section~\ref{sec:batching}.\\
\bottomrule
\end{tabular}
\end{center}

\section{Base-Stock Costs and Moments}\label{app:erlang-details}
This appendix derives the exact costs and moment bounds used to compare neighboring base-stock levels. Throughout, the base-stock level is an integer and $\mu>0$.

A level-\(j\) unit-lot base-stock system is an \(\mathrm M/\mathrm D/j/j\) loss system: a served demand occupies a position until its replacement arrives \(\mu\) time units later. The Erlang stationary distribution follows from \citet[Theorem 1, p.~241]{Zachary2007}, with its Erlang specialization on p.~240: the occupied-position count has a Poisson\((\mu)\) law conditioned to be at most \(j\). Poisson arrivals see the stationary full-system probability by \citet[Theorem 1, p.~225, and equation (3), p.~226]{Wolff1982}. The same expected average costs hold from the stated initialization. Indeed, couple the canonical and stationary level-$j$ systems using the same future demands. For $\mu>0$, a demand-free interval of length $\mu$ leaves both systems with all $j$ units on hand and no pipeline, so they coalesce. Disjoint intervals succeed independently with probability $e^{-\mu}$, giving a coupling time with finite mean. Inventory is bounded by $j$, and Lemma~\ref{lem:poisson-compensation}, applied up to the coupling time to the difference of the bounded predictable service indicators, bounds the expected difference in sales by the mean coupling time. Thus, their expected average holding and loss costs agree. At $\mu=0$, every positive base-stock policy serves every demand; level zero is the no-order policy.

Let $I_j$ denote stationary on-hand inventory and write $v_j:=\operatorname{Var}(I_j)$. To obtain its moments, let $K_j$ have the law of $N(\mu)$ conditional on $N(\mu)\le j$, so $I_j=j-K_j$. The identities $n p_n=\mu p_{n-1}$ and $n(n-1)p_n=\mu^2p_{n-2}$ give
\[
 \E K_j=\mu\frac{F_{j-1}}{F_j}=\mu(1-e_j),\qquad
 \E[K_j(K_j-1)]=\mu^2\frac{F_{j-2}}{F_j}
 =\mu^2-\mu(\mu+j)e_j.
\]
Here $F_{j-1}/F_j=1-e_j$ and $F_{j-2}/F_j=1-e_j-p_{j-1}/F_j=1-(1+j/\mu)e_j$. Thus, $\E I_j=j-\mu+\mu e_j=H_j$, and its variance is
\begin{equation}\label{eq:erlang-variance}
 v_j=\mu^2-\mu(\mu+j)e_j+\mu(1-e_j)-\mu^2(1-e_j)^2
 =\mu\{1-e_j(j+1-\mu+\mu e_j)\}
 =\mu\{1-e_j(H_j+1)\}.
\end{equation}
\subsection{Moment Identities and a Variance Bound}\label{app:erlang-moments}
The parameter derivatives below hold with the integer base-stock level $j$ fixed. The Poisson identities $p_j'=p_{j-1}-p_j$ and $F_j'=-p_j$ give ${e_j'/e_j}={j/\mu}-1+e_j={H_j/\mu}, H_j'=-1+e_j+\mu e_j'= -{v_j/\mu}$. Also $p_{j+1}=\mu p_j/(j+1)$ and $F_{j+1}=F_j+p_{j+1}$ yield
\begin{equation}\label{eq:erlang-derivatives}
 e_{j+1}=\frac{\mu e_j}{j+1+\mu e_j}
        =\frac{\mu e_j}{\mu+H_j+1},\qquad
 e_j'=\frac{e_jH_j}{\mu},\qquad H_j'=-\frac{v_j}{\mu}.
\end{equation}
The recursion makes $e_j$ strictly decreasing in $j$, since $H_j\ge0$. Moreover $A_j\ge p_0/p_j=j!/\mu^j\to\infty$.

The stationary on-hand inventory $I_j=j-K_j$, where $K_j$ has the law of $N(\mu)$ conditional on $N(\mu)\le j$, satisfies
\begin{equation}\label{eq:erlang-moment-bound}
0<v_j\le H_j(H_j+1)\qquad(j\ge1, \mu>0).
\end{equation}
We prove this bound using the point and tail probabilities of $I_j$. The ratios $\Prb(I_j=l+1)/\Prb(I_j=l)=(j-l)/\mu$ decrease in $l$, so $\Prb(I_j\ge l)/\Prb(I_j=l)$, a sum of products of such ratios, decreases in $l$. Hence, $\Prb(I_j\ge l+1)/\Prb(I_j\ge l)=1-\Prb(I_j=l)/\Prb(I_j\ge l)$ decreases in $l$, which gives $\Prb(I_j\ge l+s)\le\Prb(I_j\ge l)\Prb(I_j\ge s)$. Therefore,
\[
\frac12\E[I_j(I_j-1)]=\sum_{l,s\ge1}\Prb(I_j\ge l+s)
\le\Bigl(\sum_{l\ge1}\Prb(I_j\ge l)\Bigr)^2=H_j^2.
\]
Adding $H_j$ and subtracting $H_j^2$ gives \eqref{eq:erlang-moment-bound}. For $j=0$ the same bound holds with $H_0=v_0=0$.

\subsection{Intersection Penalties and Common Costs}\label{app:erlang-intersections}
We derive the intersection penalty in \eqref{eq:intersection-penalty} and the corresponding common cost. Fix an integer $j\ge0$, and recall that $A_j=1/e_j$ and $\delta_j=e_j-e_{j+1}$. The recursion in \eqref{eq:erlang-derivatives} gives $\delta_j=\frac{e_j(H_j+1)}{\mu+H_j+1}>0$. The identity $H_j=j-\mu+\mu e_j$ gives $H_{j+1}-H_j=1-\mu\delta_j$. Raising the base-stock level from $j$ to $j+1$ therefore increases mean holding cost by $1-\mu\delta_j$ and reduces the lost-sales cost by $a\delta_j$. Thus, $L_{j+1}(a)-L_j(a)=1-(\mu+a)\delta_j$. Setting this difference to zero gives the unique intersection penalty $a_j=\delta_j^{-1}-\mu$. To rewrite it in terms of the variance, multiply \eqref{eq:erlang-variance} by $A_j$ and use $A_je_j=1$ to obtain $A_jv_j=\mu(A_j-H_j-1)$. Combining this identity with the formula for $\delta_j$ gives \eqref{eq:intersection-penalty}:
\[
\begin{aligned}
 a_j&=\delta_j^{-1}-\mu
 =\frac{A_j(\mu+H_j+1)}{H_j+1}-\mu=A_j+\frac{\mu(A_j-H_j-1)}{H_j+1}
 =A_j\left(1+\frac{v_j}{H_j+1}\right).
\end{aligned}
\]

The equality $L_j(a_j)=L_{j+1}(a_j)$ holds because $a_j$ is their intersection penalty. Multiplying \eqref{eq:intersection-penalty} by $e_j$ and using $A_je_j=1$ gives $a_je_j=1+v_j/(H_j+1)$. Substitution into $L_j(a_j)=H_j+a_je_j$ yields $L_j(a_j)=L_{j+1}(a_j)=H_j+1+{v_j/(H_j+1)}$. For the second expression for the common cost, the variance identity in \eqref{eq:erlang-variance} gives ${v_j/(H_j+1)}={\mu/(H_j+1)}-\mu e_j$. Since $H_j=j-\mu+\mu e_j$, the terms $\mu e_j$ cancel: $H_j+1+{v_j/(H_j+1)} =j-\mu+\mu e_j+1+{\mu/(H_j+1)}-\mu e_j =j+1-\mu+{\mu/(H_j+1)}$. Thus, both expressions describe the same cost at $a_j$.

\subsection{Base-Stock Choices in the Proof and the Algorithm}\label{app:base-stock-maps}

At zero setup cost, the normalized model depends only on lead-time demand $\mu$ and the lost-sales penalty $a$. Figure~\ref{fig:base-stock-map} summarizes the base-stock choices over these two parameters. Panel (a) shows the base-stock levels used in Section~\ref{sec:interval-coverage} to establish the $\sqrt2$ guarantee over successive penalty intervals, choosing the smallest level when the intervals overlap. Panel (b) shows the cost-minimizing level $S^\star$ selected by the algorithm. The marked point $(\mu,a)=(5,10)$ illustrates why the two maps can differ: the interval comparison uses level $6$, with cost about $3.878$, while the minimizing rule selects level $7$, with cost about $3.808$. The algorithm's optimal base-stock policy satisfies the $\sqrt2$ guarantee because its cost is no greater than that of the policy used in the proof.

\begin{figure}[!htbp]
\SingleSpaced
\centering
\includegraphics[width=.98\linewidth]{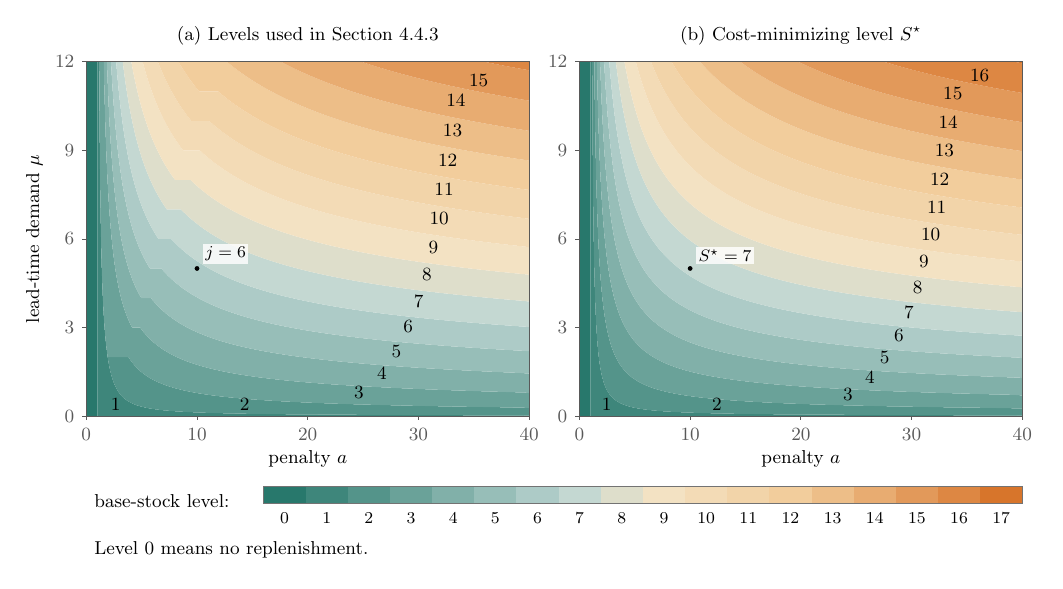}
\caption{Base-stock choices in the proof and in the algorithm}
\label{fig:base-stock-map}
\begin{minipage}{0.98\linewidth}
\SingleSpaced\footnotesize\emph{Notes.} At $k=0$, shading gives the base-stock level; zero means no replenishment. Panel (a) selects the smallest level from the interval comparisons in Section~\ref{sec:interval-coverage}; panel (b) minimizes $L_S(a)$, choosing the smaller level at ties. At the marked point $(\mu,a)=(5,10)$, the levels are $6$ and $7$, respectively.
\end{minipage}
\end{figure}

\section{Initial Supply and the Committed-Supply Bound}\label{app:rank-proof}
\subsection{Removing Finite Initial Supply}\label{app:initial-supply}

For lower bounds it suffices to consider policies starting physically empty. To see this, simulate an admissible policy $\pi$ with its original initialization, and let an empty physical copy place the same new orders at the same epochs. The simulation uses only past demand and independent policy randomness, so the copy is nonanticipative. On a common demand path, deletion of the $Y_\pi(0)<\infty$ initial units, where $Y_\pi(0)$ is the initial inventory position (units on hand or pending at time zero), gives
\[
 X^\circ(t)\le X_\pi(t),\qquad O^\circ(T)=O_\pi(T),\qquad
 0\le \mathcal A_\pi(T)-\mathcal A^\circ(T)\le\min\{Y_\pi(0),N(T)\}.
\]
For every integer $n\ge1$, independence of initial supply and future demand implies ${\E\min\{Y_\pi(0),N(T)\}}/T \le{n/T}+\Prb(Y_\pi(0)>n)$. Letting $T\to\infty$ and then $n\to\infty$ shows that the expected lost service is $o(T)$. Holding and order costs do not increase, so $C_{\mu,a,k}(\pi^\circ)\le C_{\mu,a,k}(\pi)$. This argument allows $\E Y_\pi(0)=\infty$ and arbitrary finite residual delivery times for initial units. The copy remains admissible because its inventory area does not increase and its lost count is bounded by $N(T)$.

\subsection{Proof of Proposition~\ref{prop:rank-lower}}\label{app:rank-potential}
\proof{Proof.} It is enough to prove a positive bound $g\le\min\{B_\mu(\eta),a-D_\mu(\eta)\}$ for an empty-start policy. For each rank $i\ge1$ define
\begin{equation}\label{eq:rank-value-derivative}
 V_i(\tau):=G_i(\tau)+\eta\Prb(N(\tau)\ge i)-g,
 \qquad
 V_i(0)=i-g,\quad
 V_i'(\tau)=\eta p_{i-1}(\tau)-F_{i-1}(\tau).
\end{equation}
The identity $G_i-G_{i-1}=F_{i-1}$ gives $V_i'=V_{i-1}-V_i$ for $i\ge2$. This identity explains our choice of $V_i$: a unit's change of rank at a sale will cancel the drift caused by its approaching delivery.

Rank all present and pending units FIFO, breaking ties deterministically. The $I$ on-hand units have residual time zero and are served first, so they occupy exactly ranks $1,\ldots,I$. If position is $m$ and on-hand stock is $I$, their residual delivery times satisfy $0=\tau_1=\cdots=\tau_I\le \tau_{I+1}\le\cdots\le \tau_m\le\mu$. Choose an integer $M\ge\max\{1,\mu+g\}$ and retain only the first $M$ ranks:
\[
 \Phi_M:=\sum_{i=1}^{m\wedge M}V_i(\tau_i),
 \qquad |\Phi_M|\le M(M+\eta+g).
\]
The truncation makes the potential function bounded even when the policy's inventory position is unbounded. Orders append nonnegative retained terms because $V_i(\mu)\ge B_\mu(\eta)-g\ge0$; any appended units beyond rank $M$ are omitted from the sum and leave $\Phi_M$ unchanged. Deliveries cause no jump.

Suppose $0<I\le M$, and write $m_*:=m\wedge M$. At a served demand, removing the first unit and shifting the other ranks gives exactly
\[
 \begin{split}
 \Delta_{\rm dem}\Phi_M
 &=-V_1(0)+\sum_{i=2}^{m_*}
       \{V_{i-1}(\tau_i)-V_i(\tau_i)\}
       +\mathbf1_{\{m>M\}}V_M(\tau_{M+1})\\
 &=g-I+\sum_{i=I+1}^{m_*}V_i'(\tau_i)
       +\mathbf1_{\{m>M\}}V_M(\tau_{M+1}).
 \end{split}
\]
The first $I$ terms give $-(1-g)-(I-1)=g-I$. Between events the pending terms have aging drift $-\sum_{i=I+1}^{m_*}V_i'(\tau_i)$. When $m>M$, the unit at rank $M+1$ moves to rank $M$ after a sale and enters the retained sum, contributing $V_M(\tau_{M+1})$. This contribution is nonnegative, since $V_M(\tau_{M+1})\ge M-\tau_{M+1}-g\ge0$. Here $G_M(\tau)\ge M-\tau$, $\tau\le\mu$, and $M\ge\mu+g$. If $I>M$, all retained units remain on hand after a sale, so $\Phi_M$ is unchanged. At a stockout, demand leaves every rank unchanged. Thus, writing $\dot\Phi_M$ for the aging drift, the rate-one demand contribution gives the three cases in a single inequality: {\renewcommand{\arraystretch}{0.6}
\[
\begin{aligned}
&I+a\mathbf1_{\{I=0\}}+\Delta_{\rm dem}\Phi_M+\dot\Phi_M\\
&\quad=
\left\{\begin{array}{@{}l@{\quad}l@{}}
I+(g-I)+\mathbf1_{\{m>M\}}V_M(\tau_{M+1})\ge g,
 & 1\le I\le M,\\
I\ge g, & I>M,\\[2pt]
a-\displaystyle\sum_{i=1}^{m\wedge M}V_i'(\tau_i)
 \ge a-D_\mu(\eta)\ge g, & I=0.
\end{array}\right.
\end{aligned}
\]}%
The stockout case uses only the definition of $D_\mu(\eta)$: the first $m\wedge M$ units form an ordered schedule in $[0,\mu]$, one of the schedules admitted in \eqref{eq:schedule-supremum}, so the sum of their rates is at most $D_\mu(\eta)$ in every state.

The demand integrand is bounded and predictable, so Lemma~\ref{lem:poisson-compensation} applies; the aging derivatives are also bounded. Order increments are nonnegative. Because the potential function, each demand jump, and the aging rate are bounded, telescoping bounds the sum of order increments by a constant times $1+T+N(T)$. Thus, they are integrable even when the order count has no finite expectation. Applying the lemma to demand increments and the loss indicator, and retaining the nonnegative order increments, gives
\begin{equation}\label{eq:finite-horizon-rank-bound}
 \E\left[\int_0^T X(t)\,dt+a\,\mathcal U(T)\right]
 \ge gT+\E\Phi_M(0)-\E\Phi_M(T).
\end{equation}
The bounded terminal term approaches zero after division by $T$. The initial-supply argument extends the bound to every admissible policy, proving \eqref{eq:committed-supply-bound} without restricting inventory position. \Halmos

\section{Exact Penalty Calibration}\label{app:calibration-proof}
We prove Proposition~\ref{prop:calibration}. The key observation is that the residual times maximizing individual rank terms are already ordered, so their separate maxima define an admissible delivery schedule in $D_\mu$.

\proof{Proof.} For \(\eta>1\) write \(f:=1-1/\eta\) and \(w_n(\tau):=\eta p_n(\tau)-F_n(\tau)=V_{n+1}'(\tau)\). Thus, index $n$ corresponds to the unit of rank $n+1$. For \(n\ge1\), $w_n'(\tau)=p_n(\tau)\{\eta n/\tau-(\eta-1)\}$. Here we write the argument explicitly: $H_j(t):=j-t+t\,e_j(t)=\E[j-N(t)\mid N(t)\le j]$, where $e_j(t):=p_j(t)/F_j(t)$. Our earlier notation $H_j$ and $e_j$ refers to these functions at $t=\mu$. Since $w_n(\tau)=F_n(\tau)[\eta e_n(\tau)-1]$, the sign of $w_n(\tau)$ is that of $e_n(\tau)-1/\eta$.

Recall the definition in \eqref{eq:schedule-supremum}, using $V_i'=w_{i-1}$:
\[
D_\mu(\eta)=\sup_{m\ge0,\,0\le\tau_1\le\cdots\le\tau_m\le\mu}
\sum_{i=1}^m w_{i-1}(\tau_i).
\]
The supremum allows any number $m$ of pending units and any ordered residual delivery times $\tau_i$ in $[0,\mu]$. For each rank $i\ge1$, write $\tau_i^\star:=\min\{\mu,(i-1)/f\}$ for the time that maximizes $V_i'(\tau)=w_{i-1}(\tau)$ over $[0,\mu]$; by the derivative formula, $w_n$ increases on $[0,n/f]$ and decreases afterward. In particular, $\tau_1^\star=0$ and $w_0(0)=\eta-1$. These maximizing times are nondecreasing in $i$.

The maxima $w_n(\tau_{n+1}^\star)$ are positive for consecutive indices starting at $n=0$. For an interior maximizer with $n\ge1$, $H_n(n/f)>0$ gives $e_n(n/f)=1-f+fH_n(n/f)/n>1/\eta$, so $w_n(n/f)>0$. For indices $n\ge f\mu$, the maximum occurs at $\tau=\mu$, and $w_n(\mu)>0$ if and only if $e_n(\mu)>1/\eta$. Since $e_n(\mu)$ decreases with $n$ by \eqref{eq:erlang-derivatives}, once $w_n(\mu)$ becomes nonpositive, it remains nonpositive for every larger $n$. Only finitely many can be positive: uniformly on \(\tau\le\mu\), \(p_n(\tau)=(\tau/n)p_{n-1}(\tau)\le(\mu/n)F_n(\tau)\), so \(\eta p_n(\tau)<F_n(\tau)\) for \(n>\eta\mu\). It follows that
\begin{equation}\label{eq:schedule-rank-maxima}
D_\mu(\eta)=\sum_{n\ge0}[w_n(\tau_{n+1}^\star)]^+.
\end{equation}
Any feasible schedule has total contribution at most this sum of individual positive maxima. Let $m^\star$ be the number of positive maxima. Equality is attained by choosing $m=m^\star$ and setting $\tau_i=\tau_i^\star$ for $i=1,\ldots,m^\star$; these times are ordered, and the selected units occupy the first $m^\star$ ranks.

For \(B_l:=G_l(\mu)+\eta\mathbb P(N(\mu)\ge l)\), \(B_0=\eta\) and $B_{l+1}-B_l=F_l(\mu)-\eta p_l(\mu)=-w_l(\mu)$. The sequence $(B_l)$ decreases while $w_l(\mu)>0$, which by the monotonicity of $e_l(\mu)$ happens exactly on an initial segment of indices, and is nondecreasing afterward. Its minimum $B_\mu(\eta)$ is, therefore, attained at the end of that segment, and telescoping gives $\sum_{n\ge0}[w_n(\mu)]^+=\eta-B_\mu(\eta)$. If \(n<f\mu\), then \(H_n\ge0\) implies \(e_n\ge1-n/\mu>1/\eta\), so \(w_n(\mu)>0\). All other ranks already maximize at \(\mu\). Hence,
\begin{equation}\label{eq:rank-telescoping}
B_\mu(\eta)+D_\mu(\eta)
=\eta+\sum_{n<f\mu}\{w_n(n/f)-w_n(\mu)\}.
\end{equation}
 Keep $\eta$ (and hence $f=1-1/\eta$) fixed and differentiate the sum on the right of \eqref{eq:rank-telescoping} in $\mu$. For every active index $n<f\mu$, the maximizer $n/f$ of $w_n$ over $[0,\mu]$ is interior and does not move with $\mu$, so
\[
 \frac{d}{d\mu}\{w_n(n/f)-w_n(\mu)\}
 =-w_n'(\mu)=\eta p_n(\mu)\left(f-\frac n\mu\right).
\]
The index set $\{n<f\mu\}$ of the sum in \eqref{eq:rank-telescoping} grows with $\mu$; when an index $n\ge1$ joins it as $\mu$ crosses $n/f$, its summand $w_n(n/f)-w_n(\mu)$ and this derivative are both zero. Thus, there is no boundary term, and the derivative of the sum in \eqref{eq:rank-telescoping} is $\eta\sum_{n<f\mu}p_n(\mu)\left(f-{n/\mu}\right) ={(\eta/\mu)}\,\E(f\mu-N(\mu))^+$. The sum in \eqref{eq:rank-telescoping} vanishes at $\mu=0$: the index $n=0$ is active for every $\mu>0$, with summand $w_0(0)-w_0(\mu)=(\eta-1)(1-e^{-\mu})$, and every index $n\ge1$ joins with a zero summand.

Consequently
\begin{equation}\label{eq:calibration-integral}
B_\mu(\eta)+D_\mu(\eta)
=\eta+\eta\int_0^\mu\frac{\mathbb E(ft-N(t))^+}{t}\,dt.
\end{equation}
 The finite-horizon mean form of the continuous-time Spitzer identity \citep[Proposition~1, p.~1138]{DiaLamberton2011} applies to the integrable L\'evy process $ft-N(t)$. For background on the underlying Wiener--Hopf factorization, see \citet[Section~6.4]{Kyprianou2014}. The identity's positive large-jump integrability condition holds because every jump is negative. The identity gives
\[
\mathbb E\sup_{0\le s\le\mu}\{fs-N(s)\}
=\int_0^\mu\frac{\mathbb E(ft-N(t))^+}{t}\,dt.
\]
 The integrand is bounded above by \(f\), so there is no integrability issue at zero. Substituting this equality in \eqref{eq:calibration-integral} proves \eqref{eq:calibration-identity}.

For \(\eta=1\), \(\eta p_n-F_n\le0\), so \(D_\mu(1)=0\). The rank-one value in \eqref{eq:rank-minimum} is \(1\) and successive rank values are nondecreasing, so \(B_\mu(1)=1\). The expected maximum in \eqref{eq:calibration-map} is zero, and \eqref{eq:calibration-identity} follows at this endpoint as well.

Finally, the expected maximum in \eqref{eq:calibration-map} is nondecreasing in $\eta$. Changing the slope $1-1/\eta$ by $\varepsilon$ changes this expected maximum by at most $\mu|\varepsilon|$. Thus, $\alpha_\mu$ is continuous and strictly increasing, starts at one, and is at least $\eta$. This proves the stated range. \Halmos
\endproof

\section{Bounds on the Largest Excess of Supply over Demand}\label{app:supply-excess}

Throughout this section, \(N\) is a rate one Poisson process and \(W_T(\ell):=\mathbb E\sup_{0\le t\le T}\{(1-\ell)t-N_t\}\), for \(0\le \ell\le1\). Here $W_T(\ell)$ is the expected largest excess of auxiliary cumulative supply $(1-\ell)t$ over demand $N_t$ during $[0,T]$. In this notation, the calibrated penalty \eqref{eq:calibration-map} reads \(\alpha_\mu(\eta)=\eta[1+W_\mu(1/\eta)]\), and Appendix~\ref{app:calibration-proof} established \(B_\mu(\eta)+D_\mu(\eta)=\alpha_\mu(\eta)\). Write \(\phi\) and \(\Phi\) for the standard normal density and distribution, and set \(\overline\Phi:=1-\Phi\) for the upper tail. Also set $w(z):=\phi(z)-z\overline\Phi(z)+{[\Phi(z)-1/2]/z}, z>0, w(0):=\sqrt{2/\pi}$.

\subsection{A Finite-Horizon Brownian Comparison}\label{app:brownian-comparison}
\begin{lemma}[{\sc Brownian Comparison}]\label{lem:brownian-comparison}
For all \(T>0\) and \(0\le \ell\le1\),
\[
 W_T(\ell)\le \sqrt T\,w(\ell\sqrt T).
\]
\end{lemma}
\proof{Proof.} The comparison follows because the Brownian continuation value has nonincreasing curvature: a downward Poisson jump then contributes no more than the Brownian second-derivative term. Let \(\mathcal H(t,y)\) be the probability that a standard Brownian motion with drift \(-\ell\) exceeds \(y\) at some time in \([0,t]\), and let \(\mathcal V(t,y):=\int_y^\infty\mathcal H(t,x)\,dx\), the expected excess over \(y\) of its maximum on \([0,t]\). The running-maximum law for Brownian motion with drift \citep{Harrison1985} gives, with $y_+:=(y+\ell t)/\sqrt t$ and $y_-:=(\ell t-y)/\sqrt t$, $\mathcal H(t,y)=\overline\Phi(y_+)+e^{-2\ell y}\Phi(y_-)$. Hence, \(\mathcal V_y=-\mathcal H\) and \(\mathcal V_y(t,0)=-1\). The identity \(e^{-2\ell y}\phi(y_-)=\phi(y_+)\) gives
\begin{align*}
\mathcal V_{yy}(t,y)
 &=\frac{2\phi(y_+)}{\sqrt t}+2\ell e^{-2\ell y}\Phi(y_-),\\
\mathcal V_{yyy}(t,y)
 &=-\frac{2(y+2\ell t)\phi(y_+)}{t^{3/2}}
   -4\ell^2e^{-2\ell y}\Phi(y_-)\le0,\\
\mathcal V_t(t,y)&=\frac{\phi(y_+)}{\sqrt t}-\ell\overline\Phi(y_+)
          =\ell\mathcal V_y(t,y)+\tfrac12\mathcal V_{yy}(t,y).
\end{align*}
For the last line, differentiating the tail gives \(\mathcal H_t=-\partial_y[\phi(y_+)/\sqrt t-\ell\overline\Phi(y_+)]\), whose integral over \([y,\infty)\) is the displayed expression. For every $\epsilon>0$, these formulas show that $\mathcal V$ is $C^{1,2}$ on $[\epsilon,T+\epsilon]\times[0,\infty)$, with derivatives at $y=0$ taken from the right. This is the regularity used in the drift calculation below. Also \(\mathcal V(0,y)=0\).

Write \(f:=1-\ell\), \(\mathcal X_s:=fs-N_s\), \(\overline{\mathcal X}_s:=\sup_{r\le s}\mathcal X_r\), and \(\mathcal Y_s:=\overline{\mathcal X}_s-\mathcal X_s\). Between jumps, \(\mathcal Y\) decreases at speed \(f\) when positive and \(\overline{\mathcal X}\) increases at speed \(f\) when \(\mathcal Y=0\); a jump increases \(\mathcal Y\) by one. For \(\epsilon>0\), at \(t=T+\epsilon-s\) the drift of \(\overline{\mathcal X}_s+\mathcal V(t,\mathcal Y_s)\) is $-\mathcal V_t-f\mathcal V_y+\mathcal V(t,y+1)-\mathcal V(t,y)$. This formula holds at zero as well, because the regulator drift is \(f=-f\mathcal V_y(t,0)\). Taylor's formula with integral remainder gives
\[
\mathcal V(t,y+1)-\mathcal V(t,y)
=\mathcal V_y(t,y)+\int_0^1(1-r)\mathcal V_{yy}(t,y+r)\,dr
\le \mathcal V_y(t,y)+\tfrac12\mathcal V_{yy}(t,y).
\]
The drift is, therefore, at most \(-\mathcal V_t+\ell\mathcal V_y+{(1/2)}\mathcal V_{yy}=0\). For \(0\le s\le T\), the time argument is at least \(\epsilon\), \(|\mathcal V_y|\le1\), \(\mathcal V_{yy}\le2/\sqrt{2\pi\epsilon}+2\ell\), and \(0\le\overline{\mathcal X}_s\le fT\). The demand integrand is predictable and bounded by one in absolute value, so Lemma~\ref{lem:poisson-compensation} gives $\mathbb E\overline{\mathcal X}_T+\mathbb E\mathcal V(\epsilon,\mathcal Y_T)\le \mathcal V(T+\epsilon,0)$. Drop the nonnegative second term and let \(\epsilon\downarrow0\). Here \(\mathbb E\overline{\mathcal X}_T=W_T(\ell)\), and integration of the reflection formula gives \(\mathcal V(T,0)=\sqrt T\,w(\ell\sqrt T)\), proving the comparison. \Halmos

For the estimates below, differentiating the definition of $w$ yields
\[
w'(0)=-\tfrac12,\qquad
w''(z)=\frac{2}{z^3}\int_0^z x^2\phi(x)\,dx
=\sqrt{2/\pi}\int_0^1s^2e^{-z^2s^2/2}\,ds
\quad(z>0).
\]
In particular \(0\le w''\le\sqrt{2/\pi}/3\), including its limit at zero, and Taylor's formula with \(w(0)=\sqrt{2/\pi}\) and \(w'(0)=-1/2\) gives
\begin{equation}\label{eq:workload-taylor}
w(z)\le\sqrt{2/\pi}\left(1+\frac{z^2}6\right)-\frac z2.
\end{equation}
If $f\mu\le1$, where $f=1-1/\eta$, then $ft-N(t)$ can be positive only before the first demand. Hence, only the $n=0$ term can contribute to the sum in \eqref{eq:rank-telescoping}. Thus,
\begin{equation}\label{eq:first-fluid-band}
\alpha_\mu(\eta)=\eta+(\eta-1)(1-e^{-\mu}).
\end{equation}

\subsection{A Uniform Bound on the Stockout Aging Term}\label{app:stockout-aging}
\begin{lemma}[{\sc Stockout Aging Bound}]\label{lem:simple-penalty}
For $\eta\ge1$,
\begin{equation}\label{eq:simple-penalty}
D_\mu(\eta)\le\frac{\eta^2}{2}-\frac \eta3.
\end{equation}
\end{lemma}
\proof{Proof.} Write $f:=1-1/\eta$ and $\lambda_\eta:=6/(3\eta-2)$. First, for $z\ge0$, $e^{-z}\le\frac{z^2-4z+6}{2z+6}$. Indeed, the difference $e^z(z^2-4z+6)-(2z+6)$ vanishes at zero together with its first derivative, and its second derivative is $e^zz^2\ge0$. The process $\exp\{\lambda(ft-N_t)\}$ has drift rate $f\lambda+e^{-\lambda}-1$ times its current value, and the inequality above with $z=\lambda_\eta$ shows that the choice of $\lambda_\eta$ makes this rate nonpositive: $f\lambda_\eta+e^{-\lambda_\eta}-1 \le\frac{\lambda_\eta[(3\eta-2)\lambda_\eta-6]}{\eta(2\lambda_\eta+6)}=0$. Thus, $\exp\{\lambda_\eta(ft-N_t)\}$ is a nonnegative supermartingale. Stop it at $\tau_x\wedge T$, where $\tau_x=\inf\{t:ft-N_t\ge x\}$ and $x>0$. This is a bounded stopping time; moreover the stopped exponential is at most $e^{\lambda_\eta fT}$, so the optional stopping theorem applies. Since the process has no upward jumps, its value at $\tau_x$ is $x$. Therefore, using $W_T(1/\eta)=\mathbb E\sup_{t\le T}(ft-N_t)=\int_0^\infty\Prb(\tau_x\le T)\,dx$,
\[
1\ge\mathbb E e^{\lambda_\eta(f(\tau_x\wedge T)-N_{\tau_x\wedge T})}
\ge e^{\lambda_\eta x}\Prb(\tau_x\le T),\qquad
W_T(1/\eta)\le\int_0^\infty e^{-\lambda_\eta x}\,dx
=\frac \eta2-\frac13.
\]
The residual-schedule domains increase with their horizon, so $D_\mu(\eta)\le D_T(\eta)$ for $T\ge\mu$. Applying \eqref{eq:calibration-identity} with lead-time demand $T$ in place of $\mu$ gives $D_T(\eta)=\eta[1+W_T(1/\eta)]-B_T(\eta)$, and the bound $W_T(1/\eta)\le1/\lambda_\eta$ just obtained then yields
\begin{equation}\label{eq:horizon-bound}
D_\mu(\eta)\le D_T(\eta)\le \eta+\eta/\lambda_\eta-B_T(\eta).
\end{equation}
It remains to show $B_T(\eta)\to \eta$. The positive terms in $\eta-B_T(\eta)=\sum_n[\eta p_n(T)-F_n(T)]^+$ form an initial segment, say $0,\ldots,i$. If it is nonempty, then $0\le \eta-B_T(\eta)\le \eta F_i(T)<\eta^2p_i(T)\le \eta^2\max_n p_n(T)$; the same bound is immediate for an empty segment. For $T\ge1$, the maximum is attained at the mode $\lfloor T\rfloor\ge1$. Stirling's lower bound $k!\ge\sqrt{2\pi k}(k/e)^k$ \citep{Robbins1955} and $\log x\le x-1$ give
\[
\max_n p_n(T)=\frac{e^{-T}T^{\lfloor T\rfloor}}{\lfloor T\rfloor!}
\le\frac{\exp\{\lfloor T\rfloor-T+\lfloor T\rfloor\log(T/\lfloor T\rfloor)\}}{\sqrt{2\pi\lfloor T\rfloor}}
\le\frac1{\sqrt{2\pi\lfloor T\rfloor}}\longrightarrow0\qquad\text{as }T\to\infty.
\]
Thus, $B_T(\eta)\to\eta$, and letting $T\to\infty$ in \eqref{eq:horizon-bound} gives $D_\mu(\eta)\le\eta/\lambda_\eta=\eta^2/2-\eta/3$, which is \eqref{eq:simple-penalty}. \Halmos

\section{The Starting Penalty for Small Base-Stock Levels}\label{app:small-stock}
\proof{Proof of Lemma~\ref{lem:endpoints} part~(i).} At $\eta^\circ=A_1=1+1/\mu$, the successive differences $F_i-\eta^\circ p_i$ of the rank terms in \eqref{eq:rank-minimum} (Appendix~\ref{app:calibration-proof}) are negative at $i=0$, zero at $i=1$, and positive for $i\ge2$, so ranks one and two minimize $B_\mu$. Moreover, $(1-1/\eta^\circ)\mu=\mu/(\mu+1)<1$, and \eqref{eq:first-fluid-band}, together with $B_\mu(\eta^\circ)=G_1(\mu)+\eta^\circ\Prb(N(\mu)\ge1)=e^{-\mu}+(1+1/\mu)(1-e^{-\mu})$, yields $a^\circ=1+{(2-e^{-\mu})/\mu}, g^\circ=B_\mu(\eta^\circ)=1+{(1-e^{-\mu})/\mu}$. We bound the base-stock cost at level one first, compare base stock two with it, and then handle base stock three for $\mu\ge2$.

For base stock one, $e^\mu\ge1+\mu$ gives $1-e^{-\mu}\ge\mu/(1+\mu)$. Its ratio is
\[
 \frac{L_1(a^\circ)}{g^\circ}
 =\frac{\mu}{\mu+1}\left(1+\frac2{\mu+1-e^{-\mu}}\right)
 \le1+\frac{\mu}{(\mu+1)(\mu+2)}<\frac54.
\]
The first inequality uses $1-e^{-\mu}\ge\mu/(1+\mu)$ in the denominator $\mu+1-e^{-\mu}$, and the last follows from $(\mu+1)(\mu+2)-4\mu=(\mu-1/2)^2+7/4>0$. The exact costs for base stocks one and two are $L_1(a)={(1+a\mu)/(\mu+1)}, L_2(a)=\frac{2\mu+4+a\mu^2}{\mu^2+2\mu+2}$. At $a=a^\circ$, their difference simplifies to $L_2(a^\circ)-L_1(a^\circ) =\frac{e^{-\mu}(\mu+2)-2}{(\mu+1)(\mu^2+2\mu+2)}<0$, because $e^{-\mu}\le1/(1+\mu)<2/(\mu+2)$.

For base stock three, with $P_3:=\mu^3+3\mu^2+6\mu+6$, the exact cost is
\[
 L_3(a)=\frac{3\mu^2+12\mu+18+a\mu^3}{P_3}
 \quad\Longrightarrow\quad
 P_3L_3(1+2/\mu)
 =\mu^3+5\mu^2+12\mu+18=(1+2/\mu)P_3-12/\mu.
\]
Since $L_3$ is increasing in $a$ and $a^\circ<1+2/\mu$, this gives $L_3(a^\circ)<1+2/\mu$. For $\mu\ge2$, the exponential series gives $e^{-\mu}\le e^{-2}<1/7$, because $e^2>\sum_{n=0}^4 2^n/n!=7$; hence $g^\circ\ge1+6/(7\mu)$ and ${L_3(a^\circ)/g^\circ} <\frac{\mu+2}{\mu+6/7}\le{7/5}$. The last inequality is exactly $\mu\ge2$, completing \eqref{eq:low-price-anchor}. \Halmos

\section{Base-Stock Comparisons at or Above Mean Demand}\label{app:upper-stock}
\proof{Proof of Lemma~\ref{lem:endpoints} part~(iii).} Fix an integer $j\ge\mu>0$. At weight $A_j$, two neighboring ranks minimize $B_\mu$. We first express this lower bound through conditional Poisson moments, then compare both base-stock costs with the positive root of the same quadratic.

\paragraph{The conditional-variance identity.}
Write $N\sim\operatorname{Pois}(\mu)$ and set $\Delta:=j+1-\mu\ge1, {q_+}:={\mu p_j/(1-F_j)}$, $H_+:=\E[N-j-1\mid N>j], v_+:=\operatorname{Var}(N\mid N>j)$. Both conditioning events have positive probability. Shifting indices in the Poisson sums gives the upper-tail moments $\E[N\mid N>j]=\mu+{q_+}, \E[N(N-1)\mid N>j]=\mu^2+(\mu+j){q_+}$. Consequently $H_+={q_+}-\Delta$ and
\[
 v_+=\mu^2+(\mu+j){q_+}+\mu+{q_+}-(\mu+{q_+})^2
     =\mu+\Delta {q_+}-{q_+}^2=\mu-H_+(H_++\Delta).
\]
On the lower side, $H_j+1=\Delta+\mu p_j/F_j$, so $v_j=\mu-\mu e_j(H_j+1)=\mu-(H_j+1)(H_j+1-\Delta)$ and $H_j+1+H_+=\mu p_j/(F_j(1-F_j))$. At weight $A_j$, the rank differences $F_i-A_jp_i$ change sign at $i=j$; hence ranks $j$ and $j+1$ minimize $B_\mu(A_j)$. Evaluating rank $j+1$ and using $A_j-H_j-1=A_jv_j/\mu$ gives
\begin{equation}\label{eq:underloaded-moments}
 \begin{aligned}
 B_\mu(A_j)&=F_j(H_j+1)+A_j(1-F_j)=H_j+1+(A_j-H_j-1)(1-F_j)=H_j+1+\frac{v_j}{H_j+1+H_+},\\
 v_j&=\mu-(H_j+1)(H_j+1-\Delta),\qquad v_+=\mu-H_+(H_++\Delta).
 \end{aligned}
\end{equation}
The cancellation behind the quadratic bound is now visible:
\begin{align*}
 &B_\mu(A_j)(B_\mu(A_j)-\Delta)-\mu=-v_j+\frac{v_j(2H_j+2-\Delta)}{H_j+1+H_+}+\frac{v_j^2}{(H_j+1+H_+)^2}\\
 &\quad=\frac{v_j\{(H_j+1-H_+-\Delta)(H_j+1+H_+)+v_j\}}{(H_j+1+H_+)^2}=\frac{v_j\{\mu-H_+(H_++\Delta)\}}{(H_j+1+H_+)^2}.
\end{align*}
Therefore,
\begin{equation}\label{eq:underloaded-root}
\begin{aligned}
 B_\mu(A_j)\bigl(B_\mu(A_j)-\Delta\bigr)-\mu
 &=\frac{v_jv_+}{(H_j+1+H_+)^2}\ge0,\\
 B_\mu(A_j)&\ge\beta:=\frac{\Delta+\sqrt{\Delta^2+4\mu}}2.
\end{aligned}
\end{equation}
The second line follows from the first because $x(x-\Delta)-\mu$ is nonnegative only outside its two roots, of which $\beta$ is the larger and the other is negative, while $B_\mu(A_j)>0$.

\paragraph{The two cost bounds.}
First consider the adjacent-intersection cost $L_j(a_j)=H_j+1+v_j/(H_j+1)=\Delta+\mu/(H_j+1)$. The bound $v_j\le H_j(H_j+1)$ gives $\mu\le (H_j+1)(2H_j+1-\Delta)$, and hence $L_j(a_j)(L_j(a_j)-\Delta)=\mu\left({\Delta/(H_j+1)}+{\mu/(H_j+1)^2}\right) \le\mu(2-1/(H_j+1))<2\mu$. By contrast, $\beta(\beta-\Delta)=\mu$ implies $(\sqrt2\beta)(\sqrt2\beta-\Delta) =2\mu+(2-\sqrt2)\Delta\beta>2\mu$. Since $x(x-\Delta)$ increases for $x\ge\Delta$, we have $L_j(a_j)<\sqrt2\beta$.

Next, define $\sigma_j:=\sqrt\mu/A_j$. Expanding the integer power and using $\int_0^\infty t^l e^{-t}\,dt=l!$ gives the exact integral
\[
 A_j=\sum_{l=0}^j(j)_l\mu^{-l}
 =\int_0^\infty e^{-t}(1+t/\mu)^j\,dt
 \ge\int_0^\infty e^{-t^2/(2\mu)}\,dt
 =\sqrt{\pi/2}\,\sqrt\mu.
\]
The inequality uses $j\log(1+y)\ge\mu\log(1+y)\ge\mu(y-y^2/2)$ for $y=t/\mu$; the last logarithmic bound follows by integrating $(1+y)^{-1}\ge1-y$. Thus, $0<\sigma_j\le\sqrt{2/\pi}<4/5$. The comparison with Brownian motion in Appendix~\ref{app:brownian-comparison} and \eqref{eq:workload-taylor} give
\[
 \frac{W_\mu(1/A_j)}{\sqrt\mu}+\sigma_j\le w(\sigma_j)+\sigma_j
 <\frac45+\frac{\sigma_j}2+\frac{2\sigma_j^2}{15}
 \le\frac{482}{375}<\frac97<2\sqrt{\sqrt2-1}.
\]
The middle expression increases in $\sigma_j$, so $\sigma_j\le4/5$ gives the bound $482/375$. Here $9/7$ is simply a convenient bound below the coefficient required by the next arithmetic--geometric mean inequality. The two strict constant comparisons are verified by $482/375=9/7-1/2625$ and $2-(277/196)^2=103/38416>0$. Since $H_j+1=\Delta+\sigma_j \sqrt\mu$ and $e_jA_j=1$, the definition \eqref{eq:calibration-map} of $\alpha_\mu$ gives $L_j(\alpha_\mu(A_j))=H_j+1+W_\mu(1/A_j)<\Delta+{(9/7)}\sqrt\mu$. On the other hand, $\sqrt2\beta-\Delta =(\sqrt2-1)\beta+{\mu/\beta} \ge2\sqrt{\sqrt2-1}\,\sqrt\mu>{(9/7)}\sqrt\mu$. This proves $H_j+1+W_\mu(1/A_j)<\sqrt2\beta$.

It remains to bound base stock $j+1$. The Erlang recursion and $v_j=\mu-\mu e_j(H_j+1)$ show
\[
H_{j+1}-H_j=1-\frac{\mu e_j(H_j+1)}{\mu+H_j+1}=\frac{H_j+1+v_j}{\mu+H_j+1},\qquad e_{j+1}A_j=\frac{\mu}{\mu+H_j+1}.
\]
\begingroup\displaywidowpenalty=10000
The calibrated cost of base stock $j+1$ is, therefore, a convex combination of the two costs already bounded:
\[
\begin{aligned}
 &L_{j+1}(\alpha_\mu(A_j))=H_j+1+\frac{\mu W_\mu(1/A_j)+v_j}{\mu+H_j+1}\\
 &\quad=\frac\mu{\mu+H_j+1}(H_j+1+W_\mu(1/A_j))+\frac{H_j+1}{\mu+H_j+1}L_j(a_j)<\sqrt2\beta\le\sqrt2 B_\mu(A_j).
\end{aligned}
\]
\endgroup
This proves \eqref{eq:calibrated-endpoint}. \Halmos

\section{A Poisson-to-Gaussian Rank Bound}
\label{app:poisson-gaussian}
\begin{lemma}[{\sc Poisson-to-Gaussian Rank Bound}]\label{lem:poisson-gaussian-rank}
For $G(s):=\inf_{z\in\mathbb R} \{\phi(z)+z\Phi(z)+s[1-\Phi(z)]\}$,
\begin{equation}\label{eq:poisson-gaussian-rank}
 B_\mu(\eta)>\sqrt\mu\,G(\eta/\sqrt\mu),\qquad \mu>0,\quad \eta\ge1.
\end{equation}
\end{lemma}
\proof{Proof.} The argument compares residual means at matching gamma and normal tail probabilities, then uses the gamma--Poisson identity to return to the rank cost.

Fix an integer $k\ge1$ and let $\Gamma(k,1)$ denote a gamma random variable with shape $k$ and unit scale, whose density is $t^{k-1}e^{-t}/(k-1)!$ for $t>0$. The waiting-time identity for the Poisson process, that the $k$th arrival occurs after time $\mu$ exactly when $N(\mu)\le k-1$, gives $\Prb(\Gamma(k,1)>\mu)=F_{k-1}(\mu)$. Write $z_k:=\Phi^{-1}(F_{k-1}(\mu)), {q_G}:={\phi(z_k)/\Phi(z_k)}, m_G:=z_k+{q_G}, m_\Gamma:=\mathbb E[\Gamma(k,1)-\mu\mid\Gamma(k,1)>\mu]$. First we need the strict quantile bound
\begin{equation}\label{eq:gamma-normal-quantile-gap}
z_k\sqrt\mu<k-\mu.
\end{equation}
Set $\zeta(\mu):=\operatorname{sgn}(k-\mu) \sqrt{2[k\log(k/\mu)-k+\mu]}$. The ratio of the gamma density to $-d\Phi(\zeta)/d\mu$ is ${\sqrt{2\pi}\,k^ke^{-k}/(k-1)!}{(|\zeta|/|k-\mu|)}$, with its continuous value at $\mu=k$; this follows from $\zeta\zeta'=1-k/\mu$ and $\phi(\zeta)=(\mu/k)^ke^{k-\mu}/\sqrt{2\pi}$. This ratio decreases strictly from infinity to zero. Indeed, the endpoint asymptotics are
\[
\begin{aligned}
\mu\downarrow0:&\quad \zeta\sim\sqrt{2k\log(k/\mu)},
&\frac{|\zeta|}{|k-\mu|}&\longrightarrow\infty,\\
\mu\to\infty:&\quad \zeta\sim-\sqrt{2\mu},
&\frac{|\zeta|}{|k-\mu|}&\sim\sqrt{2/\mu}\longrightarrow0.
\end{aligned}
\]
To check that the ratio decreases in $\mu$, write $\nu:=\mu/k$; since $\zeta^2=2k(\nu-1-\log \nu)$ and $(k-\mu)^2=k^2(\nu-1)^2$, the square of the ratio is $\mathrm{const}\cdot(\nu-1-\log \nu)/(\nu-1)^2$ with $\mathrm{const}=4\pi k^{2k-1}e^{-2k}/((k-1)!)^2$, and the derivative of $(\nu-1-\log \nu)/(\nu-1)^2$ has the sign of ${(-\nu+\nu^{-1}+2\log \nu)/(\nu-1)^3}<0$. The numerator vanishes at one and has derivative $-(\nu-1)^2/\nu^2$. Since $F_{k-1}'(\mu)$ is minus the gamma density and $\zeta$ decreases in $\mu$, the derivative $(F_{k-1}(\mu)-\Phi(\zeta))'=(-d\Phi(\zeta)/d\mu)(1-\text{ratio})$ is negative before the ratio crosses one and positive afterward. Since $F_{k-1}(\mu)-\Phi(\zeta)$ tends to zero at both endpoints, $F_{k-1}(\mu)<\Phi(\zeta)$. This inequality is the Poisson form of the universal inequalities of \citet{AlfersDinges1984}, proved in full by \citet{ZubkovSerov2013}; we include a short proof because the strict form is needed. Also, with $\nu=\mu/k$ as above, $(k-\mu)^2/\mu-\zeta^2=(k^2/\mu)\,\Xi(\nu)$ for $\Xi(\nu):=1-\nu^2+2\nu\log\nu$, which has $\Xi(1)=0$ and $\Xi'(\nu)=2(1-\nu+\log\nu)\le0$. Thus, $\Xi$ has the sign of $1-\nu$, that is, of $k-\mu$; taking signed square roots gives $\zeta\le(k-\mu)/\sqrt\mu$. Combining this with $z_k<\zeta$, which holds because $\Phi(z_k)=F_{k-1}(\mu)<\Phi(\zeta)$, proves \eqref{eq:gamma-normal-quantile-gap}.

Let ${q_\Gamma}:=p_{k-1}(\mu)/F_{k-1}(\mu)$ be the gamma hazard at $\mu$, the density divided by the tail probability. The Erlang tail formula and one integration by parts give $F_{k-1}(\mu)=e^{-\mu}\sum_{n=0}^{k-1}{\mu^n/n!}, m_\Gamma=k-\mu+\mu{q_\Gamma}$. The residual-mean formula $F_{k-1}(\mu)m_\Gamma=\int_\mu^\infty \Prb(\Gamma(k,1)>t)\,dt$, together with $F_{k-1}'(\mu)=-{q_\Gamma} F_{k-1}(\mu)$, yields $m_\Gamma'={q_\Gamma} m_\Gamma-1$. Also, differentiating $\Phi(z_k)=F_{k-1}(\mu)$ gives $z_k'=-{q_\Gamma}/{q_G}$, and $dm_G/dz_k=1+dq_G/dz_k=1-{q_G}m_G$. Hence, for $\mathcal T(\mu):=m_\Gamma-\sqrt\mu m_G$, the chain rule gives
\[
\mathcal T'=m_\Gamma'-\frac{m_G}{2\sqrt\mu}-\sqrt\mu\,z_k'\,\frac{dm_G}{dz_k}
={q_\Gamma} \mathcal T+\frac{\sqrt\mu{q_\Gamma}-{q_G}}{{q_G}}
             -\frac{m_G}{2\sqrt\mu}.
\]
At a zero of $\mathcal T$, the identity for $m_\Gamma$ gives ${q_\Gamma}={q_G}/\sqrt\mu-(k-\mu-z_k\sqrt\mu)/\mu$, so $\mathcal T'=-{(k-\mu-z_k\sqrt\mu)/({q_G}\sqrt\mu)}-{m_G/(2\sqrt\mu)}<0$. The first term is negative by \eqref{eq:gamma-normal-quantile-gap}, and $m_G>0$ follows, for example, from $\phi(z_k)+z_k\Phi(z_k)=\int_{-\infty}^{z_k}(z_k-x)\phi(x)\,dx>0$.

The crossing argument also needs the sign of $\mathcal T(\mu)$ for large $\mu$. It follows from an asymptotic expansion of the ratio $\mathcal M(z):=\Phi(z)/\phi(z)=\int_0^\infty e^{zt-t^2/2}\,dt$ (the Mills ratio), which we also use below. Substituting $t=v/x$ and expanding $e^{-v^2/(2x^2)}$ give the classical expansion \citep[equation 26.2.12]{AbramowitzStegun1964}
\begin{equation}\label{eq:mills-for-comparisons}
\mathcal M(-x)=x^{-1}-x^{-3}+O(x^{-5}),\qquad
{q_G}(-x)=x+x^{-1}+O(x^{-3})\quad(x\to\infty).
\end{equation}
The remainder is integrable because $0\le e^{-a}-1+a\le a^2/2$ for $a\ge0$. For fixed $k$, the finite Erlang sum gives $m_\Gamma=1+O(\mu^{-1}), \log F_{k-1}(\mu)=-\mu+(k-1)\log\mu-\log(k-1)!+O(\mu^{-1})$. Writing $z_k=-x$, the equality $F_{k-1}(\mu)=\phi(x)\mathcal M(-x)$ and \eqref{eq:mills-for-comparisons} imply $x^2/2+\log x=\mu+O(\log\mu)$, so $x/\sqrt\mu\to\sqrt2$. Since $m_G(-x)=x^{-1}+O(x^{-3})$, it follows that $\mathcal T\to1-1/\sqrt2>0$. If $\mathcal T(\mu_0)\le0$ for some $\mu_0>0$, continuity and eventual positivity imply that the zeros of $\mathcal T$ in $[\mu_0,\infty)$ have a largest element $\mu_1$. Then $\mathcal T(\mu)>0$ for every $\mu>\mu_1$, so $\mathcal T'(\mu_1)\ge0$, contradicting the strictly negative derivative at every zero. Thus, $m_\Gamma>\sqrt\mu m_G$ for all $\mu>0$.

Finally, for $N\sim\operatorname{Pois}(\mu)$, the identity $n p_n(\mu)=\mu p_{n-1}(\mu)$ shows explicitly how the residual mean enters the rank cost: $\mathbb E(k-N)^+ =(k-\mu)F_{k-1}(\mu)+\mu p_{k-1}(\mu)=F_{k-1}(\mu)m_\Gamma$. Therefore, using $F_{k-1}(\mu)=\Phi(z_k)$, $m_\Gamma>\sqrt\mu\,m_G$, and the definition of $G$,
\[
\begin{aligned}
\mathbb E(k-N)^++\eta\Prb(N\ge k)
 &=F_{k-1}(\mu)m_\Gamma+\eta(1-F_{k-1}(\mu))\\
 &>\sqrt\mu\{\phi(z_k)+z_k\Phi(z_k)+(\eta/\sqrt\mu)[1-\Phi(z_k)]\}
 \ge\sqrt\mu\,G(\eta/\sqrt\mu).
\end{aligned}
\]
The rank-zero cost is $\eta$, which satisfies the same strict bound: for finite $s>0$, choosing $z$ sufficiently negative makes $\phi(z)+(z-s)\Phi(z)=\Phi(z)[m_G(z)-s]<0$, and hence $G(s)<s$. The minimum over ranks is attained because $\mathbb E(k-N)^+\ge k-\mu\to\infty$. This proves \eqref{eq:poisson-gaussian-rank}. \Halmos

\subsection{A Rational Gaussian Bound}\label{app:rational-gaussian}
\begin{lemma}[{\sc Rational Gaussian Bound}]\label{lem:rational-gaussian}
For $s>0$,
\begin{equation}\label{eq:overload-global-profile}
G(s)>\frac{13s}{13+2s^2}.
\end{equation}
\end{lemma}
\proof{Proof.} For a real $z$, recall that $q_G(z)=\phi(z)/\Phi(z)=1/\mathcal M(z)$ and $m_G(z)=z+q_G(z)$. We suppress their arguments in the derivative calculations below. Differentiating the expression minimized in $G(s)$ gives $\Phi(z)-s\phi(z)=\phi(z)[\mathcal M(z)-s]$. Differentiating the integral form of $\mathcal M$ gives $\mathcal M'=\int_0^\infty te^{zt-t^2/2}\,dt=1+z\mathcal M>0$, the second equality by integration by parts, and $\mathcal M$ ranges from zero to infinity. Thus, the unique minimizer is $z_s^\star:=\mathcal M^{-1}(s)$, characterized by $s=\mathcal M(z_s^\star)$. Write $\mathcal S(z):=z/{q_G}+{(13/2)}{q_G}m_G$. We first prove $\mathcal S>4$. Since ${q_G}=1/\mathcal M$ gives ${q_G}'=-{q_G}m_G$ and $m_G'=1-{q_G}m_G$, setting $\xi:=z/{q_G}$ gives ${q_G}\mathcal S'=1+{q_G}^2\bigl[\xi(1+\xi)+{(13/2)}\bigr]-{(13/2)}{q_G}^4(1+\xi)(\xi+2)$. Suppose that $\mathcal S=4$ at some $z$. Then $m_G={q_G}(1+\xi)>0$ implies $\xi>-1$, and $\mathcal S=\xi+{(13/2)}{q_G}^2(1+\xi)=4$ implies ${q_G}^2=2(4-\xi)/[13(1+\xi)]$ and $\xi<4$. Eliminating ${q_G}^2$ in the preceding derivative yields ${(13/2)}(1+\xi){q_G}\mathcal S'={1/2}+4\xi+9\xi^2-2\xi^3>0$, where the inequality holds because, for $-1<\xi\le0$, the cubic equals $9(\xi+2/9)^2+1/18-2\xi^3>0$, and, for $0\le\xi<4$, it equals $1/2+4\xi+\xi^2(9-2\xi)>0$. By \eqref{eq:mills-for-comparisons}, $z/{q_G}\to-1$ and ${q_G}m_G\to1$ as $z\to-\infty$, so $\mathcal S\to11/2$. If $\mathcal S$ ever reached $4$, then $\mathcal S'\le0$ would hold at the first such $z$, contradicting the positive derivative just shown. Thus, $\mathcal S>4$ everywhere.

To connect this bound to $G$, Cauchy--Schwarz gives $\mathcal M(z)\mathcal M(-z)\ge \left(\int_0^\infty e^{-t^2/2}\,dt\right)^2=\pi/2$. At the minimizer $z_s^\star$, we have $\Phi(z_s^\star)=s\phi(z_s^\star)$ and ${G(s)/s}=1-\Phi(z_s^\star)\bigl[1-q_G(z_s^\star)m_G(z_s^\star)\bigr]$. Since $s=1/q_G(z_s^\star)$, the desired bound $G(s)>13s/(13+2s^2)$ is equivalent to positivity of the numerator on the left below:
\[
\begin{aligned}
\frac{[2+13q_G(z_s^\star)^2]G(s)/s-13q_G(z_s^\star)^2}{2\phi(z_s^\star)}
&=\mathcal M(-z_s^\star)+m_G(z_s^\star)-\frac{13}2q_G(z_s^\star)\bigl[1-q_G(z_s^\star)m_G(z_s^\star)\bigr]\\
&\ge q_G(z_s^\star)\left[\frac\pi2-\frac{11}2+\mathcal S(z_s^\star)\right]
>q_G(z_s^\star)\,\frac{\pi-3}2>0.
\end{aligned}
\]
The equality uses $1/\phi(z_s^\star)=\mathcal M(z_s^\star)+\mathcal M(-z_s^\star)$ and $\mathcal M(z_s^\star)=1/q_G(z_s^\star)$. The first inequality uses the Cauchy--Schwarz bound $\mathcal M(-z_s^\star)\ge(\pi/2)q_G(z_s^\star)$ and the definition of $\mathcal S(z_s^\star)$; the strict inequality uses $\mathcal S(z_s^\star)>4$. This proves \eqref{eq:overload-global-profile}. \Halmos

\section{Base-Stock Comparisons Below Mean Demand}\label{app:below-mean-comparisons}

Fix an integer $j\ge1$ and $\mu\ge\max\{j,3\}$. At the intersection of the costs of base stocks $j$ and $j+1$, write
\[
{\vartheta}:=\frac{v_j}{\mu},\qquad
L_j(a_j)=H_j+1+\frac{v_j}{H_j+1},\quad \tilde H_j:=H_j-\frac{v_j}{H_j+1},\quad a_j=\frac{H_j+1+v_j}{1-{\vartheta}}.
\]
Here $L_j(a_j)=L_{j+1}(a_j)$. We seek a rank weight $\widehat{\eta}_j\ge1$ for which both $B_\mu(\widehat{\eta}_j)$ and $a_j-D_\mu(\widehat{\eta}_j)$ are at least $L_j(a_j)/\sqrt2$. The proof first obtains a relation between the two moments $H_j,v_j$, then uses it to bound the available penalty $a_j$, and finally chooses $\widehat{\eta}_j$.

The variance bound~\eqref{eq:erlang-moment-bound} and the identity $\mu-H_j(H_j+1)-v_j=(H_j+1)(\mu-j)$, with $H_j\ge0$ and $\mu\ge j$, give
\begin{equation}\label{eq:overload-domain}
0<v_j\le H_j(H_j+1),\qquad \mu\ge H_j(H_j+1)+v_j=(H_j+1)(L_j(a_j)-1).
\end{equation}
Thus, $H_j>0$, $L_j(a_j)>1$, $\tilde H_j\ge0$, and $0<{\vartheta}\le1/2$; the last bound uses $\mu\ge H_j(H_j+1)+v_j\ge2v_j$. To separate scale from the moment ratios, write
\[
\rho:=\frac{v_j}{(H_j+1)^2}\in(0,1),\qquad u:=\frac{1}{H_j+1}>0,\qquad
\kappa:=\frac{\tilde H_j}{L_j(a_j)}=\frac{1-\rho-u}{1+\rho}.
\]
After this change of variables, $L_j(a_j)/(H_j+1)=1+\rho$ and $\tilde H_j/(H_j+1)=1-\rho-u$.

\subsection{A Moment Inequality}
\begin{lemma}[{\sc Moment Inequality}]\label{lem:overload-deficit}
In the notation above, for integers $j\ge1$ and $\mu\ge j$,
\begin{equation}\label{eq:overload-deficit}
\mathcal Z:=\frac{2{\vartheta}}{7}-\kappa^2>0,\qquad L_j(a_j)>3\tilde H_j.
\end{equation}
\end{lemma}
\proof{Proof.} The assumption $\mu\ge3$ is not needed for the crossing argument of the current proof. Fix $j$ and differentiate with respect to $\mu$, writing $f'$ for $df/d\mu$. The identities ${\vartheta}=1-e_j(H_j+1)$, $L_j(a_j)=j+1-\mu+\mu/(H_j+1)$, and $\tilde H_j=2H_j+1-L_j(a_j)$, together with~\eqref{eq:erlang-derivatives}, give
\[
\begin{aligned}
H_j'&=-{\vartheta},& \frac{d}{d\mu}L_j(a_j)&=-\frac{\tilde H_j}{H_j+1},\\
\tilde H_j'&=\frac{\tilde H_j}{H_j+1}-2{\vartheta},& {\vartheta}'&=-\frac{\tilde H_j(1-{\vartheta})}\mu.
\end{aligned}
\]
For example, ${\vartheta}'=-e_j'(H_j+1)-e_jH_j'=e_j({\vartheta}-H_j(H_j+1)/\mu)=-e_j(H_j+1)\tilde H_j/\mu$. We now express the derivatives of ${\vartheta}$ and $\kappa$ in the variables $\rho,\kappa,{\vartheta}$, multiplying each by $H_j+1$ to clear denominators. Since $\mu=v_j/{\vartheta}=\rho(H_j+1)^2/{\vartheta}$ and $\tilde H_j=\kappa L_j(a_j)=\kappa(1+\rho)(H_j+1)$, the formula for ${\vartheta}'$ becomes
\[
(H_j+1){\vartheta}'=-\frac{(H_j+1)\tilde H_j(1-{\vartheta})}{\mu}=-\frac{{\vartheta}(1-{\vartheta})\kappa(1+\rho)}\rho.
\]
By the quotient rule and the formulas for $\tilde H_j'$ and the derivative of $L_j(a_j)$,
\[
\begin{aligned}
(H_j+1)\kappa'&=(H_j+1)\left[\frac{\tilde H_j'}{L_j(a_j)}-\frac{\tilde H_j}{L_j(a_j)^2}\frac{d}{d\mu}L_j(a_j)\right]\\
&=\frac{\tilde H_j}{L_j(a_j)}-\frac{2{\vartheta}(H_j+1)}{L_j(a_j)}+\frac{\tilde H_j^2}{L_j(a_j)^2}=\kappa(1+\kappa)-\frac{2{\vartheta}}{1+\rho}.
\end{aligned}
\]

Since $2v_j/(H_j+1)=L_j(a_j)-\tilde H_j-1$, \eqref{eq:overload-domain} gives, for every $\mu\ge j$,
\[
{\vartheta}\le\frac{L_j(a_j)-\tilde H_j-1}{2(L_j(a_j)-1)}=\frac{1-\kappa}{2}-\frac{\kappa}{2(L_j(a_j)-1)}\le\frac{1-\kappa}{2},
\]
with strict inequality when $\kappa>0$. We show that $\mathcal Z$ can only cross zero downwards as $\mu$ increases. Suppose that $\mathcal Z(\mu')=0$ for some $\mu'\ge j$. Then ${\vartheta}(\mu')=7\kappa(\mu')^2/2$, so $\kappa(\mu')>0$, and the displayed bound is strict at $\mu'$: $7\kappa(\mu')^2/2={\vartheta}(\mu')<(1-\kappa(\mu'))/2$, that is, $7\kappa(\mu')^2+\kappa(\mu')<1$. In particular $\kappa(\mu')<6/19$, because $x\mapsto7x^2+x$ is increasing on $[0,\infty)$ and $7(6/19)^2+6/19=366/361>1$. For every $\mu$, the derivative formulas give
\[
(H_j+1)\mathcal Z'=-\frac{2{\vartheta}(1-{\vartheta})\kappa(1+\rho)}{7\rho}
-2\kappa^2(1+\kappa)+\frac{4{\vartheta}\kappa}{1+\rho}.
\]
At $\mu'$, replacing ${\vartheta}(\mu')$ by $7\kappa(\mu')^2/2$ and collecting the three terms gives
\[
(H_j(\mu')+1)\mathcal Z'(\mu')=-\frac{7\kappa(\mu')^2}{2\rho(\mu')(1+\rho(\mu'))}\,
\mathcal P_{\rho(\mu')}(\kappa(\mu')),
\]
where, for $0<\rho<1$,
\[
\mathcal P_\rho(\kappa):=\kappa\left(\frac27-\kappa^2\right)(1+\rho)^2
+\frac47\rho(1+\rho)(1+\kappa)-4\kappa\rho.
\]
With $\rho$ held fixed, $\mathcal P_\rho''(\kappa)=-6\kappa(1+\rho)^2\le0$, so it suffices to check positivity of $\mathcal P_\rho$ at the endpoints of $[0,6/19]$. Both endpoint values are positive for $0<\rho<1$, the second after completing the square:
\[
\mathcal P_\rho(0)=\frac47\rho(1+\rho)>0,\qquad
\mathcal P_\rho(6/19)=\frac4{48013}
\left[9730\left(\rho-\frac14\right)^2+138\rho+\frac{775}{8}\right]>0.
\]
Thus, $\mathcal Z'(\mu')<0$ at every zero $\mu'$ of $\mathcal Z$.

To determine the sign of $\mathcal Z$, let $\mu\to\infty$ with $j$ fixed. The finite sum $A_j=\sum_{l=0}^j(j)_l\mu^{-l}$ gives $A_j=1+j/\mu+O(\mu^{-2})$, and $\Prb(I_j=l)=(j)_l\mu^{-l}/A_j$. The event $I_j=1$ contributes $j/\mu+O(\mu^{-2})$ to both $\E[I_j]$ and $\E[I_j^2]$; all values $I_j\ge2$ contribute only $O(\mu^{-2})$. Hence, $H_j={j/\mu}+O(\mu^{-2}), \E[I_j^2]={j/\mu}+O(\mu^{-2}), v_j={j/\mu}+O(\mu^{-2})$. Therefore, ${\vartheta}=j/\mu^2+O(\mu^{-3})$, $\tilde H_j=O(\mu^{-2})$, and $\kappa=O(\mu^{-2})$, so $\mathcal Z>0$ for all sufficiently large $\mu$. An earlier nonpositive value would require a subsequent upward crossing, contrary to the negative derivative at every zero. This proves $\mathcal Z>0$ on $[j,\infty)$. For $\kappa>0$, combining $7\kappa^2<2{\vartheta}$ with the bound ${\vartheta}<(1-\kappa)/2$ shown above again gives $\kappa<6/19<1/3$; the conclusion is immediate when $\kappa=0$. Recall that $\kappa=\tilde H_j/L_j(a_j)$ and $L_j(a_j)>0$. Hence, $\kappa<1/3$ is equivalent to $L_j(a_j)>3\tilde H_j$, completing~\eqref{eq:overload-deficit}. \Halmos

\subsection{The Penalty at the Intersection}
\begin{lemma}[{\sc Intersection Penalty}]\label{lem:overload-price}
In the notation above, for integers $j\ge1$ and $\mu\ge\max\{j,3\}$,
\begin{equation}\label{eq:overload-price}
a_j>\frac{L_j(a_j)}2+\frac{L_j(a_j)^2}{4\sigma_H},\qquad
\sigma_H:=1-\frac{5L_j(a_j)^2}{28\mu}>0.
\end{equation}
\end{lemma}
\proof{Proof.} The term $L_j(a_j)^2/(4\sigma_H)$ will equal the quadratic term $\widehat{\eta}_j^2/2$ in the upper bound on $D_\mu(\widehat{\eta}_j)$ for the weight chosen in Section~\ref{app:choosing-rank-weight} when $L_j(a_j)>\sqrt2$. We first record bounds that make its denominator positive and reduce the remaining comparison to a concave quadratic.

Since $L_j(a_j)-3\tilde H_j=4v_j/(H_j+1)-2(H_j+1)+3>0$, we have $v_j>(H_j+1)^2/2-3(H_j+1)/4$. Combining this with~\eqref{eq:overload-domain} gives $\mu>3(H_j+1)^2/2-7(H_j+1)/4$. Combining this quadratic bound with $\mu\ge3$ gives
\[
\mu>\frac35\left(\frac32(H_j+1)^2-\frac74(H_j+1)\right)+\frac25\,3
=\frac23(H_j+1)^2+\frac7{30}\left(H_j+1-\frac94\right)^2+\frac3{160}
>\frac23(H_j+1)^2.
\]
The quadratic bound alone need not imply $\mu>2(H_j+1)^2/3$ when $H_j+1$ is small; combining it with $\mu\ge3$ proves the inequality for these small values. This is where the assumption $\mu\ge3$ first enters. Consequently $\mu>2(H_j+1)^2/3$, and ${\vartheta}=v_j/\mu<3v_j/(2(H_j+1)^2)=3\rho/2$. Also, \eqref{eq:overload-domain} gives $\mu\ge H_j(H_j+1)+v_j\ge2v_j$, so $\vartheta=v_j/\mu\le1/2$ and hence $1/\vartheta\ge2$. Averaging this with $1/\vartheta>2/(3\rho)$ gives a single upper bound depending only on $\rho$:
\[
\frac1{\vartheta}>\frac67\frac2{3\rho}+\frac17\,2
=\frac{2(\rho+2)}{7\rho},\qquad
{\vartheta}<\overline{\vartheta}(\rho):=\frac{7\rho}{2(\rho+2)},\qquad \rho=\frac{v_j}{(H_j+1)^2}.
\]
Also $L_j(a_j)=2H_j+1-\tilde H_j\le2H_j+1$, so $H_j+1\ge(L_j(a_j)+1)/2$ and \eqref{eq:overload-domain} implies $\mu\ge(L_j(a_j)^2-1)/2$. Thus, since $\mu\ge3$,
\begin{equation}\label{eq:overload-scale-bounds}
\frac{L_j(a_j)^2}\mu\le2+\frac1\mu\le\frac73,
\qquad \frac7{12}\le\sigma_H<1.
\end{equation}

Here is the algebra linking the moment inequality to the penalty inequality. In the variables $\rho,u,{\vartheta}$, ${a_j/(H_j+1)^2}={(u+\rho)/(1-{\vartheta})}, {L_j(a_j)/(H_j+1)}=1+\rho, \sigma_H=1-{5{\vartheta}(1+\rho)^2/(28\rho)}$. To prove Lemma~\ref{lem:overload-price}, we subtract its proposed lower bound from $a_j$ and multiply by the positive factor $4\sigma_H(1-\vartheta)/(H_j+1)^2$:
\[
\begin{aligned}
\mathcal R
&:=\frac{4\sigma_H(1-{\vartheta})}{(H_j+1)^2}
\left(a_j-\frac{L_j(a_j)}2-\frac{L_j(a_j)^2}{4\sigma_H}\right)\\
&=4\sigma_H(u+\rho)-2\sigma_H(1-{\vartheta})u(1+\rho)
  -(1-{\vartheta})(1+\rho)^2.
\end{aligned}
\]
Recall that $\mathcal Z=2\vartheta/7-\kappa^2$ and $\kappa=(1-\rho-u)/(1+\rho)$. Multiplying the definition of $\mathcal Z$ by $(1+\rho)^2$ and substituting for $\kappa$ gives $(1+\rho)^2\mathcal Z ={2\vartheta(1+\rho)^2/7}-(1+\rho)^2\kappa^2 ={2\vartheta(1+\rho)^2/7}-(1-\rho-u)^2$. We subtract this positive quantity from $\mathcal R$. The terms independent of ${\vartheta}$ reduce to $u^2$, using $4\rho-(1+\rho)^2=-(1-\rho)^2$; the remaining terms have the common factor displayed here:
\[
\begin{aligned}
\mathcal R-(1+\rho)^2\mathcal Z
&=u^2+\frac{5u{\vartheta}(1+\rho)}{14\rho}
\left[\rho^2+\frac{28}{5}\rho-1-{\vartheta}(1+\rho)^2\right]=\frac{5u}{14\rho}\mathcal E({\vartheta}),
\end{aligned}
\]
where $\mathcal E({\vartheta}):={(14/5)}\rho u+{\vartheta}(1+\rho) \left[\rho^2+{(28/5)}\rho-1-{\vartheta}(1+\rho)^2\right]$. Since $\mathcal Z>0$ and $5u/(14\rho)>0$, proving $\mathcal E(\vartheta)>0$ will establish $\mathcal R>0$. For fixed $\rho,u$, $\mathcal E''({\vartheta})=-2(1+\rho)^3<0$, so positivity on $[0,\overline{\vartheta}(\rho)]$ follows from its two endpoint values. At zero, $\mathcal E(0)=14\rho u/5>0$. For the other endpoint, dividing $L_j(a_j)>3\tilde H_j$ by $H_j+1$ gives $1+\rho>3(1-\rho-u)$, that is, $u>u_0:=(2-4\rho)/3$. Because $\mathcal E$ is increasing in $u$,
\[
\begin{aligned}
\mathcal E(\overline{\vartheta}(\rho))
&>\frac{14}{5}\rho u_0+
\frac{7\rho(1+\rho)}{2(\rho+2)}
\left[\rho^2+\frac{28}{5}\rho-1
-\frac{7\rho(1+\rho)^2}{2(\rho+2)}\right]\\
&=\frac{7\rho}{60(\rho+2)^2}
\left(4+77\rho+107\rho^2-89\rho^3-75\rho^4\right)>0.
\end{aligned}
\]
The last sign follows from the exact positive decomposition $4+77\rho+107\rho^2-89\rho^3-75\rho^4 =4+20\rho+57\rho(1-\rho) +\rho^2(1-\rho)(164+75\rho)>0$, valid for $0<\rho<1$. This endpoint evaluation is purely algebraic; $u_0$ need not be positive. Thus, $\mathcal E({\vartheta})>0$, then $\mathcal R>0$, which proves~\eqref{eq:overload-price}. \Halmos

\subsection{Choosing the Rank Weight}\label{app:choosing-rank-weight}
\proof{Proof of Lemma~\ref{lem:endpoints} part~(ii).} Write $g_j:=L_j(a_j)/\sqrt2$. Recall from \eqref{eq:concave-lower-bound} that $\mathcal L_\mu(a_j)$ is the supremum of $\min\{B_\mu(\eta),a_j-D_\mu(\eta)\}$ over all $\eta\ge1$. Thus, it suffices to choose a weight $\widehat{\eta}_j\ge1$ for which $B_\mu(\widehat{\eta}_j)\ge g_j$ and $a_j-D_\mu(\widehat{\eta}_j)\ge g_j$, giving $\mathcal L_\mu(a_j)\ge g_j$. Since $L_j(a_j)=L_{j+1}(a_j)$ by Appendix~\ref{app:erlang-intersections}, this gives $L_j(a_j)=L_{j+1}(a_j)\le\sqrt2\,\mathcal L_\mu(a_j)$, which is the lemma's claim.

We first handle $g_j\le1$, for which we choose $\widehat{\eta}_j:=1$. Then $B_\mu(1)=1$ and $D_\mu(1)=0$. Moreover, $a_j=(H_j+1+v_j)/(1-\vartheta)$, with $H_j>0$, $v_j>0$, and $0<\vartheta\le1/2$, so $a_j >H_j+1+v_j >H_j+1+{v_j/(H_j+1)} =L_j(a_j)$. This comparison verifies the second required inequality for $\widehat{\eta}_j=1$: since $D_\mu(1)=0$ and $L_j(a_j)=\sqrt2\,g_j$, $a_j-D_\mu(1)=a_j>L_j(a_j)=\sqrt2\,g_j>g_j$. Together with $B_\mu(1)=1\ge g_j$, it gives $\mathcal L_\mu(a_j) \ge\min\{B_\mu(1),a_j-D_\mu(1)\} =\min\{1,a_j\}\ge g_j$, which proves the lemma's claim for $g_j\le1$.

Suppose now that $g_j>1$. The lower bound on $G(s)$ in~\eqref{eq:overload-global-profile} will give the rank bound $B_\mu(\eta)>\eta/\sqrt{1+5\eta^2/(14\mu)}$ in the range used below. We choose $\widehat{\eta}_j$ to make its right-hand side equal to $g_j$, which requires $\widehat{\eta}_j^2\left(1-{5g_j^2/(14\mu)}\right)=g_j^2$. Substituting $g_j^2=L_j(a_j)^2/2$ into the equation for $\widehat{\eta}_j$ gives $1-5g_j^2/(14\mu)=1-5L_j(a_j)^2/(28\mu)=\sigma_H$, where the last equality uses the definition of $\sigma_H$ in~\eqref{eq:overload-price}. We have already established $\sigma_H>0$, and its definition gives $\sigma_H<1$. Thus, the chosen weight is well defined and satisfies $\widehat{\eta}_j:={g_j/\sqrt{\sigma_H}}>g_j>1$. The bounds in~\eqref{eq:overload-scale-bounds} give ${\widehat{\eta}_j^2/\mu}={L_j(a_j)^2/(2\mu\sigma_H)}\le\frac{7/3}{2(7/12)}=2$. For $0<x\le2$, $169\left(1+{5x/14}\right)-(13+2x)^2 ={x(117-56x)/14}>0$. Taking $x=\widehat{\eta}_j^2/\mu$ and applying~\eqref{eq:poisson-gaussian-rank} and~\eqref{eq:overload-global-profile} gives
\[
B_\mu(\widehat{\eta}_j)>\sqrt\mu\,G\!\left(\frac{\widehat{\eta}_j}{\sqrt\mu}\right)
>\frac{13\widehat{\eta}_j}{13+2\widehat{\eta}_j^2/\mu}
>\frac{\widehat{\eta}_j}{\sqrt{1+5\widehat{\eta}_j^2/(14\mu)}}=g_j.
\]
The last equality is precisely the equation defining $\widehat{\eta}_j$. For the stockout requirement, the schedule estimate \eqref{eq:simple-penalty}, $\widehat{\eta}_j\ge g_j$, and $\sqrt2/3<1/2$ give
\[
\begin{aligned}
g_j+D_\mu(\widehat{\eta}_j)
&\le g_j+\frac{\widehat{\eta}_j^2}{2}-\frac{\widehat{\eta}_j}{3}
\le\frac23g_j+\frac{\widehat{\eta}_j^2}{2}=\frac{\sqrt2}{3}L_j(a_j)+\frac{L_j(a_j)^2}{4\sigma_H}
<\frac{L_j(a_j)}2+\frac{L_j(a_j)^2}{4\sigma_H}<a_j.
\end{aligned}
\]
The last inequality is Lemma~\ref{lem:overload-price}. Therefore, $\min\{B_\mu(\widehat{\eta}_j),a_j-D_\mu(\widehat{\eta}_j)\}\ge L_j(a_j)/\sqrt2$ for every integer $j\ge1$ and $\mu\ge\max\{j,3\}$. \Halmos

\section{Tightness of the Base-Stock Guarantee}\label{app:sharpness-proof}
This appendix provides the feasible-policy upper bound and the uniform base-stock lower bound used in Section~\ref{sec:article-sharpness}.

We first bound the unrestricted optimum using equally spaced unit orders. Fix $0<\ell<1$, define the normalized order rate $\lambda_{\rm ord}:=1-\ell$, and order one unit every $1/\lambda_{\rm ord}$ time units from an empty system. This open-loop policy fixes all order times in advance and is admissible. All processes in the following calculation refer to this policy. The lead time only delays the first delivery. Measure time from that delivery, and let $\varsigma(t)\in[0,1)$ be the elapsed fraction of the current delivery interval. Recall that $X(t)$ is the on-hand inventory at time $t$. The process $Z(t):=X(t)+\varsigma(t)-1$ starts at zero and has no delivery jump. At each subsequent delivery time $t$, $\varsigma$ approaches one from below and resets to zero: its left-hand limit is $\varsigma(t-)=1$, while $\varsigma(t)=0$. Thus, its jump is $-1$, offsetting the unit increase in $X$. Each delivery interval has length $1/\lambda_{\rm ord}$, so its elapsed fraction increases at rate $\dot\varsigma(t)=\lambda_{\rm ord}$ between deliveries. Between demand and delivery events, $X$ is constant, so $Z=X+\varsigma-1$ increases at the same rate. Each served demand decreases $Z$ by one.

The conditional drift of $Z^2$ is its expected change per unit time, given the current state. Between events, the increase of $Z$ contributes $2\lambda_{\rm ord} Z$ to this rate. A served demand changes $Z^2$ by $(Z-1)^2-Z^2=-2Z+1$ and occurs at rate one when $X>0$. Deliveries leave $Z^2$ unchanged. Adding these contributions and using $Z=\varsigma-1\le0$ when $X=0$ gives
\[
 2\lambda_{\rm ord} Z+(-2Z+1)\mathbf1_{\{X>0\}}
 =-2\ell Z+\mathbf1_{\{X>0\}}+2Z\mathbf1_{\{X=0\}}
 \le-2\ell Z+\mathbf1_{\{X>0\}}.
\]
The term $-2\ell Z$ comes from the chosen order rate $\lambda_{\rm ord}=1-\ell<1$: while inventory is positive, $Z$ decreases in expectation at rate $\ell$.

Write $D(T)$ for the total number of units delivered by time $T$, including the unit delivered at $t=0$, if any. Recall that $\mathcal A(T)$ counts the demands served by time $T$. Then $D(T)=\lambda_{\rm ord} T+O(1)$ and pathwise flow balance gives $\mathcal A(T)=D(T)-X(T)\le D(T)$. For each finite $T$, the demand integrand $(-2Z(t-)+1)\mathbf1_{\{X(t-)>0\}}$ is bounded and predictable, since $X(t)\le D(T)$ on $[0,T]$.

Sum the continuous changes and demand jumps of $Z^2$ over $[0,T]$ and take expectations. Lemma~\ref{lem:poisson-compensation} converts the expected sum of demand jumps into a time integral, giving the first equality below; the drift inequality above gives the integrated drift inequality:
\[
\begin{aligned}
\E Z(T)^2-\E Z(0)^2
&=\E\int_0^T\bigl[2\lambda_{\rm ord}Z(t)+(-2Z(t)+1)\mathbf1_{\{X(t)>0\}}\bigr]\,dt\\
&\le-2\ell\,\E\int_0^T Z(t)\,dt+\E\int_0^T\mathbf1_{\{X(t)>0\}}\,dt=-2\ell\,\E\int_0^T Z(t)\,dt+\E\,\mathcal A(T).
\end{aligned}
\]
The last equality applies Lemma~\ref{lem:poisson-compensation} to the indicator $\mathbf1_{\{X(t-)>0\}}$, whose sum over demand epochs counts the served demands. Substituting $Z(0)=0$, rearranging, and using $\E Z(T)^2\ge0$ and $\mathcal A(T)\le D(T)=\lambda_{\rm ord}T+O(1)$ gives $2\ell\,\E\int_0^T Z(t)\,dt \le\E\,\mathcal A(T)-\E Z(T)^2\le \lambda_{\rm ord} T+O(1)$. Between consecutive deliveries, $1-\varsigma(t)$ decreases linearly from one toward zero, with average $1/2$ over each complete interval. Thus, $T^{-1}\int_0^T(1-\varsigma(t))\,dt\to1/2$. Combining this limit with $X=Z+1-\varsigma$ and the preceding bound gives, using $\lambda_{\rm ord}+\ell=1$ in the last step,
\[
\limsup_{T\to\infty}\frac1T\E\int_0^T X(t)\,dt\le\frac{\lambda_{\rm ord}}{2\ell}+\frac12=\frac1{2\ell}.
\]
To identify the loss rate, observe that $Z(t)=X(t)+\varsigma(t)-1\ge-1$ because $X(t)$ and $\varsigma(t)$ are nonnegative. Also, served demands cannot exceed total demands, so $\E\,\mathcal A(T)\le\E N(T)=T$. Substituting these bounds into the integrated drift inequality, with $Z(0)=0$, gives $\E Z(T)^2\le-2\ell\,\E\int_0^T Z(t)\,dt+\E\,\mathcal A(T) \le2\ell T+T=(2\ell+1)T$. Since $X(T)=Z(T)+1-\varsigma(T)\le Z(T)+1$, taking expectations and applying $(\E Z(T))^2\le\E[Z(T)^2]$, gives
\[
\E X(T)\le1+\E Z(T)\le1+\sqrt{\E[Z(T)^2]}\le1+\sqrt{(2\ell+1)T}=O(\sqrt T).
\]
Here $\ell$ is fixed, so the bound grows at most proportionally to $\sqrt T$ and implies $\E X(T)/T\to0$. Consequently, flow balance yields $\E\,\mathcal A(T)/T=D(T)/T-\E X(T)/T\to\lambda_{\rm ord}$ and $\E\,\mathcal U(T)/T=1-\E\,\mathcal A(T)/T\to\ell$. Before the first delivery, the empty system incurs only a finite expected loss cost, which does not change either average.

Thus, for every $0<\ell<1$, this equally spaced ordering policy has average cost at most $1/(2\ell)+a\ell$. For each $a>1/2$, choose $\ell=\ell_a:=1/\sqrt{2a}$, which lies in $(0,1)$. The optimal system cost is at most the cost of this particular admissible policy, so
\begin{equation}\label{eq:article-cycle-bound}
 \operatorname{OPT}_0(\mu,a)\le\frac1{2\ell_a}+a\ell_a=\sqrt{2a}
 \qquad(a>1/2).
\end{equation}

For the reverse comparison, the Erlang identity and the variance bound give $e_j(H_j+1)=1-v_j/\mu, v_j\le H_j(H_j+1), e_j\ge{1/(H_j+1)}-{H_j/\mu}$. If $0<a<\mu$ and $\omega:=1-a/\mu$, the formula $L_j(a)=H_j+ae_j$ and the arithmetic--geometric mean inequality therefore give, for every integer base-stock level $j$,
\[
 L_j(a)\ge\omega H_j+\frac a{H_j+1}
 =\omega(H_j+1)+\frac a{H_j+1}-\omega
 \ge2\sqrt{a\omega}-\omega.
\]

Take $\mu=a^2$ and let $a\to\infty$. The lower bound is uniform in $j$, so \eqref{eq:article-cycle-bound} implies
\[
 \frac{U_{a^2}(a)}{\operatorname{OPT}_0(a^2,a)}
 \ge\frac{2\sqrt{a-1}-1+a^{-1}}{\sqrt{2a}}
 \longrightarrow\sqrt2.
\]
Together with the zero-setup upper bound, this proves \eqref{eq:article-sharp} along a single sequence of finite parameters.

The lower endpoint also applies to the algorithm with positive setup cost. Fix $k=1$, so $Q_k=1$ and $b_k=1$, and take $\mu=a^2$. The zero-setup guarantee proved in Section~\ref{sec:interval-coverage} and the feasible-policy bound~\eqref{eq:article-cycle-bound} give $U_{a^2}(a)\le\sqrt2\operatorname{OPT}_0(a^2,a)\le2\sqrt a$. The selection rule in Corollary~\ref{cor:policy-selection} compares the replenishment cost upper bound $U_{a^2}(a)+1$ with the no-replenishment cost $a$. Since $U_{a^2}(a)+1\le2\sqrt a+1<a$ for $a>3+2\sqrt2$, it selects replenishment for all sufficiently large $a$. Base stock zero carries no inventory and loses every demand, so its normalized cost is $L_0(a)=a$. For $a>3+2\sqrt2$, the preceding inequality $U_{a^2}(a)+1<a$ implies $U_{a^2}(a)<L_0(a)$, so a cost-minimizing base-stock level must satisfy $S^\star\ge1$. Because $Q_k=1$, the selected policy is base stock. The candidate's full cost is at least $U_{a^2}(a)$ because setup cost is nonnegative. The evenly spaced unit-order policy with loss fraction $\ell_a$ incurs setup cost at rate $1-\ell_a\le1$, so $\operatorname{OPT}_1(a^2,a)\le\sqrt{2a}+1$. Consequently,
\[
\frac{C_{a^2,a,1}(\pi_{\rm heuristic})}{\operatorname{OPT}_1(a^2,a)}
\ge\frac{2\sqrt{a-1}-1+a^{-1}}{\sqrt{2a}+1}
\longrightarrow\sqrt2.
\]
This statement concerns the policy selected by our algorithm; the optimized $(r,Q)$ family may select $Q>1$.

\section{Finite-Horizon Bounds for Batching}\label{app:batch-proofs}
The proofs of the bounds~\eqref{eq:batch-credit-main} and~\eqref{eq:batch-upper-main} in Section~\ref{sec:batching} use bounded functions of inventory to relate changes at sales, orders, and deliveries to accumulated costs. The first bounds holding and setup costs in terms of the expected number of units served under every admissible policy. The second gives a cost identity for the constructed $(r,Q)$ policy and compares its sales with base stock.

\subsection{A Holding and Setup Cost Bound for Every Policy}\label{app:batch-lower}
\begin{lemma}[{\sc Service Credit}]\label{lem:batch-service-credit}
For $0\le\theta\le1$, write $b:=\sqrt{2k\theta}$. Every empty-start admissible policy satisfies
\begin{equation}\label{eq:batch-service-credit}
 \theta\,\E\int_0^T X(t)\,dt+k\E O(T)\ge b\,\E\,\mathcal A(T).
\end{equation}
\end{lemma}
\proof{Proof.} The case $\theta=0$ is immediate. For $\theta>0$, define the following function of on-hand inventory $n$: $\mathcal F(n):=-\sum_{i=1}^n(b-\theta i)^+$. The decreasing summands satisfy
\[
 \sum_{i\ge1}(b-\theta i)^+
 \le\int_0^\infty(b-\theta x)^+\,dx=\frac{b^2}{2\theta}=k
 \quad\Longrightarrow\quad -k\le\mathcal F\le0.
\]
The summand $(b-\theta i)^+$ is the amount by which $b$ exceeds allocated holding cost $\theta i$, or zero if $\theta i\ge b$. Accordingly, at a served demand with pre-demand stock $n\ge1$,
\[
 \theta n+\mathcal F(n-1)-\mathcal F(n)=\theta n+(b-\theta n)^+\ge b
 \quad\Longrightarrow\quad \mathcal F(n-1)-\mathcal F(n)\ge b-\theta n.
\]
Let $\Delta_{\rm del}\mathcal F$ denote the jump at a received order. As shown above, the values of $\mathcal F$ before and after a delivery both lie in $[-k,0]$, so a delivery can reduce $\mathcal F$ by at most $k$: $\Delta_{\rm del}\mathcal F\ge-k$. An empty-start system receives at most $O(T)$ orders by time $T$, so
\[
 \sum_{\rm deliveries\le T}\Delta_{\rm del}\mathcal F\ge-kO(T)
 \quad\Longrightarrow\quad
 \E\sum_{\rm deliveries\le T}\Delta_{\rm del}\mathcal F\ge-k\E O(T).
\]
Take the demand increment $\mathcal F(n-1)-\mathcal F(n)$ to be zero when $n=0$. This integrand and the service indicator are bounded and predictable, so telescoping and Lemma~\ref{lem:poisson-compensation} give
\[
 \begin{split}
 \E \mathcal F(X(T))
 &=\E\int_0^T
    \{\mathcal F(X(t-)-1)-\mathcal F(X(t-))\}\mathbf1_{\{X(t-)>0\}}\,dt
      +\E\sum_{\rm deliveries\le T}\Delta_{\rm del}\mathcal F\\
 &\ge b\,\E\,\mathcal A(T)-\theta\,\E\int_0^T X(t)\,dt-k\E O(T).
 \end{split}
\]
Since $\mathcal F(X(T))\le0$, this proves \eqref{eq:batch-service-credit}. For $k>0$, our admissibility conditions require finite expected setup cost over every finite time horizon, so $k\E O(T)<\infty$ and hence $\E O(T)<\infty$ for each finite $T$. Since each delivery changes $\mathcal F$ by at most $k$ in absolute value, the expected sum of delivery jumps is finite. For $k=0$, $b=0$ and $\mathcal F=0$, so the claim is immediate. \Halmos

For $\theta<1$ and $a\ge b$, adding the remaining costs and using $\E\,\mathcal A(T)+\E\,\mathcal U(T)=T$ yields
\[
 \E\left[\int_0^T X(t)\,dt+a\,\mathcal U(T)+kO(T)\right]
 \ge bT+(1-\theta)\E\left[
       \int_0^T X(t)\,dt+\frac{a-b}{1-\theta}\,\mathcal U(T)\right].
\]
Dividing by $T$, taking the limsup with the holding and lost-sales terms kept together, and then taking the infimum over policies proves
\begin{equation}\label{eq:operating-cost-transform}
 \operatorname{OPT}_k(\mu,a)
 \ge b+(1-\theta)\operatorname{OPT}_0
       \left(\mu,\frac{a-b}{1-\theta}\right)
 \quad(a\ge b).
\end{equation}
If $a<b$, the same inequality before rescaling and $\E\,\mathcal U(T)\le T$ give the lower bound $a$.

At $\theta=1$, $b=\sqrt{2k}$ and \eqref{eq:batch-service-credit} bounds all holding and setup costs below by $b\,\E\,\mathcal A(T)$. Adding the lost-sales cost gives
\[
\begin{aligned}
\E\left[\int_0^T X(t)\,dt+a\,\mathcal U(T)+kO(T)\right]
&\ge b\,\E\,\mathcal A(T)+a\,\E\,\mathcal U(T)\\
&\ge\min\{a,b\}\bigl(\E\,\mathcal A(T)+\E\,\mathcal U(T)\bigr)
=\min\{a,b\}T.
\end{aligned}
\]
The last equality uses that every demand is either served or lost and expected total demand is $T$. Dividing by $T$ and taking the limsup shows that every empty-start policy has average cost at least $\min\{a,b\}$. The initial-supply argument extends these bounds to all admissible policies, giving $\operatorname{OPT}_k(\mu,a)\ge\min\{a,b\}$ at $\theta=1$. We take the limsup of the total expected cost per unit time. The argument remains valid even if the expected holding, lost-sales, and setup costs per unit time do not converge separately.

\subsection{The Classical-Policy Cost Upper Bound}\label{app:batch-upper}
Fix an integer base-stock level $S\ge1$ and an integer lot size $Q\ge1$. Retain the reorder point $r=S-1$ and use the policy $\pi_{\rm rQ}(S-1,Q)$, initialized with $S+Q-1$ on-hand units. Let $Y$ denote inventory position and write $b_Q:={(Q-1)/2}+{k/Q}, d_Q:={k/Q}-{(Q-1)/2}, \gamma:=S-\mu+b_Q$. For each pending unit $i$, let $\tau_i(t)$ be its remaining delivery time. Define the total remaining delivery time $\mathcal W(t)$ and the function $\Psi(t)$ by
\[
 \mathcal W(t):=\sum_{i\ \mathrm{pending}}\tau_i(t),\qquad
 \Psi(t):=\frac{Y(t)(Y(t)+1)}2-\gamma Y(t)-\mathcal W(t).
\]
Adding the term $-\mathcal W(t)$ to the quadratic part of $\Psi(t)$ converts the holding-cost rate from $X(t)$ to $Y(t)$ in the calculation: between events, the rate of change of the added term is $Y(t)-X(t)$. The quadratic part $Y(t)(Y(t)+1)/2-\gamma Y(t)$ is chosen so that, while inventory is positive, holding cost plus the expected rate of change in $\Psi$ from sales and approaching deliveries equals $\gamma$. We choose $\gamma$ so that the change in $\Psi$ when an order is placed offsets its setup cost.

As pending deliveries approach, $\mathcal W'=-(Y-X)$, so holding cost plus the rate of change of $-\mathcal W$ is $Y$. At a served demand, $Y$ falls to $Y-1$ and $\mathcal W$ is unchanged. Thus, the change in $\Psi$ caused by the sale is
\[
 \left[\frac{(Y-1)Y}{2}-\gamma(Y-1)-\mathcal W\right]
 -\left[\frac{Y(Y+1)}{2}-\gamma Y-\mathcal W\right]
 =\gamma-Y.
\]
While $X>0$, every demand is served and demand has rate one, so the conditional expected change in $\Psi$ from sales per unit time is $1\cdot(\gamma-Y)=\gamma-Y$. Any order triggered by a sale contributes a separate change in $\Psi$. An order takes position from $S-1$ to $S+Q-1$ and increases the total remaining delivery time by $\mu Q$, so its setup cost plus the change in $\Psi$ is $k+Q\{S+(Q-1)/2-\gamma-\mu\}=0$. Deliveries do not change $Y$ and the arriving units have zero remaining delivery time, so deliveries cause no jump in $\Psi$. Define the expected full cost through $T$ by $J_T:=\E\left[\int_0^T X(t)\,dt+a\,\mathcal U(T)+kO(T)\right]$. Adding the expected change in $\Psi$ to $J_T$ cancels each setup charge with the jump in $\Psi$ at that order, as shown above. When $X>0$, holding cost has rate $X$, approaching deliveries increase $\Psi$ at rate $Y-X$, and rate-one sales contribute $\gamma-Y$, giving $X+(Y-X)+(\gamma-Y)=\gamma$. When $X=0$, lost demands leave $\Psi$ unchanged; their cost has rate $a$, while approaching deliveries increase $\Psi$ at rate $Y$, giving $a+Y$. The demand integrands are bounded and predictable because $S\le Y\le S+Q-1$, so Lemma~\ref{lem:poisson-compensation} gives
\[
 J_T+\E\Psi(T)-\Psi(0)
 =\E\int_0^T\bigl[\gamma\mathbf1_{\{X(t)>0\}}+(a+Y(t))\mathbf1_{\{X(t)=0\}}\bigr]\,dt.
\]
Using $\mathbf1_{\{X(t)>0\}}+\mathbf1_{\{X(t)=0\}}=1$ and rearranging terms gives the exact identity
\begin{equation}\label{eq:batch-exact-cost-identity}
 J_T=\gamma T+\E\int_0^T
       (a+Y(t)-\gamma)\mathbf1_{\{X(t)=0\}}\,dt
       +\Psi(0)-\E\Psi(T).
\end{equation}
Here $S\le Y\le S+Q-1$ and $0\le \mathcal W\le\mu(S+Q-1)$, so $\Psi$ is bounded. Substituting $\gamma=S-\mu+b_Q$ and the definitions of $b_Q$ and $d_Q$ gives $a+S+Q-1-\gamma =a+\mu+Q-1-b_Q =a+\mu+{(Q-1)/2}-{k/Q} =a+\mu-d_Q$. Write $\ell_T:=\E\,\mathcal U(T)/T$ for the expected lost-sales rate over $[0,T]$. Since expected demand has rate one, $\ell_T=T^{-1}\E\int_0^T\mathbf1_{\{X(t)=0\}}\,dt$, the expected fraction of this horizon spent out of stock. Using $Y\le S+Q-1$ in the exact identity above therefore gives
\begin{equation}\label{eq:batch-cost-identity}
 J_T/T\le\gamma+(a+\mu-d_Q)\ell_T+O(T^{-1}).
\end{equation}

To bound $\ell_T$, couple the batched policy with a stationary level-$S$ base-stock system on the same future demand path, and denote their sales since time zero by $\mathcal A_Q$ and $\mathcal A_1$. The reference system has $S$ units in total, some possibly pending at time zero. For $\mu>0$, consider an arbitrary demand arrival time $t$ and write $t_0:=(t-\mu)^+$. Supply accounting just before this demand gives
\[
 \begin{split}
 X_Q(t-)
 &=S+Q-1-\mathcal A_Q(t-)+Q\left\lfloor \mathcal A_Q(t_0)/Q\right\rfloor\ge S-\mathcal A_Q(t-)+\mathcal A_Q(t_0),\\
 X_1(t-)&\le S-\mathcal A_1(t-)+\mathcal A_1(t_0).
 \end{split}
\]
The first inequality uses $Q\lfloor n/Q\rfloor\ge n-(Q-1)$. The second allows some of the reference's initial supply still to be pending. Almost surely no positive-lead-time delivery coincides with a demand in either system: new deliveries occur a fixed positive time after past demand epochs, and the stationary reference's initial residual delivery times have continuous probability distributions and are independent of future demand. Thus, these pre-demand inventories determine service level.

Induct over demand epochs. If $\mathcal A_Q(t-)=\mathcal A_1(t-)$, the previously established comparison at the strictly earlier time $t_0$ gives $\mathcal A_Q(t_0)\ge \mathcal A_1(t_0)$, hence $X_Q(t-)\ge X_1(t-)$. The batched policy then serves whenever the reference does. If the sales counts are unequal, their integer difference is at least one, and a single demand cannot reverse the inequality. Starting with both counts zero proves $\mathcal A_Q(T)\ge \mathcal A_1(T)$ for every $T$. For $\mu=0$, both policies serve every demand. Applying Lemma~\ref{lem:poisson-compensation} to the stationary reference's bounded predictable service indicator gives $\E\,\mathcal A_1(T)=(1-e_S)T$. Therefore,
\begin{equation}\label{eq:batch-sales-comparison}
 \ell_T=1-\frac{\E\,\mathcal A_Q(T)}T\le e_S
 \qquad\text{for every }T>0.
\end{equation}

Choose $ \chi:=\sqrt{2k}, Q_k:=\max\{1,\lfloor\chi\rfloor\} $. Set $Q=Q_k$ in the preceding definitions, so $b_Q=b_k$ and $\gamma=S-\mu+b_k$. Then
\begin{equation}\label{eq:integer-lot-surcharge}
 d_{Q_k}\ge0,\qquad b_k\le\chi.
\end{equation}
For $\chi<1$ these follow from $Q_k=1$ and $k=\chi^2/2\le\chi$; otherwise $0\le\chi-Q_k<1$ gives $b_k=\chi-1/2+(\chi-Q_k)^2/(2Q_k)\le\chi$, because $(\chi-Q_k)^2<1\le Q_k$, and $d_{Q_k}=[\chi^2-Q_k(Q_k-1)]/(2Q_k)\ge0$, because $\chi\ge Q_k$. Substituting $\gamma=S-\mu+b_k$ into \eqref{eq:batch-cost-identity} gives the first inequality below. The second uses $d_{Q_k}\ge0$ and $\ell_T\ge0$; the third uses $\ell_T\le e_S$ from \eqref{eq:batch-sales-comparison} and $a+\mu\ge0$:
\[
\begin{aligned}
 J_T/T
 &\le S-\mu+b_k+(a+\mu-d_{Q_k})\ell_T+O(T^{-1})\\
 &\le S-\mu+b_k+(a+\mu)\ell_T+O(T^{-1})\le S-\mu+b_k+(a+\mu)e_S+O(T^{-1}).
\end{aligned}
\]
Taking the limsup as $T\to\infty$, recalling $L_S(a)=S-\mu+(a+\mu)e_S$ from \eqref{eq:base-stock-cost}, and using $b_k\le\chi$ gives
\begin{equation}\label{eq:batch-candidate-bound}
 C_{\mu,a,k}(\pi_{\rm rQ}(S-1,Q_k))
 \le S-\mu+b_k+(a+\mu)e_S=L_S(a)+b_k\le L_S(a)+\chi.
\end{equation}
Minimizing over $S\ge1$ and retaining no order gives the first inequality below. The second follows from the zero-setup bound~\eqref{eq:zero-setup-premise}, proved with $c=\sqrt2$ in Section~\ref{sec:interval-coverage}:
\begin{equation}\label{eq:batch-menu-bound}
 C_{\rm rQ}\le\min\{a,U_\mu(a)+\chi\}
 \le\min\{a,c\operatorname{OPT}_0(\mu,a)+\chi\}.
\end{equation}
This bound is realized by the single policy selected from a minimizing base-stock level $S^\star$ and the displayed $Q_k$, or by no order. When level zero minimizes $U_\mu$, no order already provides the bound.

\subsection{Cost Evaluation and Joint Integer Search}\label{app:numerical-search}

\textbf{An integer batching lower bound.} The finite-sum calculation in Appendix~\ref{app:batch-lower} gives a stronger bound for numerical evaluation if we retain integer lot sizes. Define $\beta_k:=\min_{Q\in\mathbb N}\{(Q+1)/2+k/Q\}$. For every integer $n\ge1$, the definition gives $n\beta_k-n(n+1)/2\le k$. Consequently, $\sum_{i\ge1}(\beta_k-i)^+=\max_{n\ge0}\{n\beta_k-n(n+1)/2\}\le k$. Using $b=\beta_k$ and $\theta=1$ in the proof of Lemma~\ref{lem:batch-service-credit} and the subsequent cost comparison in Appendix~\ref{app:batch-lower} proves $\operatorname{OPT}_k(\mu,a)\ge\min\{a,\beta_k\}$. We include this bound in $\underline C_k(\mu,a)$. For $k=1,10,50,120$, its values before taking the minimum with $a$ are $2,5,10.5,16$, respectively. At $\mu=0$, a minimizing lot with $r=0$ has cost $\beta_k$, while no replenishment has cost $a$, so this lower bound is exact.

\textbf{Evaluating a classical policy.} For the numerical comparison in Section~\ref{sec:numerics}, we optimize both integers $r$ and $Q$ and compare the resulting policy with the heuristic. Fix $\mu$, and let $\ell_{r,Q}$ be the long-run stockout fraction and $H_{r,Q}$ the mean on-hand inventory under $\pi_{\rm rQ}(r,Q)$. Rate-one Poisson demand gives a sales rate of $1-\ell_{r,Q}$ and an order rate of $(1-\ell_{r,Q})/Q$. Therefore,
\[
C_{\mu,a,k}(\pi_{\rm rQ}(r,Q))=H_{r,Q}+a\ell_{r,Q}+\frac{k}{Q}(1-\ell_{r,Q}).
\]
The physical quantities $H_{r,Q}$ and $\ell_{r,Q}$ depend on $\mu,r,Q$, so the same evaluation can be reused across all penalties and setup costs.

At a stockout, inventory position consists entirely of pending lots, so it is the unique multiple of $Q$ in $\{r+1,\ldots,r+Q\}$, namely $Q\lceil(r+1)/Q\rceil$. To calculate $H_{r,Q}$ for the original system with $k\ge0$, we temporarily set $a=k=0$ while keeping $\mu,r,Q$ fixed. The policy then makes the same replenishment decisions on each demand path, so its mean inventory $H_{r,Q}$ and stockout fraction $\ell_{r,Q}$ remain unchanged. Only holding cost remains, so average cost equals $H_{r,Q}$. With $S=r+1$ and $k=0$ for this calculation, the definition of $\gamma$ gives $\gamma=r+(Q+1)/2-\mu$. Substituting this and the stockout value of $Y$ into the exact cost identity~\eqref{eq:batch-exact-cost-identity} in Appendix~\ref{app:batch-upper}, dividing by $T$, and letting $T\to\infty$ gives the following formula, since the bounded term $\Psi(0)-\E\Psi(T)$ divided by $T$ vanishes:
\[
H_{r,Q}=\left(r+\frac{Q+1}{2}-\mu\right)(1-\ell_{r,Q})+Q\left\lceil\frac{r+1}{Q}\right\rceil\ell_{r,Q}.
\]
The resulting $H_{r,Q}$ applies to the original system and enters the full-cost formula above with the actual values of $a$ and $k$. Thus, only the stockout fraction $\ell_{r,Q}$ requires numerical estimation. When $Q=1$, it is exactly $e_{r+1}$. When $r<Q$, at most one order is pending (see footnote~\ref{fn:single-outstanding-order}). An order cycle is the interval between consecutive order placements. Immediately after an order, inventory position is $r+Q$, and the next order is triggered when it reaches $r$. Each sale reduces inventory position by one, while deliveries and lost demands leave it unchanged, so exactly $Q$ demands are served during the cycle. The cycle starts with $r$ units on hand and one lot due after $\mu$ time units; losses occur only before that lot arrives, giving expected lost demand $\E(N(\mu)-r)^+$. Hence,
\[
\ell_{r,Q}=\frac{\E(N(\mu)-r)^+}{Q+\E(N(\mu)-r)^+}\qquad(r<Q).
\]
At $\mu=0$, the stockout fraction is zero. These cases give exact policy costs and checks on the simulation used for the remaining pairs.

\begin{figure}[!t]
\SingleSpaced
\centering
\includegraphics[width=.99\linewidth]{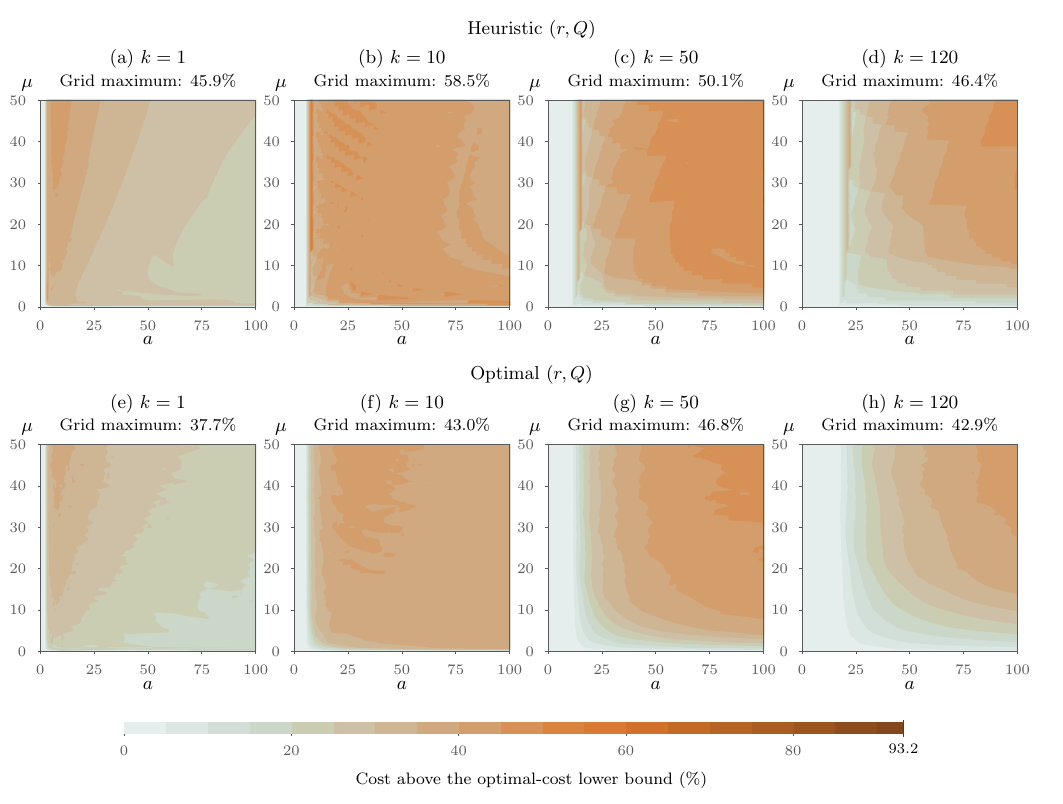}
\caption{Costs of heuristic and numerically optimized $(r,Q)$ policies relative to the optimal-cost lower bound}
\label{fig:positive-setup-gap}
\begin{minipage}{0.98\linewidth}
\SingleSpaced\footnotesize\emph{Notes.} Shading shows $100[C/\underline C_k(\mu,a)-1]$ for full policy cost $C$: heuristic above, numerical joint optimization below, including no replenishment. These estimate upper bounds on true optimality gaps. Contours interpolate grid values on a common scale; maxima are over the grid. The $93.2\%$ marker corresponds to the uniform factor $1.932$ relative to the actual optimum. At $a=0$, zero cost is displayed as zero gap. In our simulation results, the $95\%$ confidence interval for each gap extends less than one percentage point above and below the estimate.
\end{minipage}
\end{figure}

\textbf{Finite search bounds.} No replenishment has cost $a$. Because mean pending inventory is $\mu(1-\ell_{r,Q})$ and inventory position is at least $r+1$, mean on-hand inventory is at least $r+1-\mu$. Any policy improving on no replenishment must therefore have $r+1-\mu<a$.

To bound $Q$, partition positive on-hand inventory levels into $Q$ groups: group $j$ contains the levels $j,j+Q,j+2Q,\ldots$, for $j=1,\ldots,Q$. These levels have the same remainder upon division by $Q$. Each delivery adds $Q$ units and preserves this remainder, while each sale reduces inventory by one. The groups observed immediately before successive sales therefore cycle through all $Q$ groups. Every $Q$ consecutive sales includes exactly one sale from each group, so the cumulative sales counts from any two groups differ by at most one. Each group thus accounts for $1/Q$ of long-run sales. Applying Lemma~\ref{lem:poisson-compensation} to the bounded predictable indicator of each group shows that the expected fraction of time spent in that group is $(1-\ell_{r,Q})/Q$ in the long run. Since inventory is at least $j$ whenever it belongs to group $j$, summing these holding-cost contributions gives
\[
H_{r,Q}\ge\sum_{j=1}^Q j\,\frac{1-\ell_{r,Q}}Q
=\frac{Q(Q+1)}2\,\frac{1-\ell_{r,Q}}Q
=\frac{Q+1}{2}(1-\ell_{r,Q}).
\]
Adding lost-sales and setup costs to this holding-cost bound gives
\[
\begin{aligned}
C_{\mu,a,k}(\pi_{\rm rQ}(r,Q))
&=H_{r,Q}+a\ell_{r,Q}+\frac{k}{Q}(1-\ell_{r,Q})\\
&\ge\left(\frac{Q+1}{2}+\frac{k}{Q}\right)(1-\ell_{r,Q})+a\ell_{r,Q}\ge\min\left\{a,\frac{Q+1}{2}+\frac{k}{Q}\right\}.
\end{aligned}
\]
The last inequality holds because $0\le\ell_{r,Q}\le1$: the preceding expression is a weighted average of $a$ and $(Q+1)/2+k/Q$, so it is at least the smaller of these two values.

An improving lot must satisfy $(Q+1)/2+k/Q<a$, which in particular gives $Q<2a-1$. These bounds make the search finite. We evaluate the exact cases first and use their costs and the proved heuristic cost bound to exclude additional pairs that cannot improve the current upper bound. The further exclusion test combines the bound $\ell_{r,Q}\ge e_{r+Q}$ with the formulas for $H_{r,Q}$ and full cost above. Substituting the expression for $H_{r,Q}$ into the full-cost formula expresses cost as a linear function of $\ell_{r,Q}$. Its minimum over $e_{r+Q}\le\ell_{r,Q}\le1$ is a cost lower bound, obtained by evaluating that function at the two endpoints. We exclude $(r,Q)$ if this lower bound is at least the current best cost upper bound.

To justify $\ell_{r,Q}\ge e_{r+Q}$, couple the batched policy with a base-stock system of level $r+Q$, both starting with $r+Q$ units on hand and using the same demand path. Apply the induction over demand arrivals in Appendix~\ref{app:batch-upper}, now using $Q\lfloor n/Q\rfloor\le n$ in the supply accounting. This reverses the cumulative-sales comparison: the batched policy serves no more cumulative demand than the base-stock system over every time horizon. Its long-run served fraction is therefore at most $1-e_{r+Q}$, so its loss fraction satisfies $\ell_{r,Q}\ge e_{r+Q}$.

\textbf{Simulation and selection.} The simulation processes demand arrivals and scheduled deliveries in chronological order, preserving the exact inventory position and pending-order list. For each specified value of $\mu$, we search over $(r,Q)$. Across these values of $\mu$, the analytical exclusions leave a total of $31{,}394$ candidate pairs requiring simulation. Each candidate pair is evaluated at its fixed $\mu$ with four independent runs of $250{,}000$ measured time units. We minimize the resulting full-cost estimates, including the exact cases and no replenishment. We refer to the selected pair as $(r^\star,Q^\star)$; selection among close candidates remains subject to simulation error.

Reported costs use independent runs after policy selection to avoid the downward bias from reporting the smallest search estimate. We initially use eight runs of $500{,}000$ measured time units after at least $10{,}000$ warm-up units. Where greater precision is needed, we use $30$ independent runs of at least $2\times10^6$ measured time units. We evaluate the heuristic by the same formulas, including its comparison with no replenishment. Exact cases validate the simulation, and longer warm-up and measurement periods check sensitivity to initialization. The confidence intervals measure pointwise Monte Carlo uncertainty and do not establish exact parameter optimality. At $k=10$, four slightly negative estimates of the heuristic's extra cost in Figure~\ref{fig:heuristic-extra-cost-distribution}(b) reflect simulation error in the independently estimated policy costs.

\FloatBarrier


\begin{thebibliography}{}

\bibitem[{Bijvank and Vis(2011)}]{BijvankVis2011} Bijvank, M., I. F. A. Vis. 2011. Lost-sales inventory theory: A review. \emph{European J. Oper. Res.} 215(1) 1--13.

\bibitem[{Federgruen and Zheng(1992)}]{FedergruenZheng1992} Federgruen, A., Y.-S. Zheng. 1992. An efficient algorithm for computing an optimal $(r,Q)$ policy in continuous review stochastic inventory systems. \emph{Oper. Res.} 40(4) 808--813.

\bibitem[{Goldberg et~al.(2016)}]{GoldbergEtAl2016} Goldberg, D., D. Katz-Rogozhnikov, Y. Lu, M. Sharma, M. Squillante. 2016. Asymptotic optimality of constant-order policies for lost sales inventory models with large lead times. \emph{Math. Oper. Res.} 41(3) 898--913.

\bibitem[{Hadley and Whitin(1963)}]{HadleyWhitin1963} Hadley, G., T. M. Whitin. 1963. \emph{Analysis of Inventory Systems}. Prentice-Hall, Englewood Cliffs, NJ.

\bibitem[{Hill and Johansen(2006)}]{HillJohansen2006}
Hill, R. M., S. G. Johansen. 2006. Optimal and near-optimal policies for lost sales inventory models with at most one replenishment order outstanding. \emph{European J. Oper. Res.} 169(1) 111--132.

\bibitem[{Hill(2007)}]{Hill2007}
Hill, R. M. 2007. Continuous-review, lost-sales inventory models with Poisson demand, a fixed lead time and no fixed order cost. \emph{European J. Oper. Res.} 176(2) 956--963.

\bibitem[{Huh et~al.(2009)}]{HuhJMR2009} Huh, W. T., G. Janakiraman, J. A. Muckstadt, P. Rusmevichientong. 2009. Asymptotic optimality of order-up-to policies in lost sales inventory systems. \emph{Management Sci.} 55(3) 404--420.

\bibitem[{Johansen and Thorstenson(1996)}]{JohansenThorstenson1996}
Johansen, S. G., A. Thorstenson. 1996. Optimal $(r,Q)$ inventory policies with Poisson demands and lost sales: Discounted and undiscounted cases. \emph{Int. J. Production Econom.} 46--47 359--371.

\bibitem[{Kingman(1993)}]{Kingman1993} Kingman, J. F. C. 1993. \emph{Poisson Processes}. Oxford University Press, Oxford, UK.

\bibitem[{Levi et~al.(2007)}]{LeviEtAl2007} Levi, R., M. P\'al, R. O. Roundy, D. B. Shmoys. 2007. Approximation algorithms for stochastic inventory control models. \emph{Math. Oper. Res.} 32(2) 284--302.

\bibitem[{Levi et~al.(2008)}]{LeviJN2008}
Levi, R., G. Janakiraman, M. Nagarajan. 2008. A 2-approximation algorithm for stochastic inventory control models with lost sales. \emph{Math. Oper. Res.} 33(2) 351--374.

\bibitem[{Levi and Shi(2013)}]{LeviShi2013}
Levi, R., C. Shi. 2013. Approximation algorithms for the stochastic lot-sizing problem with order lead times. \emph{Oper. Res.} 61(3) 593--602.

\bibitem[{Reiman(2004)}]{Reiman2004} Reiman, M. I. 2004. A new and simple policy for the continuous review lost sales inventory model. Working paper, Bell Labs, Lucent Technologies, Murray Hill, NJ.

\bibitem[{Shang et~al.(2009)}]{ShangSongZipkin2009} Shang, K. H., J.-S. Song, P. H. Zipkin. 2009. Coordination mechanisms in decentralized serial inventory systems with batch ordering. \emph{Management Sci.} 55(4) 685--695.

\bibitem[{Shang and Zhou(2010)}]{ShangZhou2010} Shang, K. H., S. X. Zhou. 2010. Optimal and heuristic echelon $(r,nQ,T)$ policies in serial inventory systems with fixed costs. \emph{Oper. Res.} 58(2) 414--427.

\bibitem[{Song(1994)}]{Song1994} Song, J.-S. 1994. The effect of leadtime uncertainty in a simple stochastic inventory model. \emph{Management Sci.} 40(5) 603--613.

\bibitem[{Wolff(1982)}]{Wolff1982}
Wolff, R. W. 1982. Poisson arrivals see time averages. \emph{Oper. Res.} 30(2) 223--231.

\bibitem[{Xin(2021)}]{Xin2021} Xin, L. 2021. Technical note: Understanding the performance of capped base-stock policies in lost-sales inventory models. \emph{Oper. Res.} 69(1) 61--70.

\bibitem[{Xin(2022)}]{Xin2022}
Xin, L. 2022. 1.79-approximation algorithms for continuous review single-sourcing lost-sales and dual-sourcing inventory models. \emph{Oper. Res.} 70(1) 111--128.

\bibitem[{Xin(2026)}]{Xin2026} Xin, L. 2026. Capped base-stock policies: A 2.33-approximation. Working paper, Cornell University. \url{https://ssrn.com/abstract=7134538}.

\bibitem[{Xin and Goldberg(2016)}]{XinGoldberg2016} Xin, L., D. A. Goldberg. 2016. Optimality gap of constant-order policies decays exponentially in the lead time for lost sales models. \emph{Oper. Res.} 64(6) 1556--1565.

\bibitem[{Zachary(2007)}]{Zachary2007}
Zachary, S. 2007. A note on insensitivity in stochastic networks. \emph{J. Appl. Probab.} 44(1) 238--248.

\bibitem[{Zipkin(2008)}]{Zipkin2008}
Zipkin, P. 2008. Old and new methods for lost-sales inventory systems. \emph{Oper. Res.} 56(5) 1256--1263.

\end{thebibliography}

\begin{thebibliography}{}

\bibitem[{Abramowitz and Stegun(1964)}]{AbramowitzStegun1964}
Abramowitz, M., I. A. Stegun, eds. 1964. \emph{Handbook of Mathematical Functions with Formulas, Graphs, and Mathematical Tables}. National Bureau of Standards, Washington, DC.

\bibitem[{Alfers and Dinges(1984)}]{AlfersDinges1984}
Alfers, D., H. Dinges. 1984. A normal approximation for beta and gamma tail probabilities. \emph{Z. Wahrsch. Verw. Gebiete} 65(3) 399--420.

\bibitem[{Dia and Lamberton(2011)}]{DiaLamberton2011}
Dia, E. H. A., D. Lamberton. 2011. Connecting discrete and continuous lookback or hindsight options in exponential L\'evy models. \emph{Adv. Appl. Probab.} 43(4) 1136--1165.

\bibitem[{Harrison(1985)}]{Harrison1985}
Harrison, J. M. 1985. \emph{Brownian Motion and Stochastic Flow Systems}. Wiley, New York.

\bibitem[{Kyprianou(2014)}]{Kyprianou2014}
Kyprianou, A. E. 2014. \emph{Fluctuations of L\'evy Processes with Applications: Introductory Lectures}. 2nd ed. Springer, Heidelberg.

\bibitem[{Robbins(1955)}]{Robbins1955}
Robbins, H. 1955. A remark on Stirling's formula. \emph{Amer. Math. Monthly} 62(1) 26--29.

\bibitem[{Zubkov and Serov(2013)}]{ZubkovSerov2013}
Zubkov, A. M., A. A. Serov. 2013. A complete proof of universal inequalities for the distribution function of the binomial law. \emph{Theory Probab. Appl.} 57(3) 539--544.

\end{thebibliography}
\end{document}